\documentclass[12pt]{amsart}
\usepackage{amsmath,amssymb,amsthm,amsfonts,amstext,amsbsy,amscd}
\usepackage{graphics}
\usepackage{graphicx}
\usepackage[lofdepth,lotdepth]{subfig}

\usepackage{float}
\usepackage{multirow}
\usepackage{multicol}
\usepackage{latexsym}
\usepackage{enumitem}
\usepackage[utf8]{inputenc}
\usepackage{dsfont}
\usepackage{color}
\RequirePackage[colorlinks,citecolor=blue,urlcolor=blue]{hyperref}
\RequirePackage{hypernat}
\usepackage[top=3cm, bottom=3cm, left=2cm, right=2.5cm]{geometry}
\usepackage{booktabs}

\usepackage{tabularx}

\renewcommand{\phi}{\varphi}

\newcommand{\PP}{\mathbb{P}}

\newcommand{\E}{\mathbb{E}}

\newcommand{\rmd}{{\rm d}}

\newtheorem{theorem}{Theorem}[section]

\newtheorem{remark}[theorem]{Remark}
\newtheorem{lemma}[theorem]{Lemma}
\newtheorem{definition}[theorem]{Definition}

\newtheorem{example}[theorem]{Example}
\newtheorem{proposition}[theorem]{Proposition}
\newtheorem{assumption}[theorem]{Assumption}

\newtheorem{question}[theorem]{Question}

\renewcommand{\tilde}{\widetilde}

\definecolor{forestgreen}{rgb}{0.13, 0.55, 0.13}
\definecolor{orange-red}{rgb}{1.0, 0.27, 0.0}
\definecolor{purp}{rgb}{0.67, 0.17, 0.30}
\definecolor{amethyst}{rgb}{0.6, 0.4, 0.8}
\definecolor{cobalt}{rgb}{0.0, 0.28, 0.67}

\definecolor{red}{rgb}{1.0, 0.0, 0.0}

\begin{document}
\title[]{Local stability and rates of convergence to equilibrium for the Nonlinear Renewal Equation; applications to Hawkes processes}

\author{C\'eline Duval}
\address{Sorbonne Universit\'e and Universit\'e Paris Cit\'e, CNRS, Laboratoire de Probabilit\'es, Statistique et Mod\'elisation, F-75005 Paris, France \url{celine.duval@sorbonne-universite.fr}}

\author{Eric Lu\c{c}on}
\address{Universit\'e d’Orl\'eans, Universit\'e de Tours, CNRS, IDP, UMR 7013, Orl\'eans, France, \url{eric.lucon@univ-orleans.fr}.
}

\maketitle

\begin{abstract}   We study the asymptotic properties of the solutions of a  nonlinear renewal equation. The main contribution of the present article  is to provide  stability and convergence results around equilibrium solutions, under some local subcritical condition. Quantitative rates of convergence to equilibrium  are established. Instability results are given in both critical and supercritical cases. 
As an implication of these results, we establish a Central Limit Theorem for Hawkes processes in a mean-field interaction.  \end{abstract}

\section{Introduction and Motivations}

\subsection{ Nonlinear Renewal Equation and application to mean-field Hawkes processes}
\label{sec:NRE_intro}
We are concerned in this paper with the asymptotic properties of the \emph{Nonlinear Renewal Equation} (that we will sometimes denote as NRE in the rest of the paper)
\begin{equation}
\label{eq:conv_gen_lambda}
\lambda_{ t} = \Phi \left( \xi_{ t} + \int_{ 0}^{t} h(t-s) \lambda_{ s} {\rm d}s\right), \ t\geq 0,
\end{equation}
whose unknown is the nonnegative function $(\lambda_t)_{t\geq 0}$ and where the locally integrable function $h$, the nonnegative $\Phi$ and the locally integrable function $\xi$, vanishing at infinity, are given. Renewal Equations such as \eqref{eq:conv_gen_lambda} have met a considerable interest in the literature with various applications to \textit{e.g.,} renewal theory \cite{feller1941}, branching processes \cite{Athreya1972,Chover1968,Chover1973,Ney1977}, demography and spread of diseases \cite{Brauer1975}, actuarial risk theory, see \cite{Dermitzakis2022} and references therein. \medskip

The main motivation of the present work concerns the links between \eqref{eq:conv_gen_lambda} and interacting Hawkes processes. 
Since their introduction in \cite{MR0378093} (originally motivated by the modeling of earthquakes activity), Hawkes processes have been successfully applied in various contexts, such as mathematical finance (see \textit{e.g.,} \cite{MR3054533, MR3313750}), ecology \cite{bonnet2025}, genomics \cite{MR2722456} and neuroscience \cite{MR3449317,ReynaudBouret2013InferenceOF}, the latter application being the main motivation of the present work.

The use of Hawkes processes is particularly well suited to the modeling of biological neurons, as they adapt to the history of the whole network. Biological neurons form a large interactive network in which each neuron receives and transmits information to other neurons in the form of electrical signals. These signals are characterized by their similar amplitudes and very short durations, which can be summarized as single spiking instants.  Thus, a common  way to model these exchanges of information is to associate to each neuron $i\in\{1,\ldots,N\}$ within a population of size $N\geq 1$ the counting process $Z^{i,N}$ of its spikes. The jump dynamics of $Z^{i,N}$ is defined through its intensity function denoted by $\lambda_{t}^{i,N}$. This intensity allows to characterize  the probability that, at each instant $t$, the process will jump: informally
 $\PP(Z^{i,N} \mbox{ has a jump in }[t,t+\rmd t]|\mathcal F_t)\approx \lambda^{i,N}_t\rmd t$, where $\mathcal F_t=\sigma\left(Z^{j,N}_s,\ s\le t,\ j\in\{1,\ldots,N\}\right)$ denotes the history of all the processes in the network. \ We follow here the framework of \cite{MR3449317}, defining an homogeneous mean-field intensity for all particles:
 \begin{equation}
\label{eq:Hawkes_N}
\lambda_t^{i,N}=\lambda_{N, t} := \Phi \left(\xi_{N, t} + \frac{1}{N}\sum_{j=1}^N \int_0^{t-}h(t-s) {\rm d}Z_s^{j, N}\right),\ i=1\ldots, N.
\end{equation}
 All the quantities involved in \eqref{eq:conv_gen_lambda} and \eqref{eq:Hawkes_N} have a biological interpretation:    $\Phi: \mathbb{R}\to [0, +\infty)$  is the \emph{synaptic integration kernel} or \emph{firing rate function}, $h$ is the \emph{memory kernel} that modulates how the past jumps of \eqref{eq:Hawkes_N} affects the present intensity and $\xi_{N,t}$ is a \emph{source term} modeling the possible influence of the history of the process on $(-\infty, 0)$ onto the present. We assume that $\xi_{N,t}\xrightarrow[N\to\infty]{}\xi_t$ for all $t\ge0$.  A usual choice in the literature corresponds to the empty source term $\xi_{N, t}\equiv 0$ which leads to 
 \begin{equation}
 \label{eq:empty_xi}
     \xi^{\emptyset}_t:=0, \ t\geq 0.
 \end{equation}
 Whereas some attention will be granted to the particular case of \eqref{eq:empty_xi} in the following, the point of our work will be precisely to consider source terms $\xi$ as general as possible.
 One of the recent challenges is to take into account the \emph{inhibitory or excitatory} nature of neurons. From a modeling perspective, if the function $\Phi$ is increasing, excitation can be captured by positive values of the function $h$, while inhibition is modeled by negative values of $h$. Considering inhibition generically poses complex mathematical issues, in particular in the identification of the longtime dynamics of such processes (see \textit{e.g.,} \cite{costa_graham_marsalle_tran_2020,duval2021interacting}). In the rest of the article, we will sometimes refer to the \emph{excitatory case} when $h$ is assumed to be always non negative.

 The link between the particle system \eqref{eq:Hawkes_N} and the Nonlinear Renewal Equation \eqref{eq:conv_gen_lambda} comes from a standard large population limit. Since the biological network is very dense, the mean field approximation $N\to\infty$ is valid \cite{MR3449317}:  under natural assumptions on the coefficients, the system \eqref{eq:Hawkes_N} is well-described in the large population limit $N\to \infty$ by some inhomogeneous Poisson process $\bar Z_t$ whose intensity $\lambda_t$ solves \eqref{eq:conv_gen_lambda} (see Appendix \ref{sec:prf_TCL} and \eqref{eq:coupling} for a precise statement).

\subsection{Sub- and supercriticality for the NRE}
\label{sec:criticality_intro}
It is easy to see (we refer to Proposition~\ref{prop:conv_X_lambda} below for a more precise statement) that if $\Phi$ is continuous, $h$ is integrable and $\xi_t\xrightarrow[t\to\infty]{}0$, any possible limit $\ell$ of $\lambda_t$ as $t\to\infty$ solves necessarily the following fixed-point equation
\begin{equation}
\label{eq:fixed_point_intro}
    \ell= \Phi \left(\ell\int_0^{+\infty} h(u) {\rm d}u\right).
\end{equation}
The identification of the limit and the stability of such fixed-point for the dynamics induced by \eqref{eq:conv_gen_lambda} will be one of the main purpose of the paper. Whenever $h$ is integrable on $[0, +\infty)$, we will use the notation $\Vert h\Vert_1= \int_0^{+\infty} \vert h(u)\vert {\rm d}u$.
\subsubsection{The linear case}
\label{sec:linear_case}
The historical version of \eqref{eq:conv_gen_lambda} corresponds to the case where $\Phi(x)=\mu + x$ is linear, $\mu> 0$, $ h\geq 0$ and $\xi=\xi^{\emptyset}$ given by \eqref{eq:empty_xi},  \cite{feller1941,MR3449317}:
\begin{equation}
\label{eq:LRE}
\lambda_{ t} = \mu + \int_{ 0}^{t} h(t-s) \lambda_{ s} {\rm d}s, \ t\geq 0,
\end{equation}
In this case, the behavior of $\lambda_t$ as $t\to\infty$ is well-understood:
\begin{theorem}[\cite{feller1941} and \cite{MR3449317}, Th.~10 \& Th.~11]
\label{th:linear_case}
Suppose $\Phi(x) = \mu +x$ and that $h\geq 0$ is integrable on $[0, +\infty)$. Then we have the following dichotomy concerning the solution $\lambda$ to \eqref{eq:LRE}:
\begin{enumerate}[label=(\alph*)]
    \item Subcritical case: suppose that $ \Vert h \Vert_1<1$, then $\lambda_t \xrightarrow[t\to\infty]{} \frac{\mu}{1- \Vert h \Vert_1}$.
    \item Supercritical case: suppose that $\Vert h \Vert_1\geq 1$, then $\lambda_t \xrightarrow[t\to\infty]{} +\infty$. 
\end{enumerate}
\end{theorem}
Hence, the dichotomy in Theorem~\ref{th:linear_case} is between a bounded and convergent solution in the first case and an unbounded divergent solution in the second. Explicit rates of divergence can be proven in the supercritical case, see \cite{MR3449317}, Th.~11 and \cite{Horst2024}. The linear Renewal Equation has been the subject of an extensive interest in the literature that cannot be covered exhaustively here (see \textit{e.g.,} \cite{Athreya1976,Dermitzakis2022,Chover1973} and references therein). In opposition to the fully nonlinear case, what makes the analysis considerably simpler in the linear case \eqref{eq:LRE} is that its solution can be expressed explicitly as the convolution w.r.t. the resolvent kernel associated to \eqref{eq:LRE} (see \cite{feller1941,Dermitzakis2022,Horst2024}).

\medskip

Going back to the general equation \eqref{eq:conv_gen_lambda}, the main question and motivation for the present work is then:
\begin{question}
\label{qu:criticality}
When $ \Phi$ is nonlinear, what is the correct notion of criticality for Equation~\eqref{eq:conv_gen_lambda}? does it refer to the convergence of $\lambda_t$ as $t\to\infty$? or more simply to the boundedness of $\lambda$ on $[0, +\infty)$? can we identify a phase transition similar to the linear case of Theorem~\ref{th:linear_case}? if so, how does it involve $ h$ and the nonlinear kernel $\Phi$? 
\end{question}
\subsubsection{Strong subcriticality in the nonlinear case}
First begin with a very demanding statement: an assumption that is usually adopted in the Hawkes literature \cite{agathenerine22,MR1411506,Duarte:2016aa,DITLEVSEN20171840,MR3102513} is the following strong subcritical condition
\begin{equation}
\label{eq:subcritical_Lip}
    \vert \Phi \vert_{Lip} \Vert h \Vert_1<1.
\end{equation}
In this case, the behavior of $\lambda$ is described by the following straightforward result 

\begin{proposition}
\label{prop:subcritical_Lip}
Suppose that $ \Phi$ is Lipschitz continuous and $h$ integrable on $[0, +\infty)$. Suppose that \eqref{eq:subcritical_Lip} is true.
Then, there is a unique fixed-point $\ell$ solution to \eqref{eq:fixed_point_intro} and for every bounded source term $ \xi$ such that $ \xi_{ t} \xrightarrow[ t\to\infty]{}0$, the solution $ \lambda= \lambda^{ \xi}$ to \eqref{eq:conv_gen_lambda} remains bounded and converges to $\ell$ as $t\to\infty$.
\end{proposition}
There is no reason why Proposition~\ref{prop:subcritical_Lip} should be really new, but we have not found any complete statement in the literature. We have reproduced its proof in Section~\ref{sec:strong_subcrit} for completeness. However Condition~\eqref{eq:subcritical_Lip} is obviously too demanding for many interesting examples, as it prevents \textit{e.g.,} to consider situations where there are multiple solutions to \eqref{eq:fixed_point_intro}. A prominent situation is when $\Phi$ is sigmoid with three fixed-points solution to \eqref{eq:fixed_point_intro} \cite{AgatheNerine2025,Heesen2021,lucon:hal-05185413,Sulem2024}: a detailed example, where \eqref{eq:conv_gen_lambda} is driven by a kernel $h$ of Erlang type (see Section~\ref{sec:Erlang}) is given in Figure~\ref{fig:sigmoid}, Section~\ref{sec:sigmoid}.

\subsection{Boundedness and convergence of solutions to the NRE}
\label{sec:boundedness_intro}
\subsubsection{Global estimates for boundedness}
Let us address what is known in the literature concerning the boundedness of solutions to \eqref{eq:conv_gen_lambda}. The strong statement \eqref{eq:subcritical_Lip} may be replaced by the following weaker one, that is due to \cite{Brauer1975}:
\begin{equation}
\label{eq:global_subcrit}
\limsup_{ \left\vert x \right\vert\to\infty} \frac{ \Phi(x)}{ \left\vert x \right\vert} \left\Vert h \right\Vert_{ 1} <1.
\end{equation}
\begin{proposition}[\cite{Brauer1975}]
\label{prop:global_subcrit}
Suppose that $ \Phi$ is locally Lipschitz, that $h$ is integrable on $[0, +\infty)$ and $ \xi$ bounded on $[0, +\infty)$. Suppose that \eqref{eq:global_subcrit} is true. Then $ \lambda= \lambda^{ \xi}$ solution to \eqref{eq:conv_gen_lambda} with source term $\xi$  is bounded on $[0,+\infty)$.
\end{proposition}
The proof of Proposition~\ref{prop:global_subcrit} is due to Brauer \cite{Brauer1975}, it is stated in \cite{Brauer1975} under peculiar assumptions (notably that $ \Phi(0)=0$ and $h$ nonincreasing). We reproduce in Section~\ref{sec:global_subcrit} its proof with minor adaptations to convince the reader that everything works under our more general assumptions. We refer to \cite{Gripenberg1979} for further results related to boundedness of solutions to \eqref{eq:conv_gen_lambda}. To the best of our knowledge, nothing has been proven concerning the optimality of Condition~\eqref{eq:global_subcrit} w.r.t. the boundedness of solutions to \eqref{eq:conv_gen_lambda}. One contribution of the paper will be to prove that \eqref{eq:global_subcrit} is indeed optimal, in the sense that when $ \limsup_{ \left\vert x \right\vert\to\infty} \frac{ \Phi(x)}{ \left\vert x \right\vert} \left\Vert h \right\Vert_{ 1} \geq 1$, there are generically  unbounded solutions to \eqref{eq:conv_gen_lambda}. Interestingly, Condition~\eqref{eq:global_subcrit} appears also in \cite{Robert2024} as a sufficient condition for the existence of stationary univariate Hawkes processes. We refer to Section~\ref{sec:optimal_boundedness} for precise statements.

\subsubsection{Londen's Theorem}
Considering \eqref{eq:conv_gen_lambda} from a neuroscience perspective, if $\lambda_t$ stands for the intensity of a generic neuron within an infinite population, what is commonly considered as the membrane potential of the neuron \cite{AgatheNerine2022,CHEVALLIER20191,lucon:hal-05185413} is the quantity $x_t$ defined as 
\begin{equation*}
   x_{ t}:= \xi_{ t} + \int_{ 0}^{t} h(t-s) \lambda_s {\rm d}s,\ t\geq 0,
\end{equation*}
so that 
\begin{equation*}
\lambda_{ t}= \Phi \left(x_{ t}\right),\ t\geq 0.
\end{equation*}
In this setting, the NRE~\eqref{eq:conv_gen_lambda} for the intensity $\lambda$ is equivalent to the NRE for the potential $x_t$ that is written as
\begin{equation}
\label{eq:conv_gen_X}
x_{ t}= \xi_{ t} + \int_{ 0}^{t} h(t-s) \Phi(x_{ s}) {\rm d}s,\ t\geq 0.
\end{equation}
An important result is due to Londen \cite{LONDEN1973106,doi:10.1137/0505082}, see also \cite[Th.1 and Corollary]{BRAUER197632}, that is written in terms of \eqref{eq:conv_gen_X}:

\begin{theorem}[\cite{doi:10.1137/0505082}] 
\label{th:londen}
Suppose that $h$ is nonnegative, nonincreasing and integrable on $[0, +\infty)$ with $h(0)<+\infty$. Suppose that $ \Phi$ is continuously differentiable on $ \mathbb{ R}$ and that there is no interval on which $ \Phi^{ \prime}(x) \left\Vert h \right\Vert_{ 1} \equiv 1$. Suppose finally that $\xi$ is bounded on $[0, +\infty)$ with $\lim_{t\to\infty}\xi_{t}=0$. Let $x:= \left(x_{ t}^{ \xi}\right)_{ t\geq0}$ be the solution to \eqref{eq:conv_gen_X} with source term $ \xi$. Then, if $(x_t)$ is bounded, then $ x_t \xrightarrow[t\to\infty]{}\ell_x$ where $\ell_x$ solves the companion fixed-point relation
\begin{equation}
\label{eq:fixed_point_X_intro}
\ell_{ x}= \left(\int_0^{+\infty}h(u) {\rm d}u\right) \Phi(\ell_{ x}).
\end{equation}
\end{theorem}
The condition that there is no interval on which $ \Phi^{ \prime}(x) \left\Vert h \right\Vert_{ 1} \equiv 1$ implies that the solutions to \eqref{eq:fixed_point_X_intro} are isolated. From the boundedness of $ \xi$, the integrability of $h$ and the continuity of $ \Phi$, we deduce immediately that the boundedness of $x$ is equivalent to the boundedness of $\lambda$: a similar result where $x$ is replaced by $\lambda$ is also true. We refer to \cite{doi:10.1137/0505082} and \cite[Th.1 and Corollary]{BRAUER197632} for a proof of Theorem~\ref{th:londen}.

\subsubsection{Comments on Londen's Theorem}
A series of papers have addressed sufficient conditions for convergence of $\lambda_t$, most often in cases that do not match our hypotheses (assuming either that $\Phi$ is nonincreasing, hence implying the uniqueness of a solution to \eqref{eq:fixed_point_intro} or that $\Phi(0)=0$) see \textit{e.g.,} \cite{Levin1965,Levin1972,Shilepsky1974}. As deep and important as it may be, Theorem~\ref{th:londen} leaves room for further comments and perspectives.

Firstly, the hypotheses of Theorem~\ref{th:londen} are quite restrictive, especially concerning the kernel $h$: the assumption that $h$ being nonincreasing is not satisfied by many interesting examples, among which Erlang kernels (see Section~\ref{sec:Erlang}) that have been extensively considered in the Hawkes literature, see \textit{e.g.,} \cite{Duarte:2016aa} and references therein.

Secondly, Theorem~\ref{th:londen} requires the a priori knowledge that the solution $\lambda$ is bounded. As we have seen, a sufficient condition for this is \eqref{eq:global_subcrit}. However, the possibility remains that \eqref{eq:conv_gen_lambda} admits bounded solutions while \eqref{eq:global_subcrit} does not hold. Since $\lambda^{\xi}$ is uniquely determined by its source term $\xi$, a more natural criterion for boundedness than \eqref{eq:global_subcrit} (that exclusively relies on $ \Phi$) should rely instead on $\xi$ itself. Hence, the question is
\begin{question}
\label{qu:bounded}
For any given $\Phi$, can we characterize the source terms $ \xi$ for which the corresponding solution $\lambda^\xi$ to \eqref{eq:conv_gen_lambda} remains bounded?
\end{question}

Thirdly, once we know that some solution $\lambda$ to \eqref{eq:conv_gen_lambda} is bounded, Theorem~\ref{th:londen} does not say anything on the identification of its limit as $t\to\infty$ in cases where \eqref{eq:fixed_point_intro} admits several solutions. Further questions are then
\begin{question}
\label{qu:convergence}
\begin{enumerate}[label=(\alph*)]
    \item For a given solution $\ell$ to the fixed-point equation \eqref{eq:fixed_point_intro}, can we characterize the source terms $\xi$ such that the corresponding solution $\lambda^\xi$ to \eqref{eq:conv_gen_lambda} converges to $\ell$ as $t\to\infty$?
    \item If so, is it possible to obtain quantitative rates of convergence towards $\ell$?
    \item In case there is only one solution to \eqref{eq:fixed_point_intro}, does this mean that every solution $\lambda$ converge to $\ell$ as $t\to \infty$? or are there solutions that oscillate or diverge to $+\infty$?
\end{enumerate}
\end{question}

\subsection{Contributions of this paper}
\label{sec:contributions_intro}
The purpose of the paper is to address Questions~\ref{qu:criticality}, \ref{qu:bounded} and~\ref{qu:convergence}. 
\subsubsection{Local stability and rates and convergence to equilibrium}
The first contribution of the paper is to define a weaker local notion of subcriticality in replacement of \eqref{eq:subcritical_Lip} that concerns each fixed-point $\ell$ to \eqref{eq:fixed_point_intro}: namely, we say that $\ell$ solution to \eqref{eq:fixed_point_intro} is \emph{locally subcritical} if the following condition holds
\begin{equation}
\label{eq:cond_stab_ell}
\left\Vert h \right\Vert_{ 1}\left\vert \Phi^{ \prime} \left(\ell\int_{0}^{+\infty}h(u) {\rm d}u\right) \right\vert <1.
\end{equation} 
Note that in the linear case (\S~\ref{sec:linear_case}) Conditions~\eqref{eq:subcritical_Lip},~\eqref{eq:global_subcrit}, and~\eqref{eq:cond_stab_ell} coincide. 
The first main result of the paper (Theorem~\ref{th:conv_rate}) states a quantitative local version of Londen's Theorem: under the local subcritical condition~\eqref{eq:cond_stab_ell}, we prove that once we know that the solution $\lambda$ asymptotically reaches a fixed small neighborhood of $\ell$, $\lambda_t$ converges to $\ell$ as $t\to\infty$. Moreover, we give quantitative rates of convergence depending on the decay of both $\xi$ and $h$ as $t\to\infty$. We prove in the second main result (Theorem~\ref{th:stab_stationary}) that Theorem~\ref{th:conv_rate} is valid for a class of $\xi$ that are small perturbations of equilibrium source term $\xi^{eq, \ell}$ (see Proposition~\ref{prop:equilibrium} and \eqref{eq:xi_stationary} for a precise definition). We prove also in Theorem~\ref{th:conv_empty} that the result of Theorem~\ref{th:conv_rate} holds for perturbations of the empty source term $\xi^\emptyset\equiv 0$  given by \eqref{eq:empty_xi} at least in the excitatory case $h\geq 0$. Condition~\eqref{eq:cond_stab_ell} appears also in \cite{agathenerine22} (see in particular the comments in \S~2.3.5) as a sufficient condition for the long-time stability of mean-field Hawkes processes with spatial interaction. However, \cite{agathenerine22} restricts to the case of exponential kernels $h(u)=e^{-u}$ and does not address rates of convergence as we do.

\subsubsection{Optimality of the local condition in the excitatory case}
Theorems~\ref{th:conv_rate} and~\ref{th:stab_stationary} are valid without any assumption on the sign of $h$: both results are applicable in a context with inhibition. In a situation with pure excitation (\textit{i.e.,} $h\geq 0$) we complement the previous results by proving that Condition~\eqref{eq:cond_stab_ell} is indeed optimal w.r.t. the convergence of $\lambda_t$ towards $\ell$: we prove instability of the fixed-point in both the supercritical case (see \S~\ref{sec:supercritical}) and the critical case (see \S~\ref{sec:critical_case}).

\subsubsection{Optimality of the global condition for the boundedness of solutions}
We also prove the optimality of Condition~\eqref{eq:global_subcrit} as far as the boundedness of solutions to \eqref{eq:conv_gen_lambda} is concerned: when \eqref{eq:global_subcrit} does not hold, we prove the generic existence of unbounded solutions to \eqref{eq:conv_gen_lambda}, see Appendix~\ref{sec:appendix_boundedness}. 

\subsubsection{Applications to fluctuations of mean-field Hawkes processes}
 The question of fluctuation of Hawkes processes has been recently the subject of an extensive literature, mostly in the univariate case. We refer \textit{e.g.,} to \cite{MR3054533,cattiaux2021limit,Coutin25,Hillairet2022,Horst2024,Torrisi:2016aa,MR3102513} for further reference. A concrete use of  Theorem \ref{th:conv_rate} is to give fluctuation results for the mean-field Hawkes process $Z^{i,N}$ defined in Section~\ref{sec:NRE_intro}, whose conditional intensity is given by \eqref{eq:Hawkes_N}. When $(N,t)\to(\infty,\infty)$,  Central Limit Theorems have been established for $\xi=\xi^\emptyset$ given by \eqref{eq:empty_xi} and under the strong subcritical condition~\eqref{eq:subcritical_Lip} (see \cite[Corollary 1]{MR3054533} and \cite[Theorem 10]{MR3449317} for the linear case and \cite[Theorem 2]{DITLEVSEN20171840} for the nonlinear case). From a statistical perspective, these results are important because $Z^{i,N}_t/t$ concentrates  as $(N,t)\to \infty$  towards $\ell$, the unique (under Condition~\eqref{eq:subcritical_Lip}) solution of \eqref{eq:fixed_point_intro}, and $\ell$ can be estimated accordingly. 

However, in the nonlinear framework, the normalization of  the fluctuation result in \cite[Theorem 2]{DITLEVSEN20171840} depends on the unknown quantity $m_t:=\int_0^t\lambda_s\rmd s$ making it statistically unusable (see the discussion in \cite{duval2021interacting}, \S~4.2.1). To make  these results statistically meaningful, we need to study the rate of convergence of $m_t/t$ towards $\ell$ (see Equation \eqref{eq:Dec_ellT}), which is made possible by Theorem \ref{th:conv_rate}.   Our main result here is Theorem \ref{th:TCL} providing the fluctuations of $Z^{i,N}_t/t$ around $\ell$ in presence of a source term $\xi$ and for a nonlinear function $\Phi$. 







\subsection{Perspectives}
In the excitatory case $h\geq 0$,  we show in this work the optimality of the subcritical condition \eqref{eq:cond_stab_ell} w.r.t. the stability of \eqref{eq:conv_gen_lambda} around a constant intensity $\ell$ solving \eqref{eq:fixed_point_intro}. In the case with inhibition where $h$ is allowed to take negative values, a similar optimality of \eqref{eq:cond_stab_ell} is, to the very least, unclear. 

Another interesting point is to study the emergence of periodic behaviors for $\eqref{eq:conv_gen_lambda}$. Whereas Londen's Theorem implies that no such periodic behaviors are possible in case $h$ is nonnegative and nonincreasing, the case where $h$ may increase or take negative values remains largely open. In some particular case of Erlang kernels with cyclic connectivity and inhibition, periodic behaviors have been exhibited in \cite{DITLEVSEN20171840} (see also \cite{duval2021interacting}) but addressing more general kernels $h$ remains to be done.

Another perspective concerns the possible influence of the local stability condition \eqref{eq:cond_stab_ell} on the particle system \eqref{eq:Hawkes_N} (and not only on its mean-field limit \eqref{eq:conv_gen_lambda}).  The long-term stability of \eqref{eq:Hawkes_N} under \eqref{eq:cond_stab_ell} has been proven in \cite{agathenerine22} in the case of exponential kernels $h(u)=e^{-\alpha u}$, but the general case remains open. Finally, regarding the fluctuation results mentioned in Section~\ref{sec:application_Hawkes}, Theorem~\ref{th:TCL} is established under the strong subcritical condition \eqref{eq:subcritical_Lip} for technical reasons. Examining if this result remains valid  under the weaker condition \eqref{eq:cond_stab_ell} would be a valuable extension. We plan to address these questions in future works.
\subsection{Organisation of the paper}
The rest of the article is organized as follows. Section \ref{sec:main} contains the main results, in particular the stability results of Theorems~\ref{th:conv_rate} and~\ref{th:stab_stationary} in the subcritical case (Section~\ref{sec:subcritical}). Section~\ref{sec:excitatory} gathers all the results concerning the excitatory case. The application to the fluctuations of the Hawkes processes is detailed in Section~\ref{sec:application_Hawkes}. Section~\ref{sec:examples} contains a detailed analysis of examples and particular cases concerning $\Phi$ and $h$. Proofs of the main results are given in Sections~\ref{sec:proofs_WPregular},~\ref{sec:proofs} and~\ref{sec:proofs_excitatory}. Appendix~\ref{sec:appendix_boundedness} contains the results and proofs concerning the boundedness of solutions to \eqref{eq:conv_gen_lambda}. Appendix~\ref{sec:prf_TCL} contains the proof of the fluctuation result. Technical results are gathered in Appendix~\ref{sec:auxiliary}.

\section{Main results\label{sec:main}}

\subsection{Well-posedness and preliminary results}

We begin with the following well-posedness result for \eqref{eq:conv_gen_lambda}:
\begin{proposition}\label{prop:WP_solC1}
Suppose that $\Phi$ is Lipschitz continuous and nonnegative, that $\xi$ is locally bounded and $h$ is locally integrable on $[0, +\infty)$. There is a unique locally bounded $ (\lambda_{ t})_{ t\geq0}$ solution to \eqref{eq:conv_gen_lambda} that is nonnegative. In the case $ \xi$ is continuous, such a solution $ \lambda$ is continuous. Suppose moreover that $\Phi$ is of class $C^{1}$ with $ \Phi^\prime$ Lipschitz continuous, that $h$ is continuous on $[0, +\infty)$ and that $\xi$ is $C^{1}$ on $[0, +\infty)$. Then, $ (\lambda_{ t})_{ t\geq0}$ solution to \eqref{eq:conv_gen_lambda}  is of class $ \mathcal{ C}^{ 1}$ on $[0, +\infty)$.
\end{proposition}
Proof of Proposition \ref{prop:WP_solC1} is given in Section \ref{sec:proof_WP_MF}.
The next result provides the identification of the limit of any solution $\lambda$ to \eqref{eq:conv_gen_lambda} in terms of fixed-point solution to \eqref{eq:fixed_point_intro} as discussed in Section~\ref{sec:criticality_intro}:
\begin{proposition}[Identification of the limit]
\label{prop:conv_X_lambda}
Suppose that $\Phi$ is a Lipschitz continuous function, that $h$ is integrable on $[0, +\infty)$ and that $ \xi$ is bounded with $ \xi_{ t} \xrightarrow[ t\to\infty]{} 0$. Denote by
\begin{equation}
\label{eq:kappa}
\kappa:= \int_{ 0}^{ +\infty} h(u) {\rm d}u.
\end{equation}
Let $x^{ \xi}= \left(x^{ \xi}_{ t}\right)_{ t\geq0}$ be the solution to \eqref{eq:conv_gen_X} (resp. $ \lambda^{ \xi}= \left(\lambda^{ \xi}_{ t}\right)_{ t\geq0}$ be the solution to \eqref{eq:conv_gen_lambda}) with source term $ \xi$. Then, if one among $ \left\lbrace x^{ \xi}, \lambda^{ \xi}\right\rbrace$ admits a limit as $t\to\infty$, both do, and the respective limits $\ell_{ \lambda}:= \lim_{ t\to \infty} \lambda_{ t}^{ \xi}$ and $ \ell_{ x}:= \lim_{ t\to\infty} x_{ t}^{ \xi}$ satisfy 
\begin{equation}
\label{eq:lambda_x_inf}
\begin{cases}
\ell_{ \lambda}& = \Phi \left(\ell_{ x}\right),\\
\ell_{ x}&= \kappa \ell_{ \lambda}
\end{cases}
\end{equation}
In particular, $ \ell_{ \lambda}$ and $ \ell_{ x}$ satisfy the fixed-point relations
\begin{equation}
\label{eq:fixed_point_lambda}
\ell_{ \lambda} = \Phi \left( \kappa \ell_{ \lambda}\right) 
\end{equation}
and
\begin{equation}
\label{eq:fixed_point_X}
\ell_{ x}= \kappa \Phi(\ell_{ x})
\end{equation}
where $ \kappa$ in \eqref{eq:lambda_x_inf}, \eqref{eq:fixed_point_lambda} and \eqref{eq:fixed_point_X} is given by \eqref{eq:kappa}.
\end{proposition}
Proposition~\ref{prop:conv_X_lambda} is a straightforward consequence of the dominated convergence theorem and the continuity of $\Phi$.  From now on, we are interested in the possible limits $\ell_\lambda$ of $\lambda^\xi$, which we will denote by $\ell$ if no confusion arises. 

\medskip

We show now that any fixed-point $\ell$ to \eqref{eq:fixed_point_lambda}, specific equilibrium source terms $\xi= \xi^{eq, \ell}$ gives rise to constant solutions $\lambda\equiv \ell$ to \eqref{eq:conv_gen_lambda} :
\begin{proposition}
\label{prop:equilibrium}
Suppose that $ \Phi$ is Lipschitz continuous, strictly monotone of class $C^{1}$ whose derivative is Lipschitz continuous. Suppose that $h$ is bounded, continuous and integrable on $[0, +\infty)$. Suppose that $\xi$ is $C^{1}$ on $[0, +\infty)$ with $ \xi_{ t} \xrightarrow[ t\to \infty]{}0$. Let $ \lambda= \lambda^{ \xi}$ be the solution to \eqref{eq:conv_gen_lambda} with source term $ \xi$. Then, the following statements are equivalent:
\begin{enumerate}[label=(\alph*)]
\item \label{it:xieq} the source term $\xi$ in \eqref{eq:conv_gen_lambda} is equal to \begin{equation}
\label{eq:xi_stationary}
\xi_{ t}=\xi_{ t}^{ eq, \ell}:= \int_{ -\infty}^{0} h(t-s) \ell {\rm d}s= \ell \int_{ t}^{+\infty} h(u)  {\rm d}u,
\end{equation} for some  $\ell= \Phi \left(\kappa \ell\right)$ solution to the fixed-point Equation \eqref{eq:fixed_point_lambda}.

\item \label{it:lambda_const} the solution $ \lambda$ is constant: $ \lambda_{ t}= \lambda_{ 0}$ for all $t\geq0$.
\end{enumerate}
The constant value $ \lambda_0$ in Item (b) corresponds to the fixed-point $\ell = \Phi( \kappa \ell)$ found in Item (a). We will call $\xi^{eq, \ell}$ defined by \eqref{eq:xi_stationary} the equilibrium source term associated to the fixed-point $\ell$.
\end{proposition}
Proof of Proposition \ref{prop:equilibrium} is given in Section \ref{sec:proof_prop_equilibrium}.

\subsection{Stability and convergence rates around equilibrium in the subcritical case}
\label{sec:subcritical}
Consider in this section some $\ell=\ell_\lambda$ solving \eqref{eq:fixed_point_lambda}. Suppose that $\ell$ satisfies the  local subcritical condition \eqref{eq:cond_stab_ell}. With the notation $ \kappa= \int_{ 0}^{+\infty} h(u) {\rm d}u$, \eqref{eq:cond_stab_ell} rewrites as
\begin{equation*}
\left\Vert h \right\Vert_{ 1}\left\vert \Phi^{ \prime} \left( \kappa\ell\right) \right\vert <1,
\end{equation*}
As discussed in Section~\ref{sec:contributions_intro}, our first main result concerns a local version of Londen's Theorem (Theorem~\ref{th:londen}) around $\ell$, under Condition \eqref{eq:cond_stab_ell}: we require mostly (see \eqref{eq:prox_lambda_ell} below) that $\lambda_t$ ultimately stays within a small neighborhood of $\ell$. Under this assumption, we obtain convergence of $\lambda_t$ to $\ell$ as well as rates of convergence under suitable decay assumptions on both $\xi$ and $h$. Note that contrary to Theorem~\ref{th:londen}, we do not require anything on the monotonicity of both $\Phi$ and $h$ nor on the sign of $h$: Theorem~\ref{th:conv_rate} is valid in a context where possibly $h$ takes negative values. One should mention at this point the previous result of Ney \cite{Ney1977}, which establishes asymptotic estimates  in the particular case where $\Phi(0)=0$ and moment conditions on $h$.  Before stating the result, introduce the notation: 
\begin{equation}
\label{eq:Ht}
    H_t:= \int_t^{+\infty} \vert h(u) \vert {\rm d}u,\ t\geq 0.
\end{equation}
In the second part of Theorem~\ref{th:conv_rate} below, we will require one among the following decay assumptions for both $\xi$ and $H$ given by \eqref{eq:Ht}:
\begin{assumption}
\label{ass:decay}
Let $ \zeta$ be either $\xi$ or $H$ given by \eqref{eq:Ht}. Let $A, a, S$ be positive parameters. We will say that
\begin{enumerate}[label=(\alph*)]
    \item $\zeta$ has exponential decay as $t\to\infty$ with rate $a$ and constant $A$ if
    \begin{align*}
\label{eq:zeta_exp}
\left\vert \zeta_{ t} \right\vert \leq A e^{ -a t},\ t\geq0, & \hfill \left(\mathcal{E}_{a, A}\right)
    \end{align*}
    \item $\zeta$ has polynomial decay as $t\to\infty$ with rate $a$ and constant $A$ if
    \begin{equation}
\label{eq:zeta_pol}
\tag{$\mathcal{P}_{a, A}$}
\left\vert \zeta_{ t} \right\vert\leq A t^{ -a},\ t>0, 
\end{equation}
\item $\zeta$ has bounded support with parameter $S$ if
\begin{equation}
\label{eq:zeta_S}
\tag{$\mathcal{B}_{S}$}
\zeta_{ t}\equiv 0, \ t>S.
\end{equation}
\end{enumerate}
\end{assumption}

\begin{example}
The empty source term  $\xi=\xi^\emptyset\equiv 0$ (recall \eqref{eq:empty_xi}) plays an important role in practical applications. This case is of course covered by all items of Assumption~\ref{ass:decay}. A transposition of the rates of convergence of Theorem~\ref{th:conv_rate} to $\xi=\xi^\emptyset$ is given in Theorem~\ref{th:conv_empty}, Eq.~\eqref{eq:rates_empty} below.
\end{example}

\medskip
The main result is then
\begin{theorem}
\label{th:conv_rate}
Suppose $\Phi$ is a nonnegative twice differentiable Lipschitz function such that $\|\Phi''\|_{\infty}<\infty$, $h$ is integrable on $[0, +\infty)$, $ \xi$ is bounded such that $ \xi_{ t} \xrightarrow[ t\to\infty]{}0$. Suppose that $\ell$ solves \eqref{eq:fixed_point_lambda} and verifies the subcritical condition \eqref{eq:cond_stab_ell}. Let $ \lambda$ be the solution to \eqref{eq:conv_gen_lambda} with source term $ \xi$. Suppose that $\lambda$ is bounded. Fix some $ \varepsilon_{ 0}>0$ small enough such that \begin{equation}
\label{eq:rate_tauT}
\tau=\tau_{\varepsilon_0}:=\left\Vert h \right\Vert_{ 1} \left(\left\vert \Phi^{ \prime}( \kappa\ell) \right\vert+ \varepsilon_{ 0}\frac{ \left\Vert \Phi^{ \prime \prime} \right\Vert_{ \infty}}{ 2}\left( 1+ 2\ell  + \left\Vert \lambda \right\Vert_{ \infty} + \left\Vert h \right\Vert_{ 1} \right)\right)\in [0, 1).
\end{equation}
Consider the following asymptotic proximity of the solution $\lambda$ in an $ \varepsilon_0$-neighborhood of $\ell$: suppose that there exists some $t_0\geq 0$ such that
\begin{equation}
\label{eq:prox_lambda_ell}
 \left\vert \lambda_{ t} - \ell\right\vert < \varepsilon_{ 0},\quad \forall t\geq t_0.
\end{equation}
Then, we have convergence to equilibrium:
\begin{equation*}
    \lambda_{ t} \xrightarrow[ t\to \infty]{} \ell.
\end{equation*}
Secondly, suppose that both $\xi$ and $H$ satisfy one property mentioned in Assumption~\ref{ass:decay} for appropriate parameters that are specified below. Then, there exists some $C>0$ independent of $\varepsilon_0$ and $\sigma(t_0)\geq 0$ such that for all $t\geq \sigma(t_0)$
\begin{equation}
\label{eq:decay_Lambda}
    |\lambda_{t}-\ell |
\le  C \Lambda(t; t_0, \tau)
\end{equation}
where the form of the decay function $ \Lambda(t; t_0, \tau)$ is given as follows: 
\begin{enumerate}[label=(\alph*)]
\item  \label{it:xi_exp_decay} \textit{$ \xi$ with exponential decay:} suppose that there exist $A, a>0$ such that $ \xi$ satisfies $ \left(\mathcal{E}_{a, A}\right)$. Then,
\begin{align}
\label{eq:rates_lambda_exp}
\Lambda(t; t_{ 0}, \tau)=\ \begin{cases}e^{-\sqrt{ b \log \left(1/ \tau\right)}\sqrt{t}} & \text{ if for some $B, b>0$,  $H$ satisfies $ \left(\mathcal{E}_{b, B}\right)$}\\
\log(t) ^{b}t^{-b}& \text{ if for some $B, b>0$,  $H$ satisfies $ \left(\mathcal{P}_{b, B}\right)$}\\
e^{-\frac{\log \left(1/ \tau\right)}{S_h\vee t_0}t}&  \text{ if for some $S_h>0$,  $H$ satisfies $ \left(\mathcal{B}_{S_h}\right)$} .\end{cases}
 \end{align}
\item  \label{it:xi_pol_decay}\textit{$ \xi$ with polynomial decay:} suppose that there exist $A, a>0$ such that $ \xi$ satisfies $ \left(\mathcal{P}_{a, A}\right)$.
Then, 
  \begin{align}
  \label{eq:rates_lambda_pol}
\Lambda(t; t_{ 0}, \tau)=\ \begin{cases}t^{-a}&  \text{ if for some $B, b>0$,  $H$ satisfies $ \left(\mathcal{E}_{b, B}\right)$}\\
t^{-a}\vee \log(t)^b t^{-b}& \text{ if for some $B, b>0$,  $H$ satisfies $ \left(\mathcal{P}_{b, B}\right)$}\\
t^{-a}& \text{ if for some $S_h>0$,  $H$ satisfies $ \left(\mathcal{B}_{S_h}\right)$} .\end{cases}
 \end{align}
 \item  \label{it:xi_bounded_support}\textit{$ \xi$ with bounded support:} suppose that there exist $S_\xi>0$ such that $ \xi$ satisfies $ \left(\mathcal{B}_{S_\xi}\right)$.
Then, 
\begin{align}
\label{eq:rates_lambda_S}
\Lambda(t; t_{ 0}, \tau)=\ \begin{cases}e^{-\sqrt{ b \log \left(1/ \tau\right)}\sqrt{t}} & \text{ if for some $B, b>0$,  $H$ satisfies $ \left(\mathcal{E}_{b, B}\right)$}\\
\log(t) ^{b}t^{-b}& \text{ if for some $B, b>0$,  $H$ satisfies $ \left(\mathcal{P}_{b, B}\right)$}\\
e^{-\frac{\log \left(1/ \tau\right)}{S_\xi\vee S_h\vee t_0}t}&  \text{ if for some $S_h>0$,  $H$ satisfies $ \left(\mathcal{B}_{S_h}\right)$} .\end{cases}
 \end{align}
\end{enumerate}
\end{theorem}
\begin{remark}
The choice of the parameter $\varepsilon_0>0$ defining $\tau = \tau_{\varepsilon_0}\in[0,1)$ in \eqref{eq:rate_tauT} depends on the a priori bound one has on $ \Vert \lambda \Vert_\infty$. This may not be therefore explicit in general. There are however several generic situations where this is the case:
    \begin{enumerate}[label=(\alph*)]
        \item Suppose that $ \Phi$ is bounded. Then, $\Vert \lambda \Vert_\infty\leq \Vert \Phi \Vert_\infty$, so that  $\varepsilon_0$ depends on $\Phi$ and $h$ in an explicit way.
        \item Consider the framework of Theorem~\ref{th:stab_stationary} below: as seen in its proof, the constants $\delta>0$ and $\varepsilon_0>0$ respectively in \eqref{eq:etat_small} and \eqref{eq:stab_equilibrium_eps0} can be made explicit in terms of $ \Phi$, $h$ and $\ell$. In particular, \eqref{eq:stab_equilibrium_eps0} implies that $ \Vert \lambda \Vert_\infty$ can be controlled in an explicit way w.r.t. the parameters of the model.
        \item Consider the framework of Theorem~\ref{th:conv_empty} below: as seen in its proof, we have $\lambda_t\leq \ell$ for all $t\geq 0$, so that $\varepsilon_0$ can be again made explicit in terms of $ \Phi$, $h$ and $\ell$.
    \end{enumerate}
Both $C$ and $\sigma(t_0)$ in \eqref{eq:decay_Lambda} depend in an explicit way on the given parameters $A,B,a,b,S$ as well as on $h, \Phi$ and $\tau$. We refer to Equations \eqref{eq:rate_exp_exp}, \eqref{eq:rate_exp_pol}, \eqref{eq:rate_exp_comp}, \eqref{eq:rate_pol_exp}, \eqref{eq:rate_pol_pol1}, \eqref{eq:rate_pol_pol2} and \eqref{eq:rate_pol_comp} where we give explicit bounds on $C$ and $\sigma$. The constant $\tau= \tau_{\varepsilon_0}$ in \eqref{eq:rate_tauT} enters explicitly in the convergence rates found in \eqref{eq:rates_lambda_exp} and \eqref{eq:rates_lambda_S}, through the term $\log(1/\tau)$. As expected, the faster rate is reached as $\varepsilon_0\to 0$ with $\tau_{\epsilon_0} \xrightarrow[\varepsilon_0\to0]{} \tau_0$, where
\begin{equation*}
    \tau_0:= \Vert h \Vert_1 \vert \Phi^{\prime}(\kappa \ell) \vert\in [0,1).
\end{equation*}
and this rate degenerates as  $\tau_0\to1$. Note also that $C$ remains bounded as $\tau \xrightarrow[\varepsilon_0\to 0]{} \tau_0$.

The time $t_0$ in \eqref{eq:prox_lambda_ell} depends in a possibly nontrivial way on the source term $\xi$ and $\varepsilon_0$. In particular, in case $\vert \lambda_0-\ell\vert > \varepsilon_0$, one has generically that $t_0(\varepsilon_0) \xrightarrow[\varepsilon_0\to0]{}+\infty$ and hence, $ \sigma(t_0)\xrightarrow[\varepsilon_0\to0]{}+\infty$.  
\end{remark} 
In the second main result of this section, we are interested in a certain class of source terms $\xi$ for which Theorem~\ref{th:conv_rate} is applicable \textit{i.e.,} for which \eqref{eq:prox_lambda_ell} is valid, (giving in this sense a partial positive answer to Items~(a) and~(b) of Question~\ref{qu:convergence}). This class consists of small perturbations of the equilibrium source term $\xi^{eq,\ell}$ given in \eqref{eq:xi_stationary}.

\begin{theorem}
\label{th:stab_stationary}
Suppose $\Phi$ is a nonnegative twice differentiable Lipschitz function such that $\|\Phi''\|_{\infty}<\infty$, $h$ is integrable on $[0, +\infty)$, $ \xi$ is bounded such that $ \xi_{ t} \xrightarrow[ t\to\infty]{}0$. Let $ \ell$ a solution to \eqref{eq:fixed_point_lambda} satisfying the stability condition \eqref{eq:cond_stab_ell}. Rewrite the source term $ \xi$ as a perturbation of $ \xi^{ eq, \ell}$ given by \eqref{eq:xi_stationary} as follows
\begin{equation*}
\xi_{ t}= \xi_{ t}^{ eq, \ell} + \eta_{ t}
\end{equation*}
for some bounded $ \eta: [0, +\infty) \to \mathbb{ R}$ with $ \eta_{ t} \xrightarrow[ t\to \infty]{}0$. Then there exist $ \varepsilon_{ 0}>0$ such that for all $\varepsilon\in (0, \varepsilon_0)$, there exists $ \delta>0$ such that the following is true: if $ \eta$ satisfies
\begin{equation}
\label{eq:etat_small}
\sup_{ t\geq0} \left\vert \eta_{ t} \right\vert \leq \delta
\end{equation}
then 
\begin{equation}
\label{eq:stab_equilibrium_eps0}
\sup_{ t\geq0} \left\vert \lambda_{ t} - \ell\right\vert \leq \varepsilon.
\end{equation}
Consequently, $ \lambda_t \xrightarrow[t\to\infty]{}\ell$ and under the same cases~\ref{it:xi_exp_decay}, \ref{it:xi_pol_decay} and \ref{it:xi_bounded_support} for $ \xi$ and $H$ as in Theorem~\ref{th:conv_rate}, we have the following uniform-in-time equivalent of \eqref{eq:decay_Lambda}:  for any $ \tau\in ( \tau_{ 0}, 1)$, there exists some constant $C>0$ independent of $ \varepsilon$ such that for all $t\geq0$,
\begin{equation}
\label{eq:decay_Lambda_unif}
    |\lambda_{t}-\ell |
\le  \varepsilon \wedge C \Lambda(t; 0, \tau),
\end{equation}
where $ \Lambda(t; 0, \tau)$ is the same as in \eqref{eq:rates_lambda_exp}, \eqref{eq:rates_lambda_pol} and \eqref{eq:rates_lambda_S} where $t_{ 0}=0$.

\end{theorem}

\subsection{Stability and instability of equilibrium solutions in the excitatory case}
\label{sec:excitatory}
We have shown in the previous section that the subcritical condition \eqref{eq:cond_stab_ell} induces local stability around constant equilibrium solutions $\ell$ given by \eqref{eq:fixed_point_lambda}. Let us stress again that Theorems~\ref{th:conv_rate} and ~\ref{th:stab_stationary} do not require anything on the sign of $h$: these results are also valid in the case where inhibition is present. The question now concerns the optimality of Condition \eqref{eq:cond_stab_ell} w.r.t. the stability around fixed-points $\ell$. Whereas addressing this issue seems to be difficult in the general case of signed $h$, we answer positively to this problem in the purely excitatory case: in this subsection we assume that $h\geq0$.

\subsubsection{Monotonicity}
The fact that $h$ is nonnegative allows to obtain monotonicity properties that will be useful in the rest of the paper:
\begin{proposition}
\label{prop:monotone}
Suppose that $ \Phi$ is nonnegative and Lipschitz continuous, of class $C^{1}$ whose derivative is Lipschitz continuous. Suppose that $h$ is nonnegative, continuous bounded and integrable on $[0, +\infty)$. Suppose that $\xi$ is $C^{1}$ on $[0, +\infty)$. Let $ \lambda= \lambda^{ \xi}$ be the solution to \eqref{eq:conv_gen_lambda} with initial condition $ \xi$. Then, the following is true:
\begin{enumerate}[label=(\alph*)]
\item \label{it:monotone} Suppose that $ \Phi$ is nondecreasing and that 
\begin{equation}
\label{eq:rho_pos}
\xi'_t+h(t)\Phi(\xi_0)\geq 0, \ \text{(resp. $ \xi'_t+h(t)\Phi(\xi_0)\leq0$)}, \quad \text{for all }t\geq0,
\end{equation}
 Then, the solution $ \lambda$ is nondecreasing (resp. nonincreasing) on $[0, +\infty)$.
\item \label{it:lambda_strictly_increasing}Suppose that $ \Phi^{ \prime}$ is strictly positive on $ [0, +\infty)$ and that one of the following condition holds:
\begin{equation}
\label{eq:cond_h}
\text{ Either (i) }h(0)>0 \text{ or (ii) } h(t)>0, \text{ for all } t>0.
\end{equation}
Suppose also that in addition to \eqref{eq:rho_pos}, we have
\begin{equation}
\label{eq:cond_rho}
\xi'_t+h(t)\Phi(\xi_0)>0 \text{ (resp. $\xi'_t+h(t)\Phi(\xi_0)<0$)  on  $(0, \delta)$  for some $ \delta>0$}.
\end{equation}
Then, $ \lambda$ is strictly increasing (resp. strictly decreasing) on $[0, +\infty)$.
\end{enumerate}
\end{proposition}
\begin{remark}
\label{rem:monotone_empty_case} 
\begin{enumerate}[label=(\alph*)]
\item Condition \eqref{eq:rho_pos} was already considered as a sufficient condition for monotonicity in the linear case where $ \Phi(x)= \mu + x$ in several works: see \textit{e.g.,} Example, p.~258 in \cite{feller1941} and in \cite{Dermitzakis2022}. The fact that we are dealing with a nonlinear kernel raises some significant technical difficulties. The present result is similar in spirit with Th.~2.1 of \cite{Constantin2014}, although their assumptions are different from ours (\textit{e.g.,} it is supposed in \cite{Constantin2014} that $h$ in non increasing and Item~\ref{it:monotone} of Proposition~\ref{prop:monotone} is not addressed in \cite{Constantin2014}).
\item Conditions (i) and (ii) of hypothesis \eqref{eq:cond_h} both cover distinct sets of interesting examples: Condition (i) includes the case of compactly supported $h$ that are non zero in $0$. Condition (ii) includes any Erlang kernels $h_{ n}$ for all $n\geq0$ (see \eqref{eq:Erlang}, note that $h_{ n}(0)=0$ for $n\geq1$ so that Condition (i) is not met). 
\item For any $ \ell_{ 0}\geq0$, the following source term
\begin{equation}
\label{eq:xi_mono}
\xi_t=\xi^{\ell_0}_{ t} = \ell_0 \int_{ t}^{+\infty} h(u) {\rm d}u
\end{equation}
leads to
\begin{equation}
\label{eq:rho_particular_case}
\xi'_t+h(t)\Phi(\xi_0)= \left(\Phi \left(\Vert h\Vert_1\ell_{ 0}\right) - \ell_0\right)h(t),\ t\geq0.
\end{equation} 
Thus, $\xi^{\ell_0}$ in \eqref{eq:xi_mono} boils down to the equilibrium source term $ \xi^{ eq, \ell}$ in \eqref{eq:xi_stationary} if and only if $\ell_{ 0}$ solves the fixed-point relation \eqref{eq:fixed_point_lambda}. Hence, unless $ \ell_0$ solves \eqref{eq:fixed_point_lambda}, under \eqref{eq:cond_h}, the solution $ \lambda^{ \xi}$ is strictly monotone. This is in particular true in the case of the empty source term $\xi^{ \emptyset}$ in \eqref{eq:empty_xi} (take $ \ell_{ 0}=0$ in \eqref{eq:xi_mono}) where $ \lambda$ is strictly increasing (at least when $\Phi(0)>0$). The derivation of Condition \eqref{eq:cond_rho} from \eqref{eq:rho_particular_case} (which is trivial if $h(0)>0$ by continuity) also holds in the second case $(ii)$ of \eqref{eq:cond_h}. This observation is crucial in the case of Erlang kernels (see Definition~\ref{def:Erlang}) as $h_n(0)=0$ for $n\geq 1$ so that $ \xi'_0+h(0)\Phi(\xi_0)=0$. Remark that Condition \eqref{eq:cond_rho} has been specifically designed not to require anything on the value of $\xi'_0+h(0)\Phi(\xi_0)$.
\end{enumerate}
\end{remark}

\subsubsection{Sufficient conditions for local stability/instability}
When $h$ is nonnegative, the fixed-point relation \eqref{eq:fixed_point_lambda} becomes
\begin{equation}
\label{eq:fixed_point_lambda_pos}
    \ell = \Phi\left(\Vert h \Vert_1 \ell\right)
\end{equation} 
and the subcritical condition \eqref{eq:cond_stab_ell} becomes
\begin{equation}
\label{eq:cond_stab_ell_pos}
    \Vert h \Vert_1 \vert\Phi^{\prime} \left(\Vert h\Vert_1 \ell\right)\vert <1.
\end{equation}
We will prove that \eqref{eq:cond_stab_ell_pos} is indeed optimal w.r.t. the local stability of the equilibrium $\ell$, by addressing the subcritical case in \S~\ref{sec:supercritical} and the critical case in \S~\ref{sec:critical_case}. To do so, for any fixed-point $\ell$ solution to \eqref{eq:fixed_point_lambda}, we begin with the identification of a generic class of source terms $\xi$ that are perturbations of $\xi^{eq, \ell}$ given in \eqref{eq:xi_stationary} for which we can easily address the stability (see \eqref{eq:cond_Phi_stable1} and \eqref{eq:cond_Phi_stable2}) or instability (see \eqref{eq:cond_Phi_unstable1} and \eqref{eq:cond_Phi_unstable2}) of the corresponding solution $ \lambda^\xi$ to \eqref{eq:conv_gen_lambda}.

\begin{proposition}
\label{prop:generic_unstable_stable}
Suppose that $ \Phi$ is nonnegative Lipschitz continuous strictly increasing, of class $C^{1}$ with $\Phi^\prime$ Lipschitz continuous. Suppose that $h$ is nonnegative, bounded continuous and integrable on $[0, +\infty)$. Define the following class of perturbations of the equilibrium source term $ \xi^{ eq, \ell}$: let $ \chi:[0, +\infty)\to [0, +\infty]$ be a $ \mathcal{ C}^{ 1}$ strictly decreasing function such that $ \chi_{ 0}={\|h\|_1}$ and $ \chi_{ t} \xrightarrow[ t\to\infty]{}0$. For any given $ \ell_{ 0}\geq0$, define
\begin{equation}
\label{eq:xi_diverges}
\xi_{ t}^{\chi, \ell_{ 0}}:= \Phi({\|h\|_1}\ell_{ 0}) \int_{ t}^{ +\infty} h(u) {\rm d}u + \left( \ell_{ 0}- \Phi(\|h\|_1\ell_{ 0})\right)\chi_{ t},\ t\geq0.
\end{equation}
Consider the solution $ \lambda:= \lambda^{\chi, \ell_{ 0}}$ to \eqref{eq:conv_gen_lambda} with source term $ \xi^{\chi, \ell_{ 0}}$ given by \eqref{eq:xi_diverges}. Let $\ell$  be some solution to the equation  $ \ell = \Phi(\|h\|_1\ell)$. Then, the following is true:
\begin{enumerate}[label=(\alph*)]
    \item $ \xi^{\chi, \ell_{ 0}}$ in \eqref{eq:xi_diverges} is a uniform approximation of the equilibrium source term $ \xi^{eq, \ell}$ given by \eqref{eq:xi_stationary} for $\ell_{ 0}$ close to $\ell$:
\begin{equation*}
\sup_{ t\geq0} \left\vert \xi^{\chi, \ell_{ 0}}_{ t} - \xi^{ eq, \ell}_{ t} \right\vert \to 0,\ \text{ as }\ell_{ 0} \to \ell.
\end{equation*}  
\item Sufficient conditions for local instability: suppose that there exists some $ \delta>0$ such that one of the following holds
\begin{align}
&\text{for all $ \ell_{ 0}\in \left(\ell, \ell+ \delta\right)$, }\Phi(\|h\|_{1}\ell_{ 0})> \ell_{ 0} 
\label{eq:cond_Phi_unstable1}\\
\text{or }&\text{for all $ \ell_{ 0}\in \left(\ell- \delta, \ell\right)$, } \Phi(\|h\|_{1}\ell_{ 0})< \ell_{ 0} 
.\label{eq:cond_Phi_unstable2}
\end{align}
Then, if \eqref{eq:cond_Phi_unstable1} (resp. \eqref{eq:cond_Phi_unstable2}) holds, for all $\ell_{ 0}\in (\ell, \ell+\delta)$ (resp. $\ell_{ 0}\in (\ell-\delta, \ell)$), the solution $ \lambda^{\chi, \ell_{ 0}}$  is such that 
\begin{equation*}
\liminf_{ t\to\infty} \left\vert \lambda^{\chi, \ell_{ 0}}_{ t} - \ell\right\vert>0.
\end{equation*}
\item Sufficient conditions for local stability: suppose that there exists some $ \delta>0$ such that one of the following holds
\begin{align}
&\text{for all $ \ell_{ 0}\in \left(\ell, \ell+ \delta\right)$, }\Phi(\|h\|_{1}\ell_{ 0})< \ell_{ 0} 
\label{eq:cond_Phi_stable1}\\
\text{or }&\text{for all $ \ell_{ 0}\in \left(\ell- \delta, \ell\right)$, } \Phi(\|h\|_{1}\ell_{ 0})> \ell_{ 0} 
.\label{eq:cond_Phi_stable2}
\end{align}
Then, if \eqref{eq:cond_Phi_stable1} is satisfied, for any $\ell_0\in (\ell, \ell+ \delta)$, for all $t\geq0$, $ \lambda^{\chi,\ell_0}_{ t}\geq \ell$ and $ \lambda^{\chi,\ell_0}_{ t} \xrightarrow[ t\to \infty]{}\ell$, and if \eqref{eq:cond_Phi_stable2} is satisfied, for any $\ell_0\in (\ell-\delta, \ell)$, for all $t\geq0$, $ \lambda^{\chi, \ell_0}_{ t}\leq \ell$ and $ \lambda^{\chi,\ell_0}_{ t} \xrightarrow[ t\to \infty]{}\ell$.
\end{enumerate}
\end{proposition} Proof of Proposition \ref{prop:generic_unstable_stable} is given in Section \ref{sec:proof_prop_generic}.

\begin{remark}
\label{rem:comments_stab_instab}
\begin{enumerate}[label=(\alph*)]
\item Taking $\chi_t=\int_t^\infty h(u)\rmd u$ in \eqref{eq:xi_diverges}, we see that $\xi^{\chi, \ell_0}$ is equal to $\xi^{\ell_0}_t= \ell_0 \int_{ t}^{+\infty} h(u) {\rm d}u$ given by \eqref{eq:xi_mono}. 
\item Under the subcritical condition \eqref{eq:cond_stab_ell_pos}, Conditions \eqref{eq:cond_Phi_stable1} and \eqref{eq:cond_Phi_stable2} are trivially satisfied. This is of course compatible with the results of Theorems~\ref{th:conv_rate} and~\ref{th:stab_stationary}. Proposition~\ref{prop:generic_unstable_stable}
states here a weaker notion of stability, as it is restricted to the class of perturbations given by \eqref{eq:xi_diverges}. Interestingly, Proposition \ref{prop:generic_unstable_stable} allows  to consider the critical case $\|h\|_1\Phi'(\|h\|_1\ell)=1$: in this case, one will need to look at the sign of  the first higher order derivative $\Phi^{ (p)}(\|h\|_1\ell)$ that is nonzero (together with the parity $p$) to obtain sufficient conditions for \eqref{eq:cond_Phi_stable1} (stability starting from the right) or \eqref{eq:cond_Phi_stable2} (stability starting from the left). This will be the purpose of Section~\ref{sec:critical_case}.
\item \label{it:unstab_comment} Conversely, in the supercritical case $\|h\|_1\Phi'(\|h\|_1\ell)>1$, both Conditions \eqref{eq:cond_Phi_unstable1} (instability from the right) and \eqref{eq:cond_Phi_unstable2} (instability from the left) are satisfied. This is in itself sufficient to show the optimality of the Condition \eqref{eq:cond_stab_ell_pos} for the stability of $\ell$ (at least in the case where $h\geq 0$). However, in the supercritical regime, we are able to show instability for a wider class of perturbations, see \S~\ref{sec:supercritical}.
\end{enumerate}
\end{remark}

\subsubsection{Supercritical fixed-points}
\label{sec:supercritical}
We suppose in this paragraph that the fixed-point $\ell= \Phi(\Vert h\Vert_1\ell)$ satisfies now the supercritical condition
\begin{equation}
\label{eq:Phi_supercritical}
\Vert h\Vert_1\Phi^{ \prime}(\Vert h\Vert_1\ell)>1.
\end{equation}
The claim is that Condition \eqref{eq:Phi_supercritical} induces instability for the dynamics \eqref{eq:conv_gen_lambda} starting around $ \xi^{ eq,\ell}$. Recall here Item~\ref{it:unstab_comment} of Remark~\ref{rem:comments_stab_instab}: $\ell$ is then unstable from above and below in the sense of Proposition~\ref{prop:generic_unstable_stable}. The main result of this section extends the regime of instability to both generic positive and negative nontrivial perturbations of $ \xi^{ eq, \ell}$. This however requires some further (but natural) assumptions on $ \Phi$ and $h$.
\begin{theorem}
\label{th:supercritical_unstable2}
Suppose that $ \Phi$ is nonnegative Lipschitz continuous strictly increasing, of class $C^{1}$ whose derivative is Lipschitz continuous. Let $\ell$ be some fixed-point solution to \eqref{eq:fixed_point_lambda_pos} such that the supercritical condition \eqref{eq:Phi_supercritical} is true. Suppose that  $ \Psi: y \mapsto \Phi \left(\Vert h\Vert_1 y\right)$ is convex-concave around $\ell$, that is convex on $(\ell- \delta, \ell]$ and concave on $[\ell,\ell+ \delta)$ for some $ \delta>0$. Suppose that $h$ is strictly positive, bounded continuous and integrable on $[0, +\infty)$. Write the source term as $ \xi_{ t}= \xi_{ t}^{ eq, \ell} + \eta_{ t}$ with $\eta$ bounded and continuous on $[0, +\infty)$ with $ \eta_{ t} \xrightarrow[ t\to\infty]{}0$. Suppose that one of the following condition holds:
\begin{enumerate}[label=(\alph*)]
\item[(a)] Instability from above: suppose that $ \eta_{ t}\geq0$ for all $t\geq0$ and $ \eta_{ 0}>0$.
\item[(b)] Instability from below: suppose that $ \eta_{ t}\leq0$ for all $t\geq0$ and $ \eta_{ 0}<0$.
\end{enumerate}
Then, in each cases (a) and (b), the solution $ \lambda$ of \eqref{eq:conv_gen_lambda} with source term $ \xi$ does not converge to $\ell$ as $t\to\infty$.
\end{theorem}
Proof of Theorem \ref{th:supercritical_unstable2} is given in Section \ref{sec:proof_th_critical_unstable2}.

\begin{remark}
The proof of Theorem~\ref{th:supercritical_unstable2} relies on similar arguments as in Theorem~4 of Brauer \cite{Brauer1975}. The purpose of \cite[Th.~4]{Brauer1975} was to address the case where $\ell=0$ with $\Phi(0)=0$ and $\Phi^\prime(0) \Vert h \Vert_1>1$. One can see Theorem~\ref{th:supercritical_unstable2} as a generalization of the result of Brauer to all fixed-points $\ell$. In particular, it is written in \cite{Brauer1975}, page~317: \emph{"It is natural to conjecture that a condition like that of Theorem~4 should be necessary for nonzero limits of solutions of \eqref{eq:conv_gen_lambda}. However, the proof depends strongly on the fact that the approach to the limit zero is necessarily from above. For a nonzero limit, the solution may oscillate about the limit and it does not appear that any general result of this nature can be true."} We claim that the hypothesis of convexity-concavity of Theorem~\ref{th:supercritical_unstable2} allows to transport the argument of Brauer to every positive fixed-point $\ell$. The existence of damped oscillatory solutions that converge towards some fixed-point is also mentioned in the seminal article of Feller \cite{feller1941}, in the linear case. We confirm this observation in \S~\ref{sec:stable_manifold_Erlang} with an explicit solution with damped oscillations converging towards some supercritical point in the case of Erlang kernels.
\end{remark}

\subsubsection{The critical case}
\label{sec:critical_case}
Assume in this paragraph that the fixed-point $\ell$ solution to \eqref{eq:fixed_point_lambda_pos} is locally critical, \textit{i.e.,}
\begin{equation}
\label{eq:local_crit}
\Vert h\Vert_1\Phi^{ \prime} \left(\Vert h\Vert_1\ell\right)=1.
\end{equation}
In this case, the stability around the equilibrium solution $ \xi^{ eq, \ell}$ depends on higher derivatives of $ \Phi$ in $\Vert h\Vert_1\ell$. 
\begin{assumption}
\label{ass:crit_p}
Suppose that $ \Phi$ is of class $ \mathcal{ C}^{ \infty}$ in a neighborhood of $\|h\|_1\ell$ such that the sequence of derivatives $ \left(\Phi^{ (k)}(\Vert h\Vert_1\ell)\right)_{ k\geq2}$ is not equally zero: denote as 
\begin{equation}
\label{eq:p}
p:= \min \left\lbrace k\geq2,\  \Phi^{ (k)}(\Vert h\Vert_1\ell) \neq 0\right\rbrace.
\end{equation}
\end{assumption}
Under the latter assumption, Proposition~\ref{prop:generic_unstable_stable} applies: a simple Taylor expansion allows to derive directly the sign of $\Phi \left(\Vert h\Vert_1 \ell_0\right)-\ell_0$ in a neighborhood of $\ell$ from the sign of $\Phi^{ (p)}(\Vert h\Vert_1\ell)$ and the parity of $p$. Table~\ref{tab:stability_p} summarizes the corresponding results.
\begin{table}[ht]
  \centering
  \begin{tabular}{cll}
    \toprule
     & \multicolumn{1}{c}{$ \Phi^{ (p)}(\Vert h\Vert_1\ell)>0$} & \multicolumn{1}{c}{$ \Phi^{ (p)}(\Vert h\Vert_1\ell)<0$} \\[.5\normalbaselineskip]
    \midrule \\
    $p$ even & Instability from above (see \eqref{eq:cond_Phi_unstable1})& Stability from above (see \eqref{eq:cond_Phi_stable1})\\
    & Stability from below (see \eqref{eq:cond_Phi_stable2})& Instability from below (see \eqref{eq:cond_Phi_unstable2})\\[.5\normalbaselineskip]
    \midrule \\
     $p$ odd&  Instability from above (see \eqref{eq:cond_Phi_unstable1})& Stability from above (see \eqref{eq:cond_Phi_stable1})\\
    & Instability from below (see \eqref{eq:cond_Phi_unstable2})& Stability from below (see \eqref{eq:cond_Phi_stable2})\\[.5\normalbaselineskip]
    \bottomrule\\
  \end{tabular}
  \caption{Local stability around the equilibrium source term $ \xi^{ eq, \ell}$ under Assumption~\ref{ass:crit_p} depending on the parity of $p$ given by \eqref{eq:p} and the sign of $ \Phi^{ (p)}(\Vert h\Vert_1\ell)$. Stability and instability are to be understood in the context of the results of Proposition~\ref{prop:generic_unstable_stable} (we assume here that the hypotheses of this Proposition are satisfied).}
\label{tab:stability_p}
\end{table}
\begin{remark}
It is also possible to consider the degenerate case where 
\begin{equation}
\label{eq:degenerate}
\left\Vert h \right\Vert_{ 1} \Phi^{ \prime}(\Vert h\Vert_1\ell)=1, \text{ and } \Phi^{ (k)}(\Vert h\Vert_1\ell)=0 \text{ for all } k\geq2.
\end{equation}
A typical example satisfying \eqref{eq:degenerate} is of the form $ \Phi(x) \sim x+ a e^{ - \frac{ 1}{ (x-l)^{ 2}}} \mathbf{ 1}_{ x> \Vert h\Vert_1\ell} + b e^{ - \frac{ 1}{ (x-l)^{ 2}}} \mathbf{ 1}_{ x< \Vert h\Vert_1\ell}$ ($a,b\in \mathbb{ R}$) as $x\to \Vert h\Vert_1\ell$ for some fixed-point $\ell$ solution to \eqref{eq:fixed_point_lambda_pos}. A similar analysis as in Table~\ref{tab:stability_p} can be made depending on the signs of $a$ and $b$ (details left to the reader).

The case where $ \Phi$ satisfies \eqref{eq:degenerate} with $ \Phi$ being analytic around $ \ell$ is even more degenerate, as it corresponds to the case where there is some $ \delta>0$ such that $ \Phi(\Vert h\Vert_1 \ell_0)=\ell_0$ on $ \left(\ell- \delta, \ell+ \delta\right)$. In this case, stability around the fixed-point $\ell$ obviously does not hold, as for any $ \varepsilon\in \left(0, \delta\right)$, $ \lambda^{\ell+ \varepsilon}$ solution to \eqref{eq:conv_gen_lambda} starting from $ \xi^{ eq, \ell+ \varepsilon}$ is constant, equal to $\ell+ \varepsilon$ (recall Proposition~\ref{prop:monotone}) and does not converge to $\ell$.
\end{remark}

\subsubsection{Stability around the empty source term}
We analyse in this paragraph the asymptotic behavior of \eqref{eq:conv_gen_lambda} starting from (perturbations of) the empty source term $\xi^{\emptyset}$ given by  \eqref{eq:empty_xi}. The case $ \Phi(0)=0$ reduces to the fixed-point stability problem mentioned in Section~\ref{sec:subcritical}. Then $\ell=0$ is a solution to \eqref{eq:fixed_point_lambda_pos} and $\xi^{\emptyset}$ correspond to the equilibrium source term $ \xi^{ eq, 0}$. In particular, starting exactly from \eqref{eq:empty_xi} leads to the constant solution $ \lambda\equiv0$ and the stability analysis around \eqref{eq:empty_xi} goes back to the previous results of Section~\ref{sec:subcritical}.   Therefore, we assume now $\Phi(0)>0$ and that the set of solutions $\ell>0$ to \eqref{eq:fixed_point_lambda_pos} is not empty and we define $\ell_1$ as the infimum of such solutions:
\begin{equation}
\label{eq:ell1}
    \ell_{ 1}:=\inf \left\lbrace \ell>0,\ \Phi(\Vert h \Vert_1\ell)= \ell\right\rbrace.
\end{equation}
Let us introduce the class of perturbations we consider.
\begin{assumption}
\label{ass:xi_empty}
Suppose that the source term $ \xi$ satisfies $\xi_{ 0}< \Vert h\Vert_1\ell_{ 1}$, $\xi_{t} \xrightarrow[ t\to\infty]{}0$, for all $t\geq 0$, $\xi_{ t}^{ \prime} + h(t) \Phi \left(\xi_{ 0}\right)\geq0$ and that there exists some $\delta>0$ such that $\xi_{ t}^{ \prime} + h(t) \Phi \left(\xi_{ 0}\right)>0$ on $(0, \delta)$.
\end{assumption}
\begin{remark}
A class of source terms $ \xi$ satisfying Assumption~\ref{ass:xi_empty} corresponds to $\xi^{\chi, \ell_{ 0}}$ as defined in \eqref{eq:xi_diverges}, for any $\ell_{ 0}< \ell_{ 1}$ and $ \chi$ of class $C^{ 1}$ on $[0, +\infty)$, with $ \chi_{ 0}= \left\Vert h \right\Vert_{ 1}$, $ \chi_{ t} \xrightarrow[ t\to\infty]{} 0$ and $ \chi_{ t}^{ \prime}\leq 0$ with $ \chi_{ t}^{ \prime}< 0$ on $(0, \delta)$. 
The empty source term $ \xi^{ \emptyset}$ corresponds to the choice of $ \ell_{ 0}=0$ and 
$\chi^{ \emptyset}_{ t}=  \int_{ t}^{ +\infty} h(u) {\rm d}u$.
\end{remark}
The main result is the following:
\begin{theorem}
\label{th:conv_empty}
Suppose that $ \Phi(0)>0$, $ \Phi$ is nonnegative Lipschitz continuous, of class $C^{1}$ with $\Phi^\prime$ Lipschitz continuous and $ \Phi^{ \prime}>0$. Suppose that $h$ is nonnegative, bounded continuous and integrable on $[0, +\infty)$ and verifies \eqref{eq:cond_h}. Suppose that the source term $ \xi$ satisfies Assumption~\ref{ass:xi_empty}. Then, we have 
\begin{equation*}
\lambda_{ t} \xrightarrow[ t\to\infty]{} \ell_{ 1}
\end{equation*}
where $ \ell_{ 1}>0$ is given by \eqref{eq:ell1}. In particular, for all $\tau\in [0,1)$, there exists some  constants $C, c>0$ such that for all $t\geq 0$
\begin{align}
\label{eq:rates_empty}
\vert \lambda_t -\ell_1 \vert \leq C \begin{cases}e^{-\sqrt{ b \log \left(1/ \tau\right)}\sqrt{t}} & \text{ if for some $B, b>0$,  $H$ satisfies $ \left(\mathcal{E}_{b, B}\right)$}\\
\log(t) ^{b}t^{-b}& \text{ if for some $B, b>0$,  $H$ satisfies $ \left(\mathcal{P}_{b, B}\right)$}\\
e^{- c\log \left(1/ \tau\right)t}&  \text{ if for some $S_h>0$,  $H$ satisfies $ \left(\mathcal{B}_{S_h}\right)$} .\end{cases}
 \end{align}
\end{theorem} Proof of Theorem \ref{th:conv_empty} is given in Section \ref{sec:proof_th_empty}. Estimate \eqref{eq:rates_empty} is nothing else than a simple application of the rates of convergence found in \eqref{eq:rates_lambda_exp} to the present situation ($\xi^\emptyset$ has exponential decay with any rate $a>0$). 

\begin{remark}
Note that Inequality \eqref{eq:rates_empty} holds for any $t\geq0$ whereas the original inequality \eqref{eq:decay_Lambda} is written for $ t\geq \sigma(t_{ 0})$, where $t_0$ is the time  for $\lambda_t$ to enter an $\varepsilon_0$-neighborhood of $\ell_1$, see \eqref{eq:prox_lambda_ell}. This formulation is to acknowledge the fact that obtaining explicit estimates on $t_{ 0}$ is certainly out of reach in the case $ \xi\equiv 0$. As it is, there is some hidden dependence of $C$ and $c$ in the unknown parameter $t_{ 0}$ in \eqref{eq:rates_empty}. \end{remark}

\begin{remark}
\label{rem:empty}
The convergence to $ \ell_{ 1}$ starting from (perturbations of) the empty source term is valid disregarding the value of the derivatives of $ \Phi$ in $\ell_{ 1} \Vert h\Vert_1$. Note that, since $ \Phi(0)>0$, we necessarily have that $ \Vert h\Vert_1\Phi^{ \prime}(\ell_{ 1}\Vert h\Vert_1)\leq1$. In particular, it may very well be that $ \Vert h\Vert_1\Phi^{ \prime}(\ell_{ 1}\Vert h\Vert_1)=1$ and (for instance) $ \Phi^{ \prime \prime}(\ell_{ 1}\Vert h\Vert_1)>0$: $\ell_{ 1}$ may be unstable w.r.t. source terms around $ \xi^{ eq, \ell_{ 1}}$ (recall \S~\ref{sec:critical_case}) but however may attract trajectories starting from perturbations of $ \xi= \xi^{ \emptyset}$. 
\end{remark}

\subsection{Application: a functional central limit theorem in the subcritical Hawkes model} 
\label{sec:application_Hawkes}
Let us go back here to the mean-field Hawkes process $\left(Z^{1,N}, \ldots, Z^{N,N}\right)$ whose conditional intensity is given by \eqref{eq:Hawkes_N}. 
Let us recall the hypotheses at stake here:
\begin{assumption}
\label{ass:TCL}
Suppose that $\Phi$ is nonnegative twice differentiable Lipschitz with $\|\Phi''\|_{\infty}<\infty$. Suppose that $h$ a square-integrable function, bounded and continuous on $[0, +\infty)$ (note in particular that we do not assume anything on the sign of $h$).

 Let $N\ge 1$ and consider a family of predictable processes $\xi_N=(\xi_{N,t})_{t\ge 0}$ and a bounded source term $\xi$ with $\xi_t\xrightarrow[t\to\infty]{}0$ such that $\sup_{s\ge 0}\E[|\xi_{N,s}-\xi_s|]\le C_\xi/\sqrt{N}$  for some positive constant $C_\xi>0$. Define $\lambda=\lambda^\xi$ as the solution to \eqref{eq:conv_gen_lambda} with source term $\xi$ and let $\ell$ be a solution to \eqref{eq:fixed_point_intro} such that $\lambda^\xi_t \xrightarrow[t\to\infty]{}\ell$.
\end{assumption}
  From a statistical viewpoint, a classical estimator of $\ell$ is, for some $i\in\{1,\ldots,N\}$,
\begin{equation}
\label{eq:hat_ell}
\widehat \ell_{N, t}:=\frac{Z^{i,N}_t}{t},\ t\geq 0.
\end{equation} 
By exchangeability, the law of $\widehat \ell_{N, t}$ does not depend on $i$. The properties of $\widehat \ell_{N, t}$ in \eqref{eq:hat_ell} can be studied by determining the limiting distribution of $Z^{i, N}_t$ for large $N$ and $t$.  We are interested in two kinds of estimates: (i) $L^{ 1}$-proximity estimate of $\hat{\ell}_{ N,t}$ w.r.t. $\ell$ and (ii) the asymptotic normality of $\widehat \ell_{N, t}$, that is the behavior of 
\begin{equation}
\label{eq:CLT_hat_ell}
    \sqrt{ t} \left( \widehat \ell_{N,t} - \ell\right),\ t\geq 0,\ N\geq 1,
\end{equation} as $(t, N)\to\infty$. To do so, we build upon some slightly different fluctuation result already proven in \cite[Theorem 10]{MR3449317} for the linear case and in \cite[Theorem 2]{DITLEVSEN20171840} for the nonlinear case. Both results are proven under the assumption that $\xi=\xi^\emptyset$ given by \eqref{eq:empty_xi}, under the strong subcriticality condition \eqref{eq:subcritical_Lip} and concern the asymptotic behavior of the quantity
\begin{equation}
 U_{ t}^{ i, N}:= \frac{ Z^{i,N}_{  t}}{ {m_{ t}^\emptyset}} - 1
\end{equation}where
 $m^\emptyset_t:=\int_0^t\lambda^\emptyset_s\rmd s$, where $\lambda^\emptyset$ is the solution to \eqref{eq:conv_gen_lambda} driven by $\xi^\emptyset$. More precisely, for the nonlinear case that concerns us presently, it is proven in \cite{DITLEVSEN20171840} that for ${\xi=}\xi^\emptyset$, (see eq. (A.37) in \cite{DITLEVSEN20171840})
 \begin{equation}
 \label{eq:L1Ui}
\mathbb{ E} \left( \left\vert \frac{ Z^{i,N}_{  t}}{ {m_{ t}^\emptyset}} - 1 \right\vert\right) \leq \frac{ 1}{ \sqrt{ m_{ t}}}+\frac{ Ct}{ \sqrt{ N}{ m_{ t}}}
 \end{equation}
 as well as (see \cite[Theorem 2]{DITLEVSEN20171840})
\begin{equation}
\label{eq:CLT_DL}
    \sqrt{ {m_{ t}^\emptyset}} \left( \frac{ Z^{i,N}_{  t}}{ {m_{ t}^\emptyset}} - 1\right)\xrightarrow[{(t,N)\to(\infty,\infty)},\ {\frac{t}{N}\to0}]{d}\mathcal{N}(0,1).
\end{equation}  
One aim of the result below is to first extend both these results to general source terms $ \xi$ and secondly refine them to make them statistically usable and turn them into something relevant to \eqref{eq:CLT_hat_ell}: as it is, the quantity $m_t^{\emptyset}$ in \eqref{eq:L1Ui} and \eqref{eq:CLT_DL} is difficult to quantify, and careful control of its properties when $t$ gets large is necessary to obtain a result that is relevant from an applied perspective. Turning now to a generic situation where $ \xi$ is not necessarily equal to $ \xi^{ \emptyset}$, writing $ m_{ t}:= \int_{0}^{ t} \lambda_{ s} {\rm d}s$, we proceed as in \cite{duval2021interacting} and write the decomposition
\begin{align}
    \label{eq:Dec_ellT1}
 \widehat \ell_{N,t} - \ell= \frac{ m_{ t}}{ t} \left( \frac{ Z^{i,N}_{  t}}{ m_{ t}} - 1\right) + \left( \frac{ m_{ t}}{ t} - \ell\right).
\end{align}
An equivalent version of \eqref{eq:Dec_ellT1} that is adapted to the fluctuation result we have in mind is
\begin{align}
    \label{eq:Dec_ellT}
\sqrt{ t} \left( \widehat \ell_{N,t} - \ell\right)= \sqrt{ \frac{ m_{ t}}{ t}} \sqrt{ m_{ t}} \left( \frac{ Z^{i,N}_{  t}}{ m_{ t}} - 1\right) + \sqrt{ t} \left( \frac{ m_{ t}}{ t} - \ell\right):=  I_{ N, t}(1) + I_{t}(2).
\end{align} 
The control of the first term $I_{N,t}(1)$ is similar to \eqref{eq:CLT_DL}. The result of Theorem \ref{th:conv_rate} provides a valuable answer to the estimation of the second term $I_t(2)$ in the left hand side of \eqref{eq:Dec_ellT} as, thanks to Theorem \ref{th:conv_rate}, the limits of the terms $\sqrt{m_t}/t$ and $I_t(2)$ can be rendered explicit. Thus, we obtain the following result, established not only for $\xi=\xi^{\emptyset}$ but for a general source term $\xi$.

\begin{theorem}
\label{th:TCL}
Suppose that Assumption~\ref{ass:TCL} holds. Suppose that Condition \eqref{eq:subcritical_Lip} is true.  Additionally, suppose that $ \xi$ satisfies \eqref{eq:zeta_pol} for some $A>0$ and $a>1/2$ and $H$ given in \eqref{eq:Ht} satisfies $ \left(\mathcal{P}_{b, B}\right)$ for some $B>0$ and $b>1/2$.
Let $ (Z^{i,N}_t)_t\ge 0, i=1\ldots,N$, be the family of Hawkes processes defined as in Section~\ref{sec:NRE_intro} with intensity function \eqref{eq:Hawkes_N}, with source term $ \xi_N$, kernel $h$ and function $\Phi.$ Let $m_t:=\int_0^t\lambda_s\rmd s$. Then, for $\ell$ the unique solution to Equation \eqref{eq:fixed_point_lambda}, the following is true:

\begin{enumerate}
\item There exists a constant $C>0$ independent of $t$ and $N$ such that
\begin{equation}
\label{eq:L1hatell}
\mathbb{ E} \left( \left\vert \hat{\ell}_{ N, t} - \ell \right\vert\right) \leq C \left( \frac{ 1}{ \sqrt{ t}} + \frac{ 1}{ \sqrt{N}}\right).
\end{equation}
\item  For all $i=1,\ldots,N$,
\begin{align}\label{eq:Func_CLT} \left(\sqrt{m_t}\frac{ Z^{i,N}_{ut}-m_{ut}}{m_t}\right)_{u\in[0,1]}\xrightarrow[(t,N)\to (\infty,\infty),\ \frac{t}{N}\to 0]{d} B,\end{align} where $ \xrightarrow[]{d}$ is the convergence in distribution and $B=(B_u)_{u\in[0,1]}$ is a standard Brownian motion.
As a corollary
\begin{align}
\label{eq:Res_TCL_est}
\sqrt{ t} \left( \widehat \ell_{N,t} - \ell\right)\xrightarrow[{(t,N)\to(\infty,\infty)},\ {\frac{t}{N}\to0}]{d}\mathcal{N}(0,\ell).\end{align}
\end{enumerate}
\end{theorem}

Proof of Theorem \ref{th:TCL} is given in Appendix \ref{sec:prf_TCL}.

\begin{remark}
The  hypothesis made on $\xi$ and $H$ in Theorem~\ref{th:TCL} corresponds to the slowest decay behavior within Assumption~\ref{ass:decay} for which Theorem~\ref{th:conv_rate} ensures that $I_t(2)$ in \eqref{eq:Dec_ellT} goes to $0$ as $t\to\infty$. If $\xi$ (resp. $H$) satisfies the other cases \eqref{eq:zeta_exp} and \eqref{eq:zeta_S} of Assumption~\ref{ass:decay}, $\xi$ (resp. $H$) also satisfies \eqref{eq:zeta_pol}, changing the constants $A$ and $B$ if necessary.

In \cite{MR3054533} a similar result is proven in case where $\Phi$ is linear, $h\geq 0$ and $\xi=\xi^\emptyset$. More precisely, it is proven in \cite[Lemma 5]{MR3054533} that $I_{ t}(2)\to 0$ as $t\to\infty$ under the following condition
\begin{equation}
\label{eq:cond_BDHM}
\int_0^\infty\sqrt{t}h(t)\rmd  t<\infty.
\end{equation}
Theorem \ref{th:conv_rate} allows to handle this term in the nonlinear case and with a non zero source term, extending \cite[Corollary 1]{MR3054533} to these cases. Comparing the Condition \eqref{eq:cond_BDHM} and the one imposed by Theorem \ref{th:conv_rate}, we observe that they are almost equivalent. Indeed, supposing that   $\Gamma:=\int_0^\infty\sqrt{t}h(t)\rmd  t<\infty$ then,  for any $t\ge 0$  $$\sqrt{t}\int_t^\infty h(u)\rmd u\le \int_t^\infty\sqrt{u}h(u)\rmd u\le \Gamma $$ and the constraint $\int_t^\infty h(u)\rmd u\le B t^{-b}$ for any $t\ge0$ is satisfied for $B=\Gamma$ and $b=1/2.$ Conversely, if $\int_t^\infty h(u)\rmd u\le B t^{-b}$ for any $t\ge0$ is satisfied for some  $B>0$ and $b>1/2,$ then, writing for $u\geq 1$, $\sqrt{u}= \int_{1}^{\sqrt{u}} {\rm d}y +1$ and using Fubini-Tonelli Theorem, we get
\begin{align*}
    \int_0^\infty \sqrt{u}h(u)\rmd u
    &= \int_0^1 \sqrt{u}h(u)\rmd u+ \int_1^\infty\int_{y^2}^\infty h(u)\rmd u\rmd y + \int_1^{+\infty}h(u) {\rm d}u\\
    &\leq \int_0^1 \sqrt{u}h(u)\rmd u + \int_1^\infty \frac{B}{y^{2b}}\rmd y+ \int_1^{+\infty}h(u) {\rm d}u<\infty.
\end{align*}

Finally, a detailed analysis of fluctuations for univariate Hawkes processes in the linear case is done in \cite{Horst2024}, addressing in particular the optimality of the condition of \cite{MR3054533} w.r.t. the existence of Central Limit Theorems. We refer to \cite{Horst2024} for details.
\end{remark}

\section{Particular cases and examples}
\label{sec:examples}
The aim of this Section is to expand the local estimates of the previous section into global study of particular cases (concerning simultaneously  $ \Phi$, $ \xi$ and $h$) that are particularly relevant w.r.t. applications.

\subsection{The case with no fixed-point}

We start with the case where the set of fixed-points solutions to \eqref{eq:fixed_point_intro} is empty (in particular $\Phi(0)>0$ and $\Phi$ is unbounded).
\begin{proposition}
\label{prop:nofixedpoint}
Suppose that $ \Phi$ is nonnegative and continuously differentiable. Suppose that $h$ is nonnegative, nonincreasing continuous and integrable on $[0, +\infty)$. Suppose that $ \xi$ is bounded on $[0, +\infty)$ and verifies $ \xi_{ t} \xrightarrow[ t\to\infty]{}0$. Suppose that equation \eqref{eq:fixed_point_lambda_pos} has no solution. Then, the solution $ \lambda$ to \eqref{eq:conv_gen_lambda} with source term $ \xi$ verifies
\begin{equation*}
\lambda_{ t} \xrightarrow[ t\to\infty]{}+\infty.
\end{equation*}
\end{proposition}
Note that Proposition~\ref{prop:nofixedpoint} is the only moment in this paper where we make use of Londen's Theorem. In particular, it is restricted to nonnegative nonincreasing kernels $h$.
\begin{proof}[Proof of Proposition~\ref{prop:nofixedpoint}]
Denote by $ \Psi(\ell)= \Phi \left( \left\Vert h \right\Vert_{ 1}\ell\right)$. Since $ \Phi(0)>0$, the fact that, by assumption, there is no such $\ell$ such that $ \Psi(\ell)=\ell$ implies, by continuity, that $ \Psi(\ell)>\ell$ for all $\ell\in \mathbb{ R}$. Fix $M>0$ and define the truncated version of $ \Psi$, $ \Psi^{ M}(\ell):= \min ( \Psi(\ell), M)$. Then, $ \Psi^{ M} \leq \Psi$, $ \Psi^{ M}$ is bounded and $ \ell=M$ is the unique fixed-point of $ \Psi^{ M}$. A technical point here: as it is, $ \Psi^{ M}$ does not match the hypotheses of Theorem~\ref{th:londen}, as $ \Psi^{ M}$ is not smooth. In this sense, replace $ \Psi^{ M}(\ell)= \frac{ \Psi(\ell) +M - \left\vert \Psi(\ell)-M \right\vert}{ 2}$ by a smooth version
\begin{equation*}
\Psi^{ M}_{ k}(\ell)= \frac{ \Psi(\ell) +M - \frac{ 1}{ k}\log \left(2 \cosh \left(k \left(\Psi(\ell)-M\right)\right)\right)}{ 2},\ k\geq1.
\end{equation*}
It is easy to see that $ \Psi_{ k}^{ M}$ is infinitely differentiable with $ \Psi_{ k}^{ M}(\ell)\to \Psi^{ M}(\ell)$ as $k\to\infty$. Moreover, $ \Psi_{ k}^{ M}\leq \Psi^{ M}$ and for large $k\geq1$, $ \Psi_{ k}^{ M}$ has a unique fixed point $\ell_{ k,M}$ such that $\ell_{ k,M} \xrightarrow[ k\to +\infty]{}M$. Fix for the rest of the proof some sufficiently large $k\geq1$ such that $ \ell_{ k, M}\geq \frac{ M}{ 2}$. Then, if one denotes by $ \lambda^{ M}_{ k}$ the solution to \eqref{eq:conv_gen_lambda} driven by $ \Psi^{ M}_{ k}\leq M$, $ \lambda^{ M}_{ k}$ is then bounded and we obtain by Theorem~\ref{th:londen} that $ \lambda^{ M}_{k, t} \xrightarrow[ t\to\infty]{} \ell_{ k, M}$, unique fixed-point of $ \Psi_{ k}^{ M}$. Since $ \Psi_{ k}^{ M} \leq \Psi$, applying now Proposition~\ref{prop:comparison}, we have for all $t\geq0$, $ \lambda^{ M}_{ k, t}\leq \lambda_{ t}$. Taking  $ \liminf_{ t\to\infty}$ on both sides of the previous inequality gives $ \frac{ M}{ 2}\leq \ell_{ k, M}\leq \liminf_{ t\to\infty} \lambda_{ t}$. Since this is true for all $M>0$, we have $ \liminf_{ t\to\infty} \lambda_{ t}=+\infty$ and  the result follows.
\end{proof}

\subsection{The case of a unique fixed-point}

We make in this paragraph the following assumption
\begin{assumption}
\label{ass:examples}
The kernel $ \Phi$ is nonnegative, of class $ \mathcal{ C}^{ 2}$, such that $ \Phi^{ \prime}>0$ with $ \left\Vert \Phi^{ \prime \prime} \right\Vert_{ \infty}<\infty$ and $ \Phi(0)>0$. Suppose that $h$ is nonnegative, bounded continuous and integrable on $[0, +\infty)$. Suppose that $\xi$ is $C^{1}$ on $[0, +\infty)$ with $ \xi_{ t} \xrightarrow[ t\to\infty]{}0$. Suppose also that $ \Phi$ has a unique fixed-point $\ell>0$ solution to \eqref{eq:fixed_point_lambda}.
\end{assumption}
Under the present assumption, we necessarily have
$
\left\Vert h \right\Vert_{ 1}\Phi^{ \prime}(\left\Vert h \right\Vert_{ 1}\ell)\leq 1$ (recall Remark~\ref{rem:empty}).
We specify here the previous results to the following cases:
\subsubsection{Case $1$:}
Suppose here that 
\begin{equation*}
\Phi \left(\left\Vert h \right\Vert_{ 1} \ell_{ 0}\right)> \ell_{ 0} \text{ for all }  \ell_{ 0}\neq \ell.
\end{equation*}
In this case, $ \Phi$ is unbounded and we necessarily have $ \left\Vert h \right\Vert_{ 1}\Phi^{ \prime}(\left\Vert h \right\Vert_{ 1}\ell)=1$: the fixed-point $\ell$ is critical. Moreover, the following holds:
\begin{enumerate}[label=(\alph*)]
\item The equilibrium source term $ \xi^{ eq, \ell}$ is \emph{globally unstable from above}, in the sense of Proposition~\ref{prop:generic_unstable_stable}: namely, for any $ \ell_{ 0}>\ell$ (and not only on a neighborhood $ \left(\ell, \ell+ \delta\right)$), the solution $ \lambda_{ t}$ starting from the source term $ \xi^{ \ell_{ 0}}$ defined by \eqref{eq:xi_diverges} is increasing and cannot be bounded (otherwise it would converge to the unique fixed-point $\ell$ by Theorem~\ref{th:londen}, which is not possible). Therefore, $ \lambda_{ t} \xrightarrow[ t\to \infty]{}+\infty$. Hence, we are in a case where \eqref{eq:conv_gen_lambda} admits  unbounded solutions that diverge to $+\infty$ as $t\to \infty$.
\item The equilibrium source term $ \xi^{ eq, \ell}$ is \emph{globally stable from below}, in the sense of Proposition~\ref{prop:generic_unstable_stable}. Namely, for any $ \ell_{ 0} \in [0, \ell]$ (and not only in a neighborhood $(\ell- \delta, \ell]$), the solution $ \lambda$ to \eqref{eq:conv_gen_lambda} starting from the source term $ \xi $ given by \eqref{eq:xi_diverges} converges to $\ell$ as $t\to\infty$. This is in particular the case of the solution starting from the empty source term $ \xi^{ \emptyset}$.
\end{enumerate}

\subsubsection{Case 2:}
Suppose now that 
\begin{equation}
\label{eq:case2}
\Phi \left( \left\Vert h \right\Vert_{ 1} \ell_{ 0}\right)> \ell_{ 0} \text{ for all } \ell_{ 0}<\ell \text{ and } \Phi \left(\left\Vert h \right\Vert_{ 1} \ell_{ 0}\right)< \ell_{ 0} \text{ for all } \ell_{ 0}> \ell.
\end{equation}

In this case, the equilibrium source term $ \xi^{ eq, \ell}$ is \emph{globally stable from below and from above}, in the sense of Proposition~\ref{prop:generic_unstable_stable}. Namely, for any $ \ell_{ 0}\in [0, +\infty)$ (and not only in a neighborhood of $\ell$), the solution $ \lambda$ to \eqref{eq:conv_gen_lambda} starting from the source term $ \xi$ given by \eqref{eq:xi_diverges} converges to $\ell$ as $t\to\infty$. This is in particular the case of the solution starting from the empty source term $ \xi^{ \emptyset}$. Moreover, if one assumes furthermore that $ \left\Vert h \right\Vert_{ 1}\Phi^{ \prime}( \left\Vert h \right\Vert_{ 1}\ell)<1$, then strong local stability holds for $ \xi^{ eq, \ell}$ in the sense of Theorem~\ref{th:stab_stationary} and the rates of convergence to $\ell$ established in Theorem~\ref{th:conv_rate} are valid for any source term $ \xi$ such that $ \lambda_{ t}^{ \xi}$ converges to $\ell$. 
Finally, 
\begin{enumerate}[label=(\alph*)]
\item If $ \limsup_{ x\to \infty} \frac{ \Phi\left(x \right)}{ x} \Vert h \Vert_1<1$, then by Proposition~\ref{prop:global_subcrit}, any solution $ \lambda$ to \eqref{eq:conv_gen_lambda} is bounded and, in case $h$ is nonincreasing, Londen's Theorem applies and $ \lambda_{ t}^{ \xi}$ converges to the unique fixed-point $ \ell$, \emph{whatever the source term may be}.
\item However, Proposition~\ref{prop:example_infty} below gives an example of $ \Phi$ with $ \limsup_{ x\to\infty} \frac{ \Phi\left(x \right)}{ x}\Vert h \Vert_1=1$ with a unique subcritical fixed-point $\ell$, verifying \eqref{eq:case2} and such that \eqref{eq:conv_gen_lambda} admits a diverging solution.
\end{enumerate}

\subsection{The sigmoid case}
\label{sec:sigmoid}
An example of $ \Phi$ that is commonly met is the  sigmoid kernel, see Figure~\ref{fig:sigmoid}.
\begin{assumption}
\label{ass:sigmoid}
Suppose that $ \Phi$ is sigmoid, \textit{i.e.,} in addition to Assumption~\ref{ass:examples} that $ \Phi$ possesses three fixed-points $\ell_{ s}^{ -} < \ell_{ u} < \ell_{ s}^{ +}$ such that $ \left\Vert h \right\Vert_{ 1}\Phi^{ \prime} \left( \left\Vert h \right\Vert_{ 1}\ell_{ s}^{ \pm}\right)< 1$ and $ \left\Vert h \right\Vert_{ 1}\Phi^{ \prime} \left(\left\Vert h \right\Vert_{ 1}\ell_{ u}\right)>1$. Suppose also that $ \ell \mapsto \Phi \left(\left\Vert h \right\Vert_{ 1}\ell\right)$ is convex-concave around $\ell_{ u}$, that is convex on $[0, \ell_{ u}]$ and concave on $[\ell_{ u}, +\infty)$.
\end{assumption}
The NRE \eqref{eq:conv_gen_lambda} driven by $\Phi$ satisfying Assumption~\ref{ass:sigmoid} is sometimes called in the literature \emph{the multistable NRE} (see \emph{e.g.} \cite{Heesen2021,lucon:hal-05185413}). Although this denomination makes perfect sense in case $h$ is exponential (as one retrieves then the usual notion of stability for ODEs, see Remark~\ref{rem:exponential} below), this notion of multistability requires clarification for general kernels $h$. One of the main motivations of the present paper has been precisely to give a meaning of this notion of multistability as precise as possible. Namely, under Assumption~\ref{ass:sigmoid}, the following holds :
\begin{enumerate}[label=(\alph*)]
\item Both equilibrium source terms $ \xi^{ eq, \ell_{ s}^{ -}}$ and $ \xi^{ eq, \ell_{ s}^{ +}}$ are strongly stable in the sense of Theorem~\ref{th:stab_stationary}. 
\item The equilibrium source term $ \xi^{ eq, \ell_{ u}}$ is strongly unstable in the sense of Theorem~\ref{th:supercritical_unstable2}.
\item If one restrict to source terms of the form $ \xi^{ \ell_{ 0}}_{ t} = \ell_{ 0} \int_{ t}^{+\infty} h(u) {\rm d}u$ (recall \eqref{eq:xi_mono}), we know explicitly the basins of attraction of both the stables points $\ell_{ s}^{ -}$ and $\ell_{ s}^{ +}$ (recall Remark~\ref{rem:monotone_empty_case}): if $ \ell_{ 0}< \ell_{ u}$, then the solution $ \lambda$ to \eqref{eq:conv_gen_lambda} starting from $ \xi^{ \ell_{ 0}}$ converges to $\ell_{ s}^{ -}$ and if $\ell_{ 0}> \ell_{ u}$, $ \lambda$ converges to $ \ell_{ s}^{ +}$, see Figure~\ref{fig:sigmoid}. In particular, $ \lambda^{ \emptyset}$ starting from the empty source term $ \xi^{ \emptyset}$ converges to $ \ell_{ s}^{ -}$. 
\item For any solution $ \lambda$ that converges to either $ \ell_{ s}^{ -}$ or $ \ell_{ s}^{ +}$, rates of convergence found in Theorem~\ref{th:conv_rate} apply.
\end{enumerate}
\begin{figure}[ht]
\centering
\subfloat[A sigmoid kernel $\Phi$.]{\includegraphics[width=0.47\textwidth]{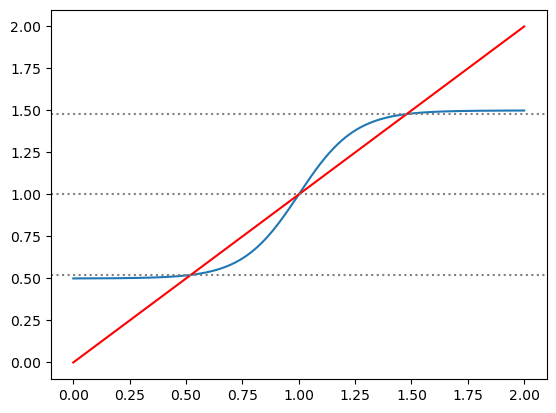}
\label{subfig:sigmoid}}
\quad\subfloat[Trajectories of solutions to \eqref{eq:conv_gen_lambda}.]{\includegraphics[width=0.47\textwidth]{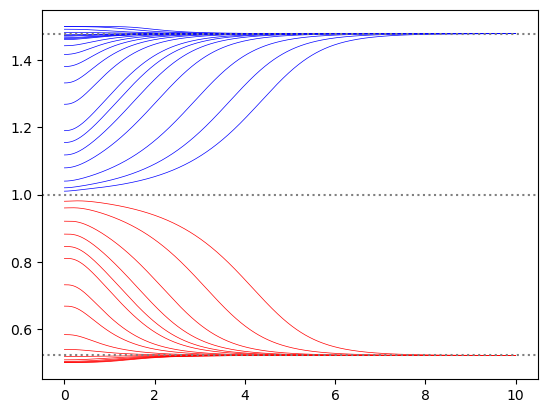}
\label{subfig:basins}}
\caption{{\footnotesize The case of the sigmoid kernel $\Phi(x)=0.5+ \left(1+\exp(-8(x-1))\right)^{-1}$ (Fig.~\ref{subfig:sigmoid}) satisfying Assumption~\ref{ass:sigmoid} with $\ell_{ s}^{ -}\approx0.5212 $, $\ell_u=1$ and $\ell_{ s}^{ +}\approx 
 1.4788$. Here $h$ is the Erlang kernel of order $n=2$ and rate $\alpha=3$ given by Definition~\ref{def:Erlang} below. Fig.~\ref{subfig:basins} shows some trajectories $\lambda$ to \eqref{eq:conv_gen_lambda} starting from $\xi^{\ell}$ given by \eqref{eq:xi_mono} for different values of $\ell$. Red (resp. blue) trajectories, corresponding to values $\ell<\ell_u=1$ (resp. $\ell>\ell_u=1$) all converge to $\ell_s^{-}$ (resp. $\ell_s^{+}$).}}%
\label{fig:sigmoid}%
\end{figure}
\subsection{The Erlang example}
\label{sec:Erlang}
A choice of kernels $h$ that have been extensively studied in the literature \cite{Chevallier2021, Duarte:2016aa}  concerns Erlang kernels:
\begin{definition}
\label{def:Erlang}
Let $ n\geq1$ and $ \alpha>0$, $h$ is the Erlang kernel of order $n$ with rate $ \alpha$ if
\begin{equation}
\label{eq:Erlang}
h_n(t)= \alpha^{ n+1} e^{ - \alpha t} \frac{ t^{ n}}{ n!},\quad \forall t\geq0.
\end{equation}
Note that the renormalization in \eqref{eq:Erlang} has been chosen so that $ \left\Vert h \right\Vert_{ 1}=1$, which simplifies the expression of the critical conditions \eqref{eq:cond_stab_ell_pos}, \eqref{eq:Phi_supercritical} and \eqref{eq:local_crit} into $ \Phi^{ \prime}(\ell)<1$ (resp. $\Phi^{ \prime}(\ell)>1$ and $\Phi^{ \prime}(\ell)=1$).
\end{definition}
The specificity of Erlang kernels is that, unless $n=0$, $h_{ n}$ is not nonincreasing and $h_{ n}(0)=0$. In particular, Londen's Theorem does not apply. However, the hypotheses made in this paper (see \textit{e.g.,} \eqref{eq:cond_h} and Remark~\ref{rem:monotone_empty_case}) are precisely designed to include Definition~\ref{def:Erlang}:  Proposition~\ref{prop:nofixedpoint} excepted (which is the only moment where we apply Londen's Theorem), every result stated in this paper applies to the case of Erlang kernels.

\subsubsection{Erlang cascade}
In the case of Erlang kernels, the particle system $(Z^{1,N}, \ldots, Z^{N, N})$ given by \eqref{eq:Hawkes_N} becomes Markovian \cite{Duarte:2016aa}. In the context of the mean-field limit, this translates into the fact that the nonlinear convolution equation \eqref{eq:conv_gen_X} can be described in terms of a cascade of coupled ODEs. Namely, write for all $k=0, \ldots, n$, 
\begin{equation}
\label{eq:xk_Erlang}
x_{k, t}= \xi_{k, t} + \int_{ 0}^{t} h_{ n-k}(t-s) \Phi (x_{0, s}) {\rm d}s
\end{equation}
where $ x_{0,t}= x_t$ is the solution to \eqref{eq:conv_gen_X} with source term $ \xi_{0, t}:= \xi_t$ and $ \xi_1, \ldots, \xi_n$ are prescribed $C^{ 1}$ functions on $[0, +\infty)$ such that $ \xi_{j,t} \xrightarrow[ t\to\infty]{}0$ and $ \frac{ {\rm d}}{ {\rm d}t} \xi_{j, t} \xrightarrow[ t\to \infty]{}0$. Denote by $c_{ 0}:= \xi_{0, 0}= \xi_{ 0}$ and $c_{ k}:= \xi_{k, 0}$ for $k=1, \ldots, n$, the initial conditions for all $\xi_k$. Using \eqref{eq:Erlang}, observe that $ h_{ n}^{ \prime}(t)= - \alpha h_{ n}(t) + \alpha h_{ n-1}(t)$, $h_{ n}(0)=0$ for $n\geq1$ and $h_{ 0}^{ \prime}(t)= - \alpha h_{ 0}(t)$, $h_{ 0}(0)= \alpha $, then $ \left(x_0, \ldots, x_n\right)$ solves for $n\geq1$ the following system of coupled non autonomous ODEs: 
\begin{equation}
\label{eq:syst_Erlang}
\begin{cases}
\frac{ {\rm d}}{ {\rm d}t}x_{k, t}&=- \alpha x_{k, t} + \alpha x_{k+1, t} + \alpha \xi_{k,t}  + \frac{ {\rm d}}{ {\rm d}t}\xi_{k,t} - \alpha  \xi_{k+1, t},\ k=0, \ldots, n-1\\
 \frac{ {\rm d}}{ {\rm d}t}x_{n, t}&=- \alpha x_{n, t} + \alpha \Phi (x_{0, t}) + \alpha \xi_{n, t}  + \frac{ {\rm d}}{ {\rm d}t}\xi_{n, t},
\end{cases}
\end{equation}
with initial condition $(x_{0, 0}, \ldots, x_{n, 0})= (c_{ 0}, c_{ 1}, \ldots, c_{ n})$. One readily sees that equilibrium points  to \eqref{eq:syst_Erlang} are $\left(\bar x_0, \ldots, \bar x_n\right)= \left(\ell, \ldots, \ell\right)$, where $ \ell= \Phi(\ell)$ solves \eqref{eq:fixed_point_lambda_pos} (recall that $ \left\Vert h \right\Vert_{ 1}=1$ here).
A simple calculation shows that \eqref{eq:syst_Erlang} reduces to the autonomous system
\begin{equation}
\label{eq:syst_Erlang_aut}
\begin{cases}
\frac{ {\rm d}}{ {\rm d}t}x_{k,t}=- \alpha x_{k,t} + \alpha x_{k+1,t},&\ k=0, \ldots, n-1\\
 \frac{ {\rm d}}{ {\rm d}t}x_{n,t}=- \alpha x_{n,t} + \alpha \Phi (x_{0,t}),&
\end{cases}
\end{equation}
if and only if 
\begin{equation*}
\xi_{n-k, t}= \left(\sum_{ j=0}^{ k} \alpha^{ j} c_{ n-k+j} \frac{ t^{ j}}{ j!} \right)e^{ - \alpha t}, t\geq0,\ \text{ for all }k=0, \ldots, n.
\end{equation*}
In other words, the unique solution $(x_0, \ldots, x_n)$ to the system \eqref{eq:syst_Erlang_aut} with initial condition $ \left(c_{ 0}, \ldots, c_{ n}\right)$ has the representation \eqref{eq:xk_Erlang}. We have the following particular cases:
\begin{itemize}
\item Taking $ c_{ 0}= c_{ 1}= \ldots = c_{ n}$ correspond to the source term $ \xi_{ t}^{ c_{ 0}}= c_{ 0} \int_{ t}^{+\infty} h_{ n}(u) {\rm d}u$ given by \eqref{eq:xi_mono} for the kernel $h_{ n}$. In particular, taking $c_{ 0}= c_{ 1}= \ldots= c_{ n}= 0$ corresponds to the empty source term $ \xi^{ \emptyset}$ given by \eqref{eq:empty_xi}.
\item More generally, for any $L\in\{1,\ldots,n-1\}$, taking $c_{ 0}= c_{ 1}= \ldots =c_{ L}\neq 0$ and $c_{ L+1}= \ldots= c_{ n}=0$ corresponds to the source term given by \eqref{eq:xi_mono} 
\begin{equation}
\label{eq:xi_L}
    \xi_{L, t}^{c_{ 0}}:= c_{ 0} \int_{ t}^{+\infty} h_{ L}(u) {\rm d}u,\ t\geq 0,
\end{equation} for the Erlang kernel  $h_{ L}$ of lower order $L<n$.
\end{itemize}

\subsubsection{Stability w.r.t. polynomial perturbations in the subcritical case}
The Jacobian matrix of \eqref{eq:syst_Erlang_aut} around some equilibrium point $(\ell, \ldots, \ell)$ is
\begin{equation}
\label{eq:Jl}
J_{ \ell}:= \alpha \begin{pmatrix}
 -1 & 1 & 0 & \cdots  & 0\\
 0 & -1 & 1 & \ddots &  \vdots \\
 \vdots & \ddots & \ddots & \ddots  & 0\\
 0 &  & \ddots & \ddots & 1\\
   \Phi^{ \prime}(\ell) & 0  & \cdots & 0 & -1
\end{pmatrix}.
\end{equation}
The eigenvalues of $J_{ \ell}$ are $ \lambda_{ k}= \alpha \left(\Phi^{ \prime}(\ell)^{ \frac{ 1}{ n+1}} e^{ \frac{ 2i k \pi}{ n+1}} -1\right)$, $k=0,\ldots, n$. Therefore, $ \Re \lambda_{ k}<0$
for all $ k=0, \ldots, n$ as long as we are in the critical case $ \Phi^{\prime}(\ell)<1$. It allows to deduce that the equilibrium source term $ \xi^{ eq, \ell}$ is locally stable w.r.t. to polynomial perturbations of the form 
\begin{equation}
\label{eq:xi0_Erlang}
\xi_t=\xi_{0, t}= \left(\sum_{ j=0}^{ n} \alpha^{ j} c_{j} \frac{ t^{ j}}{ j!} \right)e^{ - \alpha t}, t\geq0
\end{equation} 
as long as the coefficients $ \left(c_{ 0}, \ldots, c_{ n}\right)$ are sufficiently close to $(\ell, \ldots, \ell)$. We actually prove far more in Theorems~\ref{th:conv_rate} and ~\ref{th:stab_stationary}, even in this simple Erlang case: $ \xi^{ eq, \ell}$ is actually stable w.r.t. to \emph{any perturbations} that is sufficiently close to $ \xi^{ eq, \ell}$, not only w.r.t. polynomial perturbations of the form \eqref{eq:xi0_Erlang}. Note that the previous local stability argument around any $(\ell, \ldots, \ell)$ does not say anything about the convergence of the solution starting from $(0, \ldots, 0)$ (this corresponds to the empty source $ \xi^{ \emptyset}$) nor from $(\ell, \ldots, \ell, 0, \ldots, 0)$ (this corresponds to $\xi^{\ell}_L:= \ell \int_{ t}^{+\infty} h_{ L}(u) {\rm d}u $ given by \eqref{eq:xi_L} for an order $L< n$), as both initial conditions are generically not close to $(\ell, \ldots, \ell)$. Whereas the first case of $\xi^{\emptyset}$ relies on a very specific monotonicity argument (recall Theorem~\ref{th:conv_empty}), the second case of $\xi_L$ is out of the scope of the present paper. We illustrate in Figure~\ref{fig:Erlang_diff} the difficulty of dealing with this second case in general.

\begin{figure}[ht]
\centering
\includegraphics[width=0.5\textwidth]{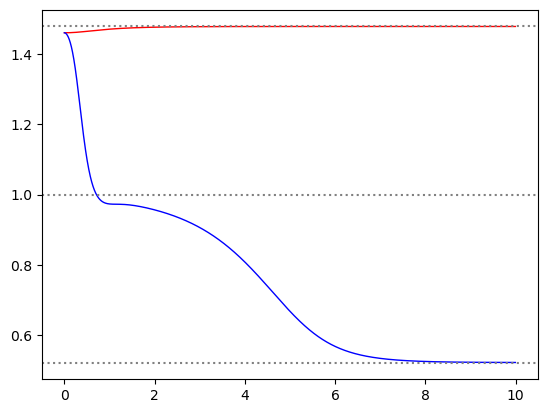}
\caption{{\footnotesize Two different trajectories of \eqref{eq:conv_gen_lambda} in case $\Phi$ is given as in Fig.~\ref{subfig:sigmoid} and $h$ is Erlang with order $n=2$ and $\alpha=3$, for different choices of source terms $\xi$: the red trajectory corresponds to the choice of $\xi_L^{\ell}$ given by \eqref{eq:xi_L} for $L=n=2$ and the blue one to $\xi_L^{\ell}$ for $L=1<n$. Both correspond to $\ell=1.4$. The red trajectory converges to the upper stable point $\ell_s^{+}$ in accordance to Fig.~\ref{fig:sigmoid}. On the contrary, choosing some lower order source term $\xi_L$ (that obviously does not satisfy the proximity condition \eqref{eq:etat_small} of Theorem~\ref{th:stab_stationary}) leads to the blue trajectory that crosses the unstable point $\ell_u=1$ and  converges to the lower stable point $\ell_s^{-}$. Identifying precisely the basin of attraction around the equilibrium source term $\xi^{eq, \ell_s^{+}}$ is certainly a difficult task.}}%
\label{fig:Erlang_diff}%
\end{figure}

\begin{remark}[The exponential case]
\label{rem:exponential}
It is worth mentioning the specific case $n=0$. The kernel $h$ is then exponential: $h(u)=h_{ 0}(u)= \alpha e^{ - \alpha u}$, $u\geq0$ and the system \eqref{eq:syst_Erlang} reduces to the single ODE 
\begin{equation}
\label{eq:syst_n0}
\frac{ {\rm d}}{ {\rm d}t}x_{t}=- \alpha x_{t} + \alpha \Phi (x_{t}) + \alpha \xi_{t}  + \frac{ {\rm d}}{ {\rm d}t}\xi_{t},\ t\geq 0,
\end{equation} with $ x_{0}= c_{ 0}$. Again, this system becomes autonomous 
\begin{equation}
\label{eq:syst_n0_aut}
\frac{ {\rm d}}{ {\rm d}t}x_{t}=- \alpha x_{t} + \alpha \Phi (x_{t}),\ t\geq 0.
\end{equation}
if and only if 
\begin{equation}
\label{eq:xi_exp_n0}
 \xi_{ t}= \xi_{ 0} e^{ - \alpha t},\ t\geq 0.
\end{equation} 
 The choice of \eqref{eq:xi_exp_n0} is by no means insignificant: this reduction of the dynamics of the NRE \eqref{eq:conv_gen_lambda} into the autonomous ODE \eqref{eq:syst_n0_aut} has been transposed in \cite{CHEVALLIER20191} (and further developed in \cite{agathenerine22,AgatheNerine2025,lucon:hal-05185413}) to a context of Hawkes processes with spatial extension. This has lead to a microscopic interpretation of the well-known \emph{Neural Field Equation} \cite{Amari:1977,Wilson1973}
 \begin{equation}
 \label{eq:NFE}
     \partial_{ t} u_t(x) = - \alpha u_t(x) + \alpha \int_{ \mathcal{ D}} W(x,y) \Phi(u_t(y)) {\rm d}y
 \end{equation}
that is a version of \eqref{eq:syst_n0_aut} with spatial extension. We refer to the aforementioned references for more details and the precise  interpretation of \eqref{eq:NFE}.

\end{remark}

\subsubsection{Stable manifold in the supercritical case}
\label{sec:stable_manifold_Erlang}
Consider the supercritical case \eqref{eq:Phi_supercritical}, that reduces here to $ \Phi^{\prime}(\ell)>1$. It is proven in Theorem~\ref{th:supercritical_unstable2} that the equilibrium solution $\ell$ is then unstable w.r.t. source terms $\xi$ that are uniformly above or below $\xi^{eq, \ell}$, giving rise to solutions $\lambda$ that remain either above or below $\ell$. The possibility nonetheless remains about the existence of generic perturbations of $\xi^{eq, \ell}$ that would lead to solutions $\lambda$ that converge to $\ell$ while oscillating around $\ell$. The possibility of oscillations around its limit has been already discussed by Feller \cite{feller1941} in the linear case (see also Brauer \cite{Brauer1975}). We illustrate this phenomenon in the simple case of the Erlang kernel of order $n=2$ but this example can easily be generalized to higher orders. Take $n=2$ and $\alpha=1$ for simplicity. The matrix in \eqref{eq:Jl} becomes $J_\ell= \begin{pmatrix}
 -1 & 1 & 0\\
 0 & -1 & 1  \\
 \tau_0 & 0  & -1
\end{pmatrix}$ for $\tau_0:= \Phi^{\prime}(\ell)>1$. The eigenvalues of $J_\ell$ are then $ \left(\lambda_0, \lambda_1, \lambda_2\right):= \left(\tau_0^{\frac{1}{3}}-1, \mu + i \nu, \mu - i \nu\right)$, with $\lambda_0>0$, $ \mu= -1 - \frac{\tau_0}{2}<0$ and $\nu= \frac{\sqrt{3}\tau_0}{2}$. Denote $(v, w, \bar w)$ a basis of eigenvectors associated to $(\lambda_0, \lambda_1, \lambda_2)$. In particular, one can choose $w=w_1 + i w_2$ with $w_1:=\begin{pmatrix}1\\ \tau_0^{\frac{1}{3}}\\ -2 \tau_0^{\frac{2}{3}}\end{pmatrix}$ and $w_2:= \begin{pmatrix}-\sqrt{3}\\ \sqrt{3}\tau_0^{\frac{1}{3}}\\ 0\end{pmatrix}$ and a direct computation shows that any solution to the linear system $Y^{\prime}(t) =J_\ell Y(t)$ may be written as
\[Y(t)= a e^{\lambda_0 t} v + e^{\mu t} \left(\cos\left(\nu t\right)\left(b w_1- c w_2 \right) + \sin\left(\nu t\right)\left(-b w_1+ c w_2 \right)\right),\ t\geq 0,\]
where $a,b,c$ are arbitrary constants. The stable subspace corresponds to the choice $a=0$. Take for simplicity $b=\varepsilon$ and $c=0$ for some small $\varepsilon>0$. Hence, 
\begin{equation}
\label{eq:Yeps}
Y_\varepsilon(t)= \varepsilon e^{\mu t} \left(\cos\left(\nu t\right) - \sin\left(\nu t\right)\right) w_1,\ t\geq 0,    
\end{equation}
is a solution to $Y^{\prime}(t) =J_\ell Y(t)$ such that $Y_\varepsilon(t) \xrightarrow[t\to\infty]{}0$, with initial condition $Y_\varepsilon(0)= \varepsilon w_1$. Then, by Hartman-Grobman Theorem, there exists some homeomorphism $g$ from a neighborhood of $0$ onto a neighborhood of the fixed point $\left(\ell,\ell, \ell\right)$ with $g(0)= \left(\ell,\ell, \ell\right)$ such that, for small $\varepsilon>0$, $X_\varepsilon(t):= g \left(Y_\varepsilon(t)\right)$ is a solution to the nonlinear system \eqref{eq:syst_Erlang_aut} that verifies $X_\varepsilon(t) \xrightarrow[t\to\infty]{}\left(\ell,\ell, \ell\right)$.  This means that for such a fixed $\varepsilon>0$, for the choice of $\left(c_0, c_1, c_2\right):=g\left(Y_\varepsilon(0)\right)= h(\varepsilon w_1)$, the source term $\xi= \xi_0$ given by \eqref{eq:xi0_Erlang} gives rise to a solution $ \lambda^\xi$ to \eqref{eq:conv_gen_lambda} such that $\lambda_t^\xi \xrightarrow[t\to\infty]{}\ell$ under the supercritical condition \eqref{eq:Phi_supercritical}. In view of \eqref{eq:Yeps}, this solution is indeed oscillating around $\ell$. Characterizing the stable manifold around $\ell$ in the supercritical case for general kernels $h$ is certainly a difficult task, but this simple example shows that Theorem~\ref{th:supercritical_unstable2} cannot be significantly improved.

\section{Proofs of the well-posedness and regularity results}
\label{sec:proofs_WPregular}
\subsection{Well-posedness of the mean-field limit: Proof of Proposition \ref{prop:WP_solC1}}
\label{sec:proof_WP_MF}

We first establish uniqueness of a locally bounded solution to \eqref{eq:conv_gen_lambda}. Suppose there exist two solutions $\lambda$ and $\tilde\lambda$ to \eqref{eq:conv_gen_lambda}, using that $\Phi$ is Lipschitz continuous we get
\begin{align*}
v_{t}:&= \left\vert \lambda_{t}- \tilde\lambda_{t} \right\vert\le |\Phi|_{Lip}\int_{0}^{t} \left\vert h(t-s) \right\vert v_{ s} {\rm d}s,
\end{align*} uniqueness then follows from Lemma~\ref{lem:Gronwall_eps}. For the existence part, we follow a standard Picard iteration method. Set $\lambda^{0}\equiv 0$ and $$ \lambda^{n+1}_{t}:= \Phi\left(\xi_{t}+\int_0^{t}h(t-u)\lambda^{n}_{u} {\rm d}u\right)\rmd s.$$ Note that for all $n\ge 0$, it holds that 
\begin{align*}
\lambda^{n+1}_{t}&\le \Phi(0)+ |\Phi|_{Lip}\left( \left\vert \xi_{t} \right\vert+\int_{0}^{t} \left\vert h(t-s) \right\vert \lambda^{n}_{s} {\rm d}s \right).
\end{align*}As $t\mapsto \xi_{t}$ is locally bounded, using Lemma~\ref{lem:Gronwall_eps}, we get  $\sup_{t\in[0,T]}\sup_{n\ge 0}\lambda^{n+1}_{t}\le C_{T}$ implying that $\lambda^{n+1}$ is locally bounded. Then, setting $\delta^{n+1}_{t}:= \left\vert \lambda^{n+1}_{t}-{\lambda}^{n}_{t} \right\vert$ one gets using that $\Phi$ is Lipschitz continuous $\delta^{n+1}_{t}\le |\Phi|_{Lip}\int_{0}^{t} \left\vert h(t-u) \right\vert\delta^{n}_{u}du$, so that Lemma~\ref{lem:Gronwall_eps} implies that $$\sup_{t\in[0,T]} \sum_{n\ge 0}\delta^{n}_{t}\le C_{T}$$ and allows to conclude that there exists $\lambda$ locally bounded such that $\lim_{n\to\infty} \sup_{ t\in [0, T]}\left\vert \lambda_{t}-\lambda^{n}_{t} \right\vert=0$. If $ \xi$ is continuous, we see by induction that $ \lambda^{ n}$ is continuous, and hence by local uniform convergence, so is $ \lambda$.

The fact that the solution is $C^1$ is a direct consequence of the following result.
\begin{proposition}
\label{prop:C1}
Suppose that $ \Phi$ is nonnegative Lipschitz continuous, of class $C^{1}$ with $\Phi^\prime$ Lipschitz continuous. Suppose that $h$ is continuous on $[0, +\infty)$ and that $\xi$ is $C^{1}$ on $[0, +\infty)$. Let $ \lambda= \lambda^{ \xi}$ be the solution to \eqref{eq:conv_gen_lambda} with source term $ \xi$. Then, $ \lambda$ is of class $ \mathcal{ C}^{ 1}$ on $[0, +\infty)$ and its derivative $ \lambda^{ \prime}$ verifies
\begin{equation}
\label{eq:derivative_lambda}
\lambda'_{t}=\left(\rho(h, \xi, t)+\int_{0}^{t}h(t-u)\lambda'_{u}du\right)\Phi'\left(\xi_{t}+\int_{0}^{t}h(t-u)\lambda_{u}du\right)
\end{equation} 
where the source term in \eqref{eq:derivative_lambda} is given by
\begin{equation}
\label{eq:rho}
\rho(h, \xi, t):=\xi'_{t}+h(t) \Phi \left(\xi_{ 0}\right),\quad t\geq0.
\end{equation}
\end{proposition}
\subsection{Proof of Proposition~\ref{prop:C1}}
We proceed in several steps:

\medskip
\noindent
\textit{Step 1:  suppose first that $h$ is $C^{ 1}$.} Then, from Proposition~\ref{prop:WP_solC1} $ \lambda$ is continuous and in this case $ t \mapsto \int_{ 0}^{t} h(t-s) \lambda_{ s} {\rm d}s$ is $ C^{ 1}$ (with derivative given by $ t \mapsto \int_{ 0}^{t} h^{ \prime}(t-s) \lambda_{ s} {\rm d}s+ h(0) \lambda_{ t}$). Coming back to \eqref{eq:conv_gen_lambda}, we deduce immediately that $ \lambda$ is $C^{ 1}$, since $ \Phi$ and $ \xi$ are $C^{ 1}$ by assumption. The identity \eqref{eq:derivative_lambda} is then an immediate application of the chain rule. The rest of the proof is now devoted to prove the same result under the assumption that $h$ is continuous only.

\medskip
\noindent \textit{Step 2:  A $C^{1}$ mollification of $h$.} Let $ \left(\chi_{ \varepsilon}\right)_{ \varepsilon>0}$ a $C^{ 1}$ approximation of the identity, \textit{i.e.,} $ \chi_{ \varepsilon}(s)= \varepsilon^{ -1}\chi \left(s/ \varepsilon\right)$ for some $C^{ 1}$ nonnegative function $ \chi$ satisfying $ \int_{ 0}^{+\infty} \chi(u) {\rm d}u=1$. Introduce $h_{ \varepsilon}= h\ast \chi_{ \varepsilon}$ and denote by $ \lambda_{ \varepsilon}$ the solution to \eqref{eq:conv_gen_lambda} with source term $ \xi$ driven by the kernel $ h_{ \varepsilon}$. By Step 1, $ \lambda_{ \varepsilon}$ is $C^{ 1}$ and we have
\begin{equation}
\label{eq:derivative_lambda_eps}
\lambda^{ \prime}_{ \varepsilon, t}=\left(\rho(h_{ \varepsilon}, \xi, t)+\int_{0}^{t}h_{ \varepsilon}(t-u)\lambda^{ \prime}_{ \varepsilon, u}du\right)\Phi'\left(\xi_{t}+\int_{0}^{t}h_{ \varepsilon}(t-u)\lambda_{ \varepsilon, u}du\right).
\end{equation} 

\medskip
\noindent \textit{Step 3: Uniform control in $ \varepsilon$ of $ \lambda_{ \varepsilon}$ and $ \lambda_{ \varepsilon}^{ \prime}$.} Remark first that, by Lipschitz continuity of $ \Phi$, for all $T>0$, $ \varepsilon>0$, $t\in [0,T]$
\begin{align*}
\lambda_{ \varepsilon, t} \leq \Phi(0) + \left\vert \Phi \right\vert_{ Lip} \left( \left\vert \xi_{ t} \right\vert+ \int_{ 0}^{t} \left\vert h_{ \varepsilon}(t-s)  \right\vert\lambda_{ \varepsilon, s} {\rm d}s\right).
\end{align*}
We are therefore in position to apply Lemma~\ref{lem:Gronwall_eps}: there exist some $ \varepsilon_{ 0}>0$ and some constant $C_{ T}>0$ independent of $ \varepsilon>0$ such that 
\begin{equation}
\label{eq:aprioribound_lambdaeps}
 \sup_{ t\in [0, T]}\lambda_{ \varepsilon, t} \leq C_{ T},\ \varepsilon\in \left(0, \varepsilon_{ 0}\right).
\end{equation}
We proceed similarly with $ \lambda_{ \varepsilon}^{ \prime}$: starting from \eqref{eq:derivative_lambda_eps}, we have
\begin{align}
\left\vert \lambda_{ \varepsilon, t}^{ \prime} \right\vert\leq \left\vert \Phi^{ \prime} \left(\xi_{t}+\int_{0}^{t}h_{ \varepsilon}(t-u)\lambda_{ \varepsilon, u}du\right) \right\vert \left( \left\vert \xi_{ t}^{ \prime} \right\vert+ \left\vert h_{ \varepsilon}(t) \right\vert \Phi( \xi_{ 0}) + \int_{ 0}^{t} \left\vert h_{ \varepsilon}(t-u) \right\vert \left\vert \lambda_{ \varepsilon, u}^{ \prime} \right\vert {\rm d}u\right). \label{aux:bound_lambdaprimeeps}
\end{align}
Since $ \Phi^{ \prime}$ is Lipschitz, there is some constant $C>0$ such that
\begin{align*}
\left\vert \Phi^{ \prime} \left(\xi_{t}+\int_{0}^{t}h_{ \varepsilon}(t-u)\lambda_{ \varepsilon, u}du\right) \right\vert \leq C \left(1+ \left\vert \xi_{ t} \right\vert+ \int_{ 0}^{t} \left\vert h_{ \varepsilon}(t-u) \right\vert \lambda_{ \varepsilon, u}  {\rm d}u\right)\leq C_{ T}
\end{align*}
for some $C_{ T}>0$ independent of $ \varepsilon$. Here, we have used the local boundedness of $ \xi$, estimate \eqref{eq:aprioribound_lambdaeps} and the fact that $ \int_{ 0}^{t} \left\vert h_{ \varepsilon}(u) \right\vert {\rm d}u \xrightarrow[ \varepsilon\to 0]{} \int_{ 0}^{t} \left\vert h(u) \right\vert {\rm d}u\leq \int_{ 0}^{T} \left\vert h(u) \right\vert {\rm d}u<+\infty$. Moreover, since $h$ is continuous, $ h_{ \varepsilon}$ converges uniformly to $h$ as $ \varepsilon\to 0$ on $[0, T]$ so that $ \sup_{ t\geq0}\left\vert h_{ \varepsilon}(t) \right\vert\leq \sup_{ t\in[0, T]} \left\vert h(t) \right\vert+1$ 
uniformly in $ \varepsilon>0$. By the local boundedness $ \xi^{ \prime}$, we see from \eqref{aux:bound_lambdaprimeeps} that there exist some $C_{ 1, T}, C_{ 2, T}>0$ independent of $ \varepsilon>0$ such that
\begin{equation*}
\left\vert \lambda_{ \varepsilon, t}^{ \prime} \right\vert\leq C_{ 1,T} + C_{ 2, T} \int_{ 0}^{t} \left\vert h_{ \varepsilon}(t-u) \right\vert \left\vert \lambda_{ \varepsilon, u}^{ \prime} \right\vert {\rm d}u.
\end{equation*}
We conclude again from Lemma~\ref{lem:Gronwall_eps} that there is some $C_{ T}>0$ independent of $ \varepsilon>0$ such that
\begin{equation}
\label{eq:aprioribound_lambdaprimeeps}
\sup_{ t\in [0, T]} \left\vert \lambda_{ \varepsilon, t}^{ \prime} \right\vert \leq C_{ T}.
\end{equation}
\medskip
\noindent \textit{Step 4: Uniform local convergence of $ \lambda_{ \varepsilon}$ to $ \lambda$ as $ \varepsilon\to 0$.} Again by the Lipschitz continuity of $ \Phi$, for $t\in [0, T]$
\begin{align*}
\left\vert \lambda_{ \varepsilon,t} - \lambda_{ t}\right\vert \leq \left\vert \Phi \right\vert_{ Lip} \int_{ 0}^{t} \left\vert h_{ \varepsilon}(t-s)- h(t-s) \right\vert \lambda_{ \varepsilon, s} {\rm d}s + \left\vert \Phi \right\vert_{ Lip} \int_{ 0}^{t} \left\vert h(t-s) \right\vert \left\vert \lambda_{ \varepsilon, s} - \lambda_{ s}\right\vert {\rm d}s.
\end{align*}
By \eqref{eq:aprioribound_lambdaeps}, the first term above may be bounded by $C_{ T} \int_{ 0}^{t} \left\vert h_{ \varepsilon}(u) - h(u)\right\vert {\rm d}u \xrightarrow[ \varepsilon\to 0]{}0$. Therefore, we conclude from Lemma~\ref{lem:Gronwall_eps} that
\begin{equation}
\label{eq:lambdaeps_conv_lambda}
\sup_{ t\in [0, T]} \left\vert \lambda_{ \varepsilon, t} - \lambda_{ t}\right\vert \xrightarrow[ \varepsilon\to 0]{}0.
\end{equation}
\medskip
\noindent \textit{Step 5: The family $ \left(\lambda_{ \varepsilon}^{ \prime}\right)$ is Cauchy  for the uniform norm on $[0, T]$, $\forall T>0$.} Take $ \varepsilon_{ 1}, \varepsilon_{ 2}>0$, then for $t\in [0, T]$, it holds
\begin{align}
\label{eq:bound_lambdas_AB}
\left\vert \lambda_{ \varepsilon_{ 1}, t}^{ \prime} - \lambda_{ \varepsilon_{ 2}, t}^{ \prime} \right\vert \leq A_{ t} + B_{ t},
\end{align}
where we give the definitions of $A_{ t}$ and $B_{ t}$ below. Firstly,
\begin{multline*}
A_{ t}:= \left\vert\Phi^{ \prime} \left(\xi_{ t} + \int_{ 0}^{t} h_{ \varepsilon_{ 1}}(t-u) \lambda_{ \varepsilon_{ 1}, u} {\rm d}u\right) - \Phi^{ \prime} \left(\xi_{ t} + \int_{ 0}^{t} h_{ \varepsilon_{ 2}}(t-u) \lambda_{ \varepsilon_{ 2}, u} {\rm d}u\right) \right\vert \\
\times \left\vert\xi_{ t}^{ \prime} + h_{ \varepsilon_{ 1}}(t) \Phi \left(\xi_{ 0}\right)+\int_{0}^{t}h_{ \varepsilon_{ 1}}(t-u)\lambda^{ \prime}_{ \varepsilon_{ 1}, u}du\right\vert.
\end{multline*}
Using the local boundedness of $ \xi^{ \prime}$, the fact that $h_{ \varepsilon}$ converges uniformly on $[0, T]$ to $h$, estimate \eqref{eq:aprioribound_lambdaprimeeps} together with the fact that $ \int_{ 0}^{t} h_{ \varepsilon_{ 1}}(u) {\rm d}u \xrightarrow[ \varepsilon_{ 1}\to 0]{} \int_{ 0}^{t} \left\vert h(u) \right\vert {\rm d}u$, we see that there exists some $C_{ T}>0$ such that
\begin{align*}
 \left\vert\xi_{ t}^{ \prime} + h_{ \varepsilon_{ 1}}(t) \Phi \left(\xi_{ 0}\right)+\int_{0}^{t}h_{ \varepsilon_{ 1}}(t-u)\lambda^{ \prime}_{ \varepsilon_{ 1}, u}du\right\vert \leq C_{ T}.
\end{align*}
Using now the Lipschitz continuity of $ \Phi^{ \prime}$, we deduce that for some $C_{ T}>0$
\begin{align*}
A_{ t}& \leq C_{ T} \left( \int_{ 0}^{t} \left\vert h_{ \varepsilon_{ 1}}(t-u) - h_{ \varepsilon_{ 2}}(t-u)\right\vert \lambda_{ \varepsilon_{ 1}, u} {\rm d}u + \int_{ 0}^{t} \left\vert h_{ \varepsilon_{ 2}}(t-u) \right\vert \left\vert \lambda_{ \varepsilon_{ 1}, u} - \lambda_{ \varepsilon_{ 2}, u}\right\vert {\rm d}u\right).
\end{align*}
Using \eqref{eq:aprioribound_lambdaeps}, that $ \int_{ 0}^{T} \left\vert h_{ \varepsilon_{ 1}}(u)- h_{ \varepsilon_{ 2}}(u) \right\vert {\rm d}u \xrightarrow[ \varepsilon_{ 1}, \varepsilon_{ 2} \to 0]{} 0$ and \eqref{eq:lambdaeps_conv_lambda}, we obtain that $\sup_{ t\in [0, T]} A_{ t} \xrightarrow[ \varepsilon_{ 1}, \varepsilon_{ 2}\to 0]{}0$.

Secondly, set
\begin{multline*}
B_{ t}:= \left\vert \Phi^{ \prime} \left(\xi_{ t} + \int_{ 0}^{t} h_{ \varepsilon_{ 2}}(t-u) \lambda_{ \varepsilon_{ 2}, u} {\rm d}u\right) \right\vert \\ \times\left\vert  \left(h_{ \varepsilon_{ 1}}(t) - h_{ \varepsilon_{ 2}}(t)\right)\Phi \left(\xi_{ 0}\right) + \int_{ 0}^{t} h_{ \varepsilon_{ 1}}(t-u) \lambda_{ \varepsilon_{ 1}, u}^{ \prime} {\rm d}u - \int_{ 0}^{t} h_{ \varepsilon_{ 2}}(t-u) \lambda_{ \varepsilon_{ 2}, u}^{ \prime} {\rm d}u\right\vert.
\end{multline*}
As before, there is a constant $C_{ T}>0$ such that $ \left\vert \Phi^{ \prime} \left(\xi_{ t} + \int_{ 0}^{t} h_{ \varepsilon_{ 2}}(t-u) \lambda_{ \varepsilon_{ 2}, u} {\rm d}u\right) \right\vert\leq C_{ T}$ uniformly in $ \varepsilon_{ 2}>0$. Moreover, since $h$ is continuous, we have $\sup_{ t\in [0, T]} \left\vert h_{ \varepsilon_{ 1}}(t) - h_{ \varepsilon_{ 2}}(t) \right\vert \xrightarrow[ \varepsilon_{ 1}, \varepsilon_{ 2}\to0]{}0$. It remains to write 
\begin{align*}
\left\vert \int_{ 0}^{t} h_{ \varepsilon_{ 1}}(t-u) \lambda_{ \varepsilon_{ 1}, u}^{ \prime} {\rm d}u - \int_{ 0}^{t} h_{ \varepsilon_{ 2}}(t-u) \lambda_{ \varepsilon_{ 2}, u}^{ \prime} {\rm d}u \right\vert &\leq \int_{ 0}^{t} \left\vert h_{ \varepsilon_{ 1}}(t-u) - h_{ \varepsilon_{ 2}}(t-u) \right\vert  \left\vert \lambda_{ \varepsilon_{ 1}, u}^{ \prime} \right\vert {\rm d}u \\&+ \int_{ 0}^{t} \left\vert h_{ \varepsilon_{ 2}}(t-u)  \right\vert  \left\vert \lambda_{ \varepsilon_{ 1}, u}^{ \prime} - \lambda_{ \varepsilon_{ 2}, u}^{ \prime} \right\vert {\rm d}u 
\end{align*}
and to note that the first term in the righthand side goes to $0$ as $ \varepsilon_{ 1}, \varepsilon_{ 2}\to 0$ (uniformly on $[0, T]$), by \eqref{eq:aprioribound_lambdaprimeeps}. Putting everything together back to \eqref{eq:bound_lambdas_AB} and using again Lemma~\ref{lem:Gronwall_eps}, we conclude that
\begin{equation*}
\sup_{ t\in [0, T]} \left\vert \lambda_{ \varepsilon_{ 1}, t}^{ \prime} - \lambda_{ \varepsilon_{ 2}, t}^{ \prime}\right\vert \xrightarrow[ \varepsilon_{ 1}, \varepsilon_{ 2}\to 0]{}0.
\end{equation*}
\medskip
\noindent \textit{Conclusion:} by the previous steps, $ \lambda_{ \varepsilon}$ converges as $ \varepsilon\to 0$ to $ \lambda$ and $ \lambda$ is $C^{ 1}$. \qedhere

\subsection{Proof of Proposition~\ref{prop:equilibrium}\label{sec:proof_prop_equilibrium}}
Suppose that $ \xi= \xi^{ eq, \ell}$ for some $ \ell = \Phi \left( \kappa \ell\right)$. Hence, we see from \eqref{eq:derivative_lambda} and \eqref{eq:rho} that for all $t\geq0$
\begin{equation*}
\lambda^{ \prime}_{t}=\Phi^{ \prime}\left(\xi_{t}+\int_{0}^{t}h(t-u)\lambda_{u}\rmd u\right)\int_{0}^{t}h(t-u)\lambda^{ \prime}_{u}\rmd u.
\end{equation*}
By continuity, for all $T>0$, there exists $C_{ T}>0$ such that $ \sup_{ t\in [0, T]} \left\vert \Phi^{ \prime}\left(\xi_{t}+\int_{0}^{t}h(t-u)\lambda_{u}\rmd u\right) \right\vert\leq C_{ T}$ so that for all $t\in [0, T]$
\begin{equation*}
\left\vert \lambda^{ \prime}_{t} \right\vert \leq C_{ T} \int_{0}^{t} \left\vert h(t-u) \right\vert \left\vert \lambda^{ \prime}_{u} \right\vert\rmd u.
\end{equation*}
One conclude from Lemma~\ref{lem:Gronwall_eps} that $ \lambda_{ t}^{ \prime}=0$ on $[0, T]$ for all $T$. Hence $ \lambda$ is constant.
Conversely, suppose that $ \lambda$ is constant: $ \lambda_{ t}= \lambda_{ 0}$ for all $t\geq0$. Then $ \lambda_{ 0} = \Phi( \xi_{ t} + \lambda_{ 0} \int_{ 0}^{t}h(u) {\rm d}u)$. Passing to the limit as $t\to \infty$, we obtain that $ \lambda_{ 0} = \Phi \left( \kappa \lambda_{ 0}\right)$ so that $ \lambda_{ 0} = \ell$ is a solution to \eqref{eq:fixed_point_lambda}. Hence for all $t\geq0$, $ \kappa\ell= \Phi^{ -1}(\ell)= \xi_{ t}+ \ell \int_{ 0}^{t} h(u) {\rm d}u$, which gives that $ \xi_{ t}= \xi^{ eq, \ell}_{ t}$ for all $t\geq0$.

\section{Proofs of the stability results in the subcritical case\label{sec:proofs}}

\subsection{Proof of Theorem \ref{th:conv_rate}}
\subsubsection{Preliminary}
Theorem  \ref{th:conv_rate} is a consequence of the following result that we first prove.
\begin{proposition}
\label{prop:conv_rate}
Suppose $\Phi$ is a nonnegative twice differentiable Lipschitz continuous function such that $\|\Phi''\|_{\infty}<\infty$, $h$ is integrable on $[0, +\infty)$, $\xi$ is bounded with $ \xi_{ t} \xrightarrow[ t\to\infty]{}0$. Define
\begin{equation}
\label{eq:v_xi}
v_{ t}^{ \xi}:= \sup_{ s\geq t} \left\vert \xi_{ s} \right\vert, t\geq 0.
\end{equation}
Suppose that $\ell$ verifies the subcritical condition \eqref{eq:cond_stab_ell} and that $\lambda=\lambda^\xi$ is bounded. Let $ \varepsilon_{ 0}>0$ and $t_0\geq 0$ such that Conditions \eqref{eq:rate_tauT} and \eqref{eq:prox_lambda_ell} are satisfied. Recall in particular the definition of $\tau= \tau(\varepsilon_0)\in [0, 1)$ given by \eqref{eq:rate_tauT}.  
Then, for any $ k\ge1$ and any $M\ge t_{0}$, it holds 
\begin{align}
|\lambda_{(k+1)M}-\ell |
&\le \frac{ \tau}{ \left\Vert h \right\Vert_{ 1}} \sum_{j=0}^{k-1}\tau^{j}v^{ \xi}_{ (k+1-j)M}+ \frac{\tau(\left\Vert \lambda \right\Vert_{ \infty}+2\ell)}{ \left\Vert h \right\Vert_{ 1} \left(1-\tau\right)}{ \int_{ M}^{+\infty} |h(u)| {\rm d}u }+\tau^{k}(\left\Vert \lambda \right\Vert_{ \infty}+ \ell).\label{eq:rate_to_ell}
\end{align} 
\end{proposition}

\begin{proof}[Proof of Proposition~\ref{prop:conv_rate}]
Recall that $ \ell$ solves $ \ell= \Phi \left( \kappa \ell\right)$ with $ \kappa= \int_{ 0}^{+\infty} h(u) {\rm d}u$. Let $t\geq0$, write
\begin{align*}
\lambda_{t}-\ell &= \Phi\left(\xi_{ t}+ \int_{0}^{t}h(t-u)\lambda_{u} {\rm d}u\right) - \Phi\left(\ell \int_{ 0}^{+\infty} h(u) {\rm d}u\right) \\
&=\Phi\left( \xi_{ t}- \ell \int_{ t}^{+\infty} h(u) {\rm d}u+  \int_{0}^{t}h(t-u)(\lambda_{u}-\ell) {\rm d}u + \ell\int_{0}^{+\infty}h(u) {\rm d}u \right) - \Phi\left(\ell \int_{ 0}^{+\infty} h(u) {\rm d}u\right).
\end{align*}
Therefore, a Taylor expansion gives
\begin{align}
\label{aux:Taylor_lambda}
\lambda_{t}-\ell &= \left(\xi_{ t}-  \ell\int_{t}^{+\infty}h(u) {\rm d}u+ \int_{0}^{t}h(t-u)(\lambda_{u}-\ell) {\rm d}u\right) \Phi^{ \prime}( \kappa\ell) + R_{ t},
\end{align}
where the remainder term is
\begin{align}
\label{eq:Rt}
R(t)&:=\int_{ \kappa\ell}^{\beta_{t}}\Phi''(x)(\beta_{t}-x) {\rm d}x\le \frac{ \|\Phi''\|_{\infty}}{ 2}\left(\xi_{ t}-  \ell\int_{t}^{+\infty}h(u) {\rm d}u+ \int_{0}^{t}h(t-u)(\lambda_{u}-\ell) {\rm d}u\right)^{2}.
\end{align}
for $\beta_{t}:=\xi_{ t}+ \int_{0}^{t}h(t-u)\lambda_{u} {\rm d}u$. Fix now $ \varepsilon_{ 0}>0$ and $t_0\geq 0$ satisfying both \eqref{eq:rate_tauT} and \eqref{eq:prox_lambda_ell}: 
For $t\geq t_{ 0}$, $ \left\vert \lambda_{ t} - \ell \right\vert < \varepsilon_{ 0}$ so that
\begin{align*}
\int_{0}^{t} \left\vert h(t-u) \right\vert|\lambda_{u}-\ell| {\rm d}u&=\int_{0}^{t_{ 0}} \left\vert h(t-u) \right\vert|\lambda_{u}-\ell| {\rm d}u+ \int_{t_{0}}^{t} \left\vert h(t-u) \right\vert|\lambda_{u}-\ell| {\rm d}u,\\
&\leq \left( \left\Vert \lambda \right\Vert_{ \infty} + \ell\right)\int_{t-t_{ 0}}^{t} \left\vert h(u) \right\vert {\rm d}u + \varepsilon_{ 0} \int_{0}^{t-t_{ 0}} \left\vert h(u) \right\vert {\rm d}u,\\
&\leq c_{ 0}\int_{t-t_{ 0}}^{+\infty} \left\vert h(u) \right\vert {\rm d}u + \varepsilon_{ 0} \left\Vert h \right\Vert_{ 1},
\end{align*}
where $c_{ 0}:= \left\Vert \lambda \right\Vert_{ \infty}+ \ell$. Recalling the definition of $H_t$ in \eqref{eq:Ht} and
using the previous bound into \eqref{eq:Rt}, we obtain that $t\geq t_{ 0}$, 
\begin{equation}
\label{aux:bound_Rt}
 \left\vert R(t) \right\vert \leq \frac{ \left\Vert \Phi^{ \prime \prime} \right\Vert_{ \infty}}{ 2}\left( \left\vert \xi_{ t} \right\vert+ \ell H_{ t} + c_{ 0} H_{ t-t_{ 0}} + \varepsilon_{ 0} \left\Vert h \right\Vert_{ 1}\right)\left( \left\vert \xi_{ t} \right\vert+  \ell H_{ t}+ \int_{0}^{t} \left\vert h(t-u) \right\vert \left\vert \lambda_{u}-\ell \right\vert {\rm d}u\right).
\end{equation} 
Going back to \eqref{aux:Taylor_lambda}, it follows that for $t\ge t_{0}$, we have
\begin{align*}
|\lambda_{t}-\ell | &\leq \left\vert \Phi^{ \prime}( \kappa\ell) \right\vert\left( \left\vert \xi_{ t} \right\vert +  \ell H_{ t} + \int_{0}^{t} \left\vert h(t-u) \right\vert \left\vert \lambda_{u}-\ell \right\vert {\rm d}u\right)  \\
&+ \frac{ \left\Vert \Phi^{ \prime \prime} \right\Vert_{ \infty}}{ 2}\left( \left\vert \xi_{ t} \right\vert+ \ell H_{ t} + c_{ 0} H_{ t-t_{ 0}} + \varepsilon_{ 0} \left\Vert h \right\Vert_{ 1}\right)\left( \left\vert \xi_{ t} \right\vert+  \ell H_{ t}+ \int_{0}^{t} \left\vert h(t-u) \right\vert \left\vert \lambda_{u}-\ell \right\vert {\rm d}u\right).
\end{align*}
With no loss of generality, we can further assume that $t_{ 0}>0$ given in \eqref{eq:rate_tauT} is chosen such that
\begin{equation}
\label{aux:constraint_t0}
\left\vert \xi_{ u} \right\vert\leq \varepsilon_{ 0}, \text{ and } H_{ u}\leq \varepsilon_{ 0},\ u\geq t_{ 0}.
\end{equation}
Therefore, for $t\geq 2t_{ 0}$,
\begin{align*}
|\lambda_{t}-\ell | &\leq \left\lbrace \left\vert \Phi^{ \prime}( \kappa\ell) \right\vert+ \varepsilon_{ 0}\frac{ \left\Vert \Phi^{ \prime \prime} \right\Vert_{ \infty}}{ 2}\left( 1+ \ell  + c_{ 0} + \left\Vert h \right\Vert_{ 1} \right)\right\rbrace\left( \left\vert \xi_{ t} \right\vert +  \ell H_{ t} + \int_{0}^{t} \left\vert h(t-u) \right\vert \left\vert \lambda_{u}-\ell \right\vert {\rm d}u\right).
\end{align*}
Note that the prefactor above is nothing else than $ \frac{ \tau}{ \left\Vert h \right\Vert_{ 1}}$ where $ \tau$ is given in \eqref{eq:rate_tauT}. Hence, we obtain, for $t\geq 2 t_{ 0}$
\begin{align}
\label{eq:Gron_lambda_ell}
|\lambda_{t}-\ell | &\leq \frac{ \tau}{ \left\Vert h \right\Vert_{ 1}}\left( \left\vert \xi_{ t} \right\vert +  \ell H_{ t} + \int_{0}^{t} \left\vert h(t-u) \right\vert \left\vert \lambda_{u}-\ell \right\vert {\rm d}u\right).
\end{align}
Let us now make an auxiliary calculation: for any general $T\geq 0$ and $t\geq T$, it holds
\begin{align}
\int_{0}^{t} \left\vert h(t-u) \right\vert|\lambda_{u}-\ell| {\rm d}u&=\int_{0}^{T} \left\vert h(t-u) \right\vert|\lambda_{u}-\ell| {\rm d}u+\int_{T}^{t} \left\vert h(t-u) \right\vert|\lambda_{u}-\ell| {\rm d}u \nonumber\\
&\leq c_{ 0}\int_{t-T}^{+\infty} \left\vert h(u) \right\vert {\rm d}u+ \left\Vert h \right\Vert_{ 1}\sup_{u\ge T}|\lambda_{u}-\ell| ,\label{eq:conv_h_lambdaell}
\end{align} 
where we recall that $c_{ 0}= \left\Vert \lambda \right\Vert_{ \infty} + \ell$. Combining \eqref{eq:Gron_lambda_ell} and \eqref{eq:conv_h_lambdaell}, we obtain, for $t\ge \max(T, 2t_{ 0})$,
\begin{align}
|\lambda_{t}-\ell | \leq \frac{ \tau}{ \left\Vert h \right\Vert_{ 1}} \left( \left\vert \xi_{ t} \right\vert +  \ell H_{ t} + c_{ 0}H_{ t-T} + \left\Vert h \right\Vert_{ 1}\sup_{ u\geq T} \left\vert \lambda_{ u}-\ell \right\vert\right)\nonumber.
\end{align} 
The rest of the proof consists in iterating the resulting inequality: recalling the definition of $v^{ \xi}$ in \eqref{eq:v_xi}, for any $T_{ 1}\geq T_{ 0} \geq 2t_{ 0}$,
\begin{align}
\label{eq:iter_T0_T1}
\sup_{ t\geq T_{ 1}}|\lambda_{t}-\ell | \leq \frac{ \tau}{ \left\Vert h \right\Vert_{ 1}} \left( v^{ \xi}_{ T_{ 1}}+  \ell H_{ T_{ 1}} + c_{ 0}H_{ T_{ 1}-T_{ 0}} + \left\Vert h \right\Vert_{ 1}\sup_{ u\geq T_{ 0}} \left\vert \lambda_{ u}-\ell \right\vert\right).
\end{align}
Start with some $ M \ge 2t_{0}>0$ and set $T_{ 0}=M$ and $T_{ 1}=2M$: from \eqref{eq:iter_T0_T1}, it holds that 
\begin{align*}
\sup_{t\ge 2M}|\lambda_{t}-\ell | &\leq \frac{ \tau}{ \left\Vert h \right\Vert_{ 1}} \left(v_{ 2M}^{ \xi}+\ell H_{2M} + c_{ 0} {H_{M} }+ \left\Vert h \right\Vert_{ 1}\sup_{ u\geq M} \left\vert \lambda_{ u}-\ell \right\vert\right).\end{align*} 
Applying now \eqref{eq:iter_T0_T1} with $T_{ 0}=2M$ and $T_{ 1}=3M$, we have
\begin{align*}
\sup_{t\ge 3M}|\lambda_{t}-\ell | &\leq \frac{ \tau}{ \left\Vert h \right\Vert_{ 1}} \left(v_{ 3M}^{ \xi}+\ell H_{3M}+ c_{ 0} {H_{M} }+ \left\Vert h \right\Vert_{ 1}\sup_{ u\geq 2M} \left\vert \lambda_{ u}-\ell \right\vert\right)\\
 &\leq \frac{ \tau}{ \left\Vert h \right\Vert_{ 1}} (v_{ 3M}^{ \xi}+ \tau v_{ 2M}^{ \xi} ) + \frac{ \tau}{ \left\Vert h \right\Vert_{ 1}}\ell (H_{ 3M}+ \tau H_{2M})+ c_{ 0} \frac{ \tau}{ \left\Vert h \right\Vert_{ 1}}(1+ \tau){H_{M} }+ \tau^{2}\sup_{ u\geq M} \left\vert \lambda_{ u}-\ell \right\vert.\end{align*} 
Iterating the latter we get for any $ k\ge1$, 
\begin{align}
\sup_{t\ge (k+1)M}|\lambda_{t}-\ell | &\leq \frac{ \tau}{ \left\Vert h \right\Vert_{ 1}} \sum_{j=0}^{k-1} \tau^{j}v^{\xi}_{ (k+1-j)M}+ \frac{ \tau}{ \left\Vert h \right\Vert_{ 1}}\ell \sum_{j=0}^{k-1} \tau^{j} H_{ (k+1-j)M}+ c_{ 0}{H_{M} } \frac{ \tau}{ \left\Vert h \right\Vert}\sum_{j=0}^{k-1} \tau^{j}+ \tau^{k}c_{0}\nonumber\\ 
&\le \frac{ \tau}{ \left\Vert h \right\Vert_{ 1}} \sum_{j=0}^{k-1} \tau^{j}v^{\xi}_{ (k+1-j)M}+ \frac{\tau(c_{ 0}+\ell)}{ \left\Vert h \right\Vert_{ 1}(1- \tau)}{H_{M} }+ \tau^{k}c_{0}, \nonumber
\end{align} as $H$ is non decreasing, which completes the proof.\end{proof}

\subsubsection{Proof of the convergence}
We prove in this paragraph Item 1 of Theorem~\ref{th:conv_rate}, that is the convergence of $ \lambda_{ t}$ to $\ell$ as $t\to\infty$. Since $ \xi_{ t} \xrightarrow[ t\to \infty]{}0$, we have $ \frac{ \tau}{ \left\Vert h \right\Vert_{ 1}} \sum_{j=0}^{k-1} \tau^{j}v^{\xi}_{ (k+1-j)M}\leq \frac{ \tau}{ \left\Vert h \right\Vert_{ 1} \left(1- \tau\right)}v_{ 2M}^{ \xi} \xrightarrow[ M\to\infty]{}0$. Hence, for $t$ large, take $ k+1= \left\lfloor \sqrt{ t}\right\rfloor$ and $ M = \frac{ t}{ k+1}$ so that both $k+1\to\infty$ and $M\to \infty$ as $t\to\infty$ and we have $ \left\vert \lambda_{ t}- \ell \right\vert \xrightarrow[ t\to\infty]{}0$ from \eqref{eq:rate_tauT}, as required.

\subsubsection{A technical lemma}
 We now turn to the proof of the rates of convergence to $\ell$. First we give a technical lemma controlling the order of the first term on the righthand side of \eqref{eq:rate_to_ell}.  It should be noted that for this term to be small, it is enough that one parameter among $k$ or $M$ is large, the other parameter may remain bounded.
 
 \begin{lemma}
 \label{lem:rate_sum_xi} 
Fix $\tau\in(0,1)$, $k\ge 1$ and $M>0$. Then the following holds:
\begin{enumerate}[label=(\alph*)]
\item Exponential decay on $ \xi$: if for some constants $A>0$ and $a>0$, we have $v^{\xi}_{t}=A e^{-at}$, $t\geq0$, it holds if $ e^{ -aM}< \tau$,  that \begin{align}\label{eq:rate_sum_xi_exp}
\sum_{j=0}^{k-1}\tau^{j}v^{\xi}_{ (k+1-j)M}\le  A\frac{ \tau^{k} e^{-aM}}{ \tau e^{ aM}-1}.
 \end{align}  
  \item Polynomial decay on $ \xi$: if for some constants $A>0$ and $a>0$, we have $v^{\xi}_{t}=A t^{-a}$, $t\geq0$,  then it holds for some positive constant $C_1>0$, that
 \begin{align}\label{eq:rate_sum_xi_pol}
 \sum_{j=0}^{k-1}\tau^{j}v^{\xi}_{ (k+1-j)M} \le C_1M^{-a}k^{-a}.
 \end{align}
 Moreover the constant $C_1=C_1(a,\tau)$ given in \eqref{eq:value_C1} depends only on $a$ and $\tau$ and remains bounded as $\tau \to \tau_0= \Vert h_1 \Vert_1 \vert \Phi^\prime(\kappa \ell)\vert\in [0, 1)$.
\end{enumerate}
\end{lemma}
\begin{proof}[Proof of Lemma \ref{lem:rate_sum_xi}] Since $ e^{ -aM}< \tau$, Inequality \eqref{eq:rate_sum_xi_exp} follows directly from
\begin{align*}
\sum_{j=0}^{k-1}\tau^{j}v^{\xi}_{ (k+1-j)M}
&=Ae^{-aM} \left(\frac{\left(e^{-aM}\right)^{ k}-\tau^{ k}}{1-\tau e^{aM}}\right) \le Ae^{-aM} \frac{\tau^{ k}}{\tau e^{aM}-1}.
\end{align*}

Now turn to Inequality~\eqref{eq:rate_sum_xi_pol}: we write
\begin{align}
\label{eq:vxi_dec_pol}
\sum_{j=0}^{k-1}\tau^{j}v^{\xi}_{ (k+1-j)M}&=AM^{-a}\sum_{j=0}^{k-1}\tau^{j}(k+1-j)^{-a }=AM^{-a}\sum_{j=0}^{k-1} \psi(j)
\end{align}
where $ \psi: x \mapsto \tau^{x}(k+1-x)^{-a }$. A direct calculation shows that $ \psi$ is strictly decreasing (resp. increasing) on $ [0, \alpha_{ k}]$ (resp. on $[ \alpha_{ k}, k+1]$), for $ \alpha_{ k}:= \frac{ (k+1) \log \left(\tau\right)+a}{ \log(\tau)}$ (indeed $ \alpha_{ k} >0$ for large $k$). Decompose the above sum into
\begin{align}
\label{eq:sum_psi}
\sum_{j=0}^{k-1} \psi(j)= \psi(0) + \sum_{ j=1}^{ \left\lfloor \alpha_{ k}\right\rfloor} \psi(j)+ \sum_{ j=\left\lfloor  \alpha_{ k}\right\rfloor+1}^{ k-1} \psi(j)
\end{align}
the second sum being by convention $0$ in case $ \left\lfloor \alpha_{ k}\right\rfloor+1 > k-1$. This sum is non empty when
\begin{align*}
k-1 - \left( \left\lfloor \alpha_{ k}\right\rfloor+1\right) \geq k-1 - \left( \alpha_{ k} +1\right)= \frac{3\log(\tau) +a}{ \log(1/\tau)} >0
\end{align*}
\textit{i.e.,} when $a > 3 \log \left(1/ \tau\right)$. It can be bounded by
\begin{align*}
\sum_{ j= \left\lfloor \alpha_{ k}\right\rfloor+1}^{ k-1} \psi(j) \leq \int_{ \left\lfloor \alpha_{ k}\right\rfloor+1}^{k} \psi(x) {\rm d}x \leq \int_{ \left\lfloor \alpha_{ k}\right\rfloor+1}^{k} \tau^{ x} {\rm d}x= \frac{ \tau^{ k} - \tau^{ \left\lfloor \alpha_{ k}\right\rfloor+1}}{ \log \left(\tau\right)} \leq \frac{ \tau^{ \left\lfloor \alpha_{ k}\right\rfloor+1}}{ \log \left(1/ \tau\right)} \leq \frac{\tau^{\frac{a}{\log(\tau)}+1}}{\log(1/\tau)} \tau^{ k}.
\end{align*}

Turn now to the first sum in \eqref{eq:sum_psi}: since $ \psi$ is decreasing on $[0, \alpha_{ k}]$
\begin{align*}
\sum_{j=1}^{ \left\lfloor \alpha_{ k}\right\rfloor} \psi(j)= \sum_{ j=0}^{ \left\lfloor \alpha_{ k}\right\rfloor-1} \psi(j+1) \leq \int_{ 0}^{ \left\lfloor \alpha_{ k}\right\rfloor-1}  \psi(x) {\rm d}x\leq\int_{ 0}^{ k} \psi(x) {\rm d}x= \int_{ 0}^{k} \frac{ \tau^{ x}}{ \left(k+1- x\right)^{ a}} {\rm d}x.
\end{align*}
 Let us estimate the last integral:
\begin{align}
I_{ k}(a)&:= \int_{ 0}^{k} \frac{ \tau^{ x}}{ \left(k+1- x\right)^{ a}} {\rm d}x= \frac{ 1}{ k^{ a-1}}\int_{ 0}^{1} \frac{ e^{ - k\log \left(1/ \tau\right)y}}{ \left(1+ \frac{ 1}{ k}- y\right)^{ a}} {\rm d}y, \nonumber\\
&= \frac{ 1}{ k^{ a}} \left( \frac{ 1}{\left(1+\frac1k\right)^{a} \log(1/ \tau)}-\frac{ k^{a}\tau^{ k}}{ \log(1/ \tau)} + \frac{ a}{ \log \left(1/ \tau\right)} \int_{ 0}^{1} \frac{ e^{ -k \log \left(1/ \tau\right)y}}{ \left(1+ \frac{ 1}{ k}-y\right)^{ a+1}} {\rm d}y\right)\label{aux:Ika}
\end{align}
by integration by parts. The remaining of the proof is to show that the integral within the last estimate is bounded in $k$:
\begin{align*}
J_{ k}(a)&:=  \int_{ 0}^{1} \frac{ e^{ -k \log \left(1/ \tau\right)y}}{ \left(1+ \frac{ 1}{ k}-y\right)^{ a+1}} {\rm d}y =  \int_{ 0}^{1/ \sqrt{ k}} \frac{ e^{ -k \log \left(1/ \tau\right)y}}{ \left(1+ \frac{ 1}{ k}-y\right)^{ a+1}} {\rm d}y +  \int_{ 1/ \sqrt{ k}}^{1} \frac{ e^{ -k \log \left(1/ \tau\right)y}}{ \left(1+ \frac{ 1}{ k}-y\right)^{ a+1}} {\rm d}y,\\
&\le \frac{ 1}{ \sqrt{ k} \left(1+ \frac{ 1}{ k} - \frac{ 1}{ \sqrt{ k}}\right)^{ a+1}}+ \left(1- \frac{ 1}{ \sqrt{ k}}\right)k^{ a+1}e^{ - \sqrt{ k}\log \left(1/ \tau\right)}
.\end{align*}
From this we deduce that $J_{ k}(a) \xrightarrow[ k\to\infty]{}0$ and hence, coming back to \eqref{aux:Ika}, we deduce that there exists some constant $C=C(a, \tau)>0$, \textit{e.g.,} $C= \frac{1}{\log(1/\tau)}\left(1+ a+a\left(\frac{2(a+1)}{e\log(1/\tau)}\right)^{2(a+1)}\right)$, such that for all large $k$, $ I_{ k}(a)\leq \frac{ C}{ k^{ a}}$. Putting this estimate together with \eqref{eq:sum_psi} into \eqref{eq:vxi_dec_pol}, we obtain \eqref{eq:rate_sum_xi_pol} for the following constant \begin{align}
    \label{eq:value_C1} C_1=A\left(1+\frac{\tau^{\frac{a}{\log(\tau)}+1}}{\log(1/\tau)}\left(\frac{a}{e\log(1/\tau)}\right)^a+ \frac{1}{\log(1/\tau)}\left(1+ a+a\left(\frac{2(a+1)}{e\log(1/\tau)}\right)^{2(a+1)}\right)\right).
\end{align}\qedhere
 \end{proof}
 
\subsubsection{Proof of Theorem \ref{th:conv_rate}}
We compute the rates implied by Proposition \ref{prop:conv_rate} under the different decay constraints on $\xi$ and $H$ for a positive large time $T>t_0$ where $t_0$ is defined in \eqref{eq:prox_lambda_ell}.    In the sequel we select the largest possible values for $k$ and $M$ subject to the constraint $(k+1)M=T$ that give a minimal upper bound for \eqref{eq:rate_to_ell}. Below $A,\, B,\, a$ and $b$ are positive constants. We will also make use of the so called Lambert $W$ function defined as the inverse of the function $x \mapsto x e^{ x}$ on $[-1, +\infty)$ onto $[- e^{ -1}, +\infty)$.
 
  \medskip
  
 \noindent \textbf{Case $\xi$ is exponentially decaying.} Suppose that $v^{\xi}_{t}=A e^{-at}$, $t\geq0$, we consider the possible decay behavior for $H$. In addition to $M\geq t_0$ imposed by the assumptions of Proposition~\ref{prop:conv_rate}, we require also that $M$ is large enough so that $e^{-aM}< \frac{\tau}{2}$.  Gathering the results of Proposition \ref{prop:conv_rate} and Lemma \ref{lem:rate_sum_xi} we get
\begin{align}
|\lambda_{(k+1)M}-\ell |
&\le \frac{\tau(\left\Vert \lambda \right\Vert_{ \infty}+2\ell)}{ \left\Vert h \right\Vert_{ 1} \left(1-\tau\right)} H_M+ \tau^k \left(\frac{ A\tau e^{-aM}}{ \left\Vert h \right\Vert_{ 1}(\tau e^{ aM}-1)} +\left\Vert \lambda \right\Vert_{ \infty}+ \ell\right),\nonumber\\
&\le \frac{\tau(\left\Vert \lambda \right\Vert_{ \infty}+2\ell)}{ \left\Vert h \right\Vert_{ 1} \left(1-\tau\right)} H_M+ \tau^k \left(\frac{ A\tau^2}{ 2\left\Vert h \right\Vert_{ 1}} +\left\Vert \lambda \right\Vert_{ \infty}+ \ell\right),\nonumber\\
&\leq  C_2\left(H_{M}+\tau^{k}\right),\quad \mbox{where }C_2:=\frac{\tau(\left\Vert \lambda \right\Vert_{ \infty}+2\ell)}{ \left\Vert h \right\Vert_{ 1} \left(1-\tau\right)}\vee\left(\frac{ A\tau^2}{ 2\left\Vert h \right\Vert_{ 1}} +\left\Vert \lambda \right\Vert_{ \infty}+ \ell\right)\label{eq:rate_to_ell_xiexp}
\end{align} 
so that $\tau\mapsto C_2$ remains bounded as $\tau\to \tau_0= \Vert h \Vert_1 \vert\Phi^{\prime}(\kappa \ell)\vert\in [0,1)$.
 
\medskip
\noindent \textit{In the case $H_{t}\leq B e^{-bt}$}, $t\geq0$, it follows that $|\lambda_{(k+1)M}-\ell |
\le C_2\left(B e^{-bM}+\tau^{k}\right).$ 
Let $T>0$ be large and define $k+1= \left\lfloor \sqrt{T}\sqrt{b/\log(1/\tau)}\right\rfloor$ and $M= \frac{ T}{ k+1}
$ so that $e^{-bM}$ and $\tau^{k}$ are of the same order.  
Consider $T$  such that $M\ge t_0$ and $e^{-aM}< \frac{\tau}{2}$, namely $T\ge b\log(1/\tau)\max\left(t_0^2, \frac{\log^2\left(\frac{2}{\tau}\right)}{a^2} \right)$. 
It follows that, for $C_2$ defined in \eqref{eq:rate_to_ell_xiexp}, \begin{align}
     \label{eq:rate_exp_exp}
    |\lambda_{T}-\ell |
\le C e^{- \sqrt{ b \log \left(1/ \tau\right)}\sqrt{T}},\quad  C:=C_2(B+ 1),\  T\ge \sigma(t_0):=b\log(1/\tau)\left(t_0^2\vee \frac{\log^2\left(\frac{2}{\tau}\right)}{a^2} \right).
 \end{align}
\medskip
\noindent \textit{If $H_{t}=B t^{-b}$}, $t\geq0$, it follows that $|\lambda_{(k+1)M}-\ell |
\le C_2\left(M^{-b}+\tau^{k}\right).$ 
Similarly, let $T>0$ be large and define $k+1= \left\lfloor b/\log(1/\tau)(\log T-\log \log T)\right\rfloor $ and $M= \frac{ T}{ k+1}$, so that $M^{-b}$ and $\tau^{k}$ are of the same order. Consider $T$ such that $M\ge t_0$ and $e^{-aM}< \frac{\tau}{2}$, namely $T\ge \max\left(t_0, \frac{\log\left(\frac{2}{\tau}\right)}{a}\right) b\log(1/\tau)\log T$. 
  It follows that, for $C_2$ defined in \eqref{eq:rate_to_ell_xiexp},\begin{align}
     \label{eq:rate_exp_pol}
    |\lambda_{T}-\ell |\le 2C_2\log(T) ^{b}T^{-b},\quad  T\ge \sigma(t_0):= -c W\left(-\frac1{c}\right) \mbox{ with  }c:=b\log\left(\frac1\tau\right)\left(t_0\vee \frac{\log\left(\frac{2}{\tau}\right)}{a^2} \right),
 \end{align} where $W$ is the Lambert $W$ function.
 
\medskip
\noindent \textit{If $h$ is compactly supported on $[0, S]$,} then $H_{t}=0$ for $t\geq S$. We select $M= S \vee t_{ 0}$, then $k= \left\lfloor T/M\right\rfloor-1$. 
  It follows that, for $C_2$ defined in \eqref{eq:rate_to_ell_xiexp}, \begin{align}
     \label{eq:rate_exp_comp}
     |\lambda_{T}-\ell |\le 2C_2e^{- \log \left(1/ \tau\right) \frac{ T}{ S\vee t_{ 0}}},\quad T\ge \sigma(t_0):=S\vee t_0 .
 \end{align}

\medskip
\noindent    \textbf{Case $\xi$ is polynomially  decaying.} Suppose that $v^{\xi}_{t}=A t^{-a}$, $t>0$, we consider the possible decay behaviors for $H$. Gathering the results of Proposition \ref{prop:conv_rate} and Lemma \ref{lem:rate_sum_xi} we get, for a constant $C_{ 3}=C_{ 3}(a, \tau, \left\Vert \lambda \right\Vert_{ \infty}, h)$ that remains bounded as $ \tau\to \tau_{ 0}$
\begin{align}
\label{eq:lambda_ell_xi_pol}
|\lambda_{(k+1)M}-\ell |& \leq C_{ 3} \left(M^{-a}k^{-a}+ H_{M}+ \tau^{k}\right)\\
     C_3&:= \frac{ \tau C_1}{ \left\Vert h \right\Vert_{ 1}} \vee  \frac{\tau(\left\Vert \lambda \right\Vert_{ \infty}+2\ell)}{ \left\Vert h \right\Vert_{ 1} \left(1-\tau\right)}\vee(\left\Vert \lambda \right\Vert_{ \infty}+ \ell)\label{eq:value_C3},
\end{align}  where $C_1$ is defined in \eqref{eq:value_C1}.

\medskip 
\noindent \textit{If $H_{t}\leq B e^{-bt}\mathbf{1}_{t\ge 0}$},  it follows that $|\lambda_{(k+1)M}-\ell |
\le C_{ 3}\left((kM)^{-a}+e^{-bM} + \tau^{ k}\right)$. 
Set $T>0$ large and define $k+1= \left\lfloor (b/a) T/\log T\right\rfloor$ and $M= \frac{ T}{ k+1}= \frac{ T}{ \left\lfloor (b/a) T/\log T\right\rfloor}$ so that $e^{-bM}$ and $(kM)^{-a}$ are of the same order. We claim that for $T$ large, $ \tau^{ k}$ is negligible w.r.t. the two other terms: indeed, for this choice of parameters, $ \tau^{ k} \leq (kM)^{ -a}$ and $\tau^{k} \leq e^{-bM}$ if $ \frac{ a^{ 2}}{b \log (1/ \tau)} (\log T)^{ 2} \leq T$. Recalling that we assume also that $M\ge t_0$, it follows that, 
for $C_3$ defined in \eqref{eq:value_C3},\begin{align}
     \label{eq:rate_pol_exp}
    |\lambda_{T}-\ell |
\le 3C_3 T^{-a},\quad    T\ge \sigma(t_0):=\max \left(e^{t_0\frac ba},4cW^2\left(\frac{-1}{2\sqrt c}\right)\right), \  \mbox{with }c:=  \frac{ a^{ 2}}{b \log (1/ \tau)},
 \end{align} where $W$ is the Lambert $W$ function.

\medskip
\noindent \textit{If $H_{t}\leq B t^{-b}\mathbf{1}_{t\ge 0}$}, it follows from \eqref{eq:lambda_ell_xi_pol} that 
\begin{equation}
\label{eq:lambda_ell_pol_ab}
|\lambda_{(k+1)M}-\ell |\le C_3\left((kM)^{-a}+M^{-b}+\tau^{k}\right).
\end{equation} 
Similarly, set $\beta=\left(1-\tfrac{a}{b}\right)\vee 0$ and $T>0$ large. Define $k= \left\lfloor \frac{b}{\log(1/\tau)}T^{\beta}(\log T)^{1- \frac{\beta}{1-a/b}}\right\rfloor$ and $M= \frac{ T}{ k+1}$. Below, we compute each term of \eqref{eq:lambda_ell_pol_ab}.

\noindent Suppose that $b\leq a$: then $\beta=0$ and we have $M^{-b}\leq \left(\frac{b}{\log(1/\tau)}\right)^bT^{-b} (\log T)^b$, $\tau^k\leq \tau^{-1} T^{-b}$. 
Using that $M\ge t_0$, \textit{i.e.,} $T\geq t_0 \frac{b}{\log(1/\tau)}\log T$, we get, for $C_3$ defined in \eqref{eq:value_C3}, that\begin{align}
     \label{eq:rate_pol_pol1}
    |\lambda_{T}-\ell |
\le 3C_3 (\log T)^bT^{-b},\quad    T\ge \sigma(t_0):=-cW\left(\frac{-1}{ c}\right), \ \mbox{with }c:=   \frac{t_0b}{\log(1/\tau)},
 \end{align} where $W$ is the Lambert $W$ function.

\noindent Suppose that $a<b$: then $\beta= 1- \frac{a}{b}$ and we have $M^{-b}\leq \left(\frac{b}{\log(1/\tau)}\right)^bT^{-a}$ and $\tau^k\leq \tau^{-1} e^{-b T^{1-a/b}}\leq \tau^{-1} T^{-a}$ as long as $\frac{a}{b}(\log T) \leq T^{1-\frac{a}{b}}$. 
Using also that $M\ge t_0$, that is $T\geq \left(t_0 \frac{b}{\log(1/\tau)}\right)^{\frac{b}{a}}$, we get, for $C_3$ defined in \eqref{eq:value_C3}, that\begin{align}
     \label{eq:rate_pol_pol2}
    |\lambda_{T}-\ell |
\le 3C_3 T^{-a},\quad    T\ge \sigma(t_0):=\left(t_0 \frac{b}{\log(1/\tau)}\right)^{\frac{b}{a}}\vee \left(- \frac{a}{b-a}W\left(1-\frac{b}{a}\right)\right), 
 \end{align} where $W$ is the Lambert $W$ function.

\medskip
\noindent \textit{If  $h$ is compactly supported on $[0, S]$}, then $H_{t}=0$ for $t\ge S$. We select $M=S \vee t_0$, then defining $k+1= \left\lfloor T/M\right\rfloor$ immediately gives the bound for 
$C_3$ defined in \eqref{eq:value_C3},\begin{align}
     \label{eq:rate_pol_comp}
    |\lambda_{T}-\ell |
\le 3C_3 T^{-a},\quad    T\ge \sigma(t_0):=S\vee t_0.
 \end{align} 
 
\medskip
\noindent    \textbf{Case $\xi$ is with compact support.} Suppose that $ \xi_t=0$ for $t> S_\xi$ for some $S_\xi>0$. Then, for $M\geq \frac{S_\xi}{2}\vee t_0$, we obtain from \eqref{eq:rate_to_ell} that
\[|\lambda_{(k+1)M}-\ell |
\le \frac{\tau(\left\Vert \lambda \right\Vert_{ \infty}+2\ell)}{ \left\Vert h \right\Vert_{ 1} \left(1-\tau\right)}{ H_M}+\tau^{k}(\left\Vert \lambda \right\Vert_{ \infty}+ \ell).\]
Hence, we are back to a situation similar to \eqref{eq:rate_to_ell_xiexp} and the result follows.

\subsection{Proof of Theorem~\ref{th:stab_stationary}}
 Start in a similar way as for the proof of Theorem~\ref{th:conv_rate}: write 
 \begin{align*}
\lambda_{t}-\ell &=\Phi\left( \xi_{ t}+ \int_{0}^{t}h(t-u)(\lambda_{u}-\ell) {\rm d}u + \int_{0}^{t}h(u) {\rm d}u \ell \right) - \Phi\left(\int_{0}^{+\infty}h(u)  {\rm d}u \ell\right).
\end{align*}
Using the definition of $ \xi_{ t}= \xi_{ t}^{ eq, \ell} + \eta_{ t}$, we obtain
 \begin{align*}
\lambda_{t}-\ell &=\Phi\left( \eta_{ t}+ \int_{0}^{t}h(t-u)(\lambda_{u}-\ell) {\rm d}u + \int_{0}^{+\infty}h(u) {\rm d}u \ell \right) - \Phi\left(\int_{0}^{+\infty}h(u)  {\rm d}u \ell\right).
\end{align*}
By a Taylor expansion, we obtain
\begin{align}
\label{eq:lambdat_stationary}
\lambda_{t}-\ell &= \Phi^{ \prime}( \kappa\ell) \left(\eta_{ t}+ \int_{0}^{t}h(t-u)(\lambda_{u}-\ell) {\rm d}u \right)+ S_{ t},
\end{align}
where the remainder $S_{ t}$ satisfies
\begin{align}
\label{eq:St}
 \left\vert S_{ t} \right\vert&\leq  \frac{ \left\Vert \Phi^{ \prime \prime} \right\Vert_{ \infty}}{ 2} \left( \eta_{ t}+ \int_{0}^{t}h(t-u)(\lambda_{u}-\ell) {\rm d}u\right)^{ 2}.
\end{align}
Fix now some constants $\varepsilon_0>0$ and $\varepsilon\in (0, \varepsilon_0)$ that we will define later and consider
\begin{equation}
\label{eq:t_ast_stat}
t^{ \ast}:= \inf \left\lbrace t>0,\ \left\vert \lambda_{ t} - \ell \right\vert > \varepsilon\right\rbrace.
\end{equation}
Since $ \lambda_{ 0}= \Phi \left( \int_{ 0}^{+\infty} h(u) {\rm d}u \ell + \eta_{ 0}\right)= \Phi \left( \kappa\ell + \eta_{ 0}\right)$, there exists $ \delta>0$ such that if $ \left\vert \eta_{ 0} \right\vert< \delta$ then $ \left\vert \lambda_{ 0} - \ell\right\vert = \left\vert \Phi( \kappa\ell + \eta_{ 0}) - \Phi( \kappa\ell) \right\vert\leq \|\Phi'\|_{\infty}|\eta_0|\leq \frac{ \varepsilon}{ 2}$ if $\delta \leq \frac{\varepsilon}{2 \| \Phi'\|_{\infty}}$. 

By the previous estimate on $ \left\vert \lambda_{ 0} - \ell \right\vert$, we have $ t^{ \ast}>0$ and for all $ u\in [0, t^{ \ast}]$, $ \left\vert \lambda_{u} - \ell\right\vert \leq \varepsilon$.  
 Then, we obtain from \eqref{eq:St}, for $t\leq t^{ \ast}$,
\begin{align*}
\left\vert S_{ t} \right\vert \leq \frac{ \left\Vert \Phi^{ \prime \prime} \right\Vert_{ \infty}}{ 2} \left( \left\vert \eta_{ t} \right\vert + \varepsilon  \int_{ 0}^{t} \left\vert h(t-u) \right\vert {\rm d} u\right)^{ 2}\leq \frac{ \left\Vert \Phi^{ \prime \prime} \right\Vert_{ \infty}}{ 2} \left( \left\vert \eta_{ t} \right\vert + \varepsilon \left\Vert h \right\Vert_{ 1}\right)^{ 2},
\end{align*}
Using the same argument in the first term of \eqref{eq:lambdat_stationary}, we obtain finally, for $t\leq t^{ \ast}$,
\begin{align*}
\left\vert \lambda_{t}-\ell \right\vert &\leq  
\left( \frac{ \left\vert \eta_{ t} \right\vert}{ \left\Vert h \right\Vert_{ 1}} + \varepsilon \right) \left( \left\Vert h \right\Vert_{ 1}\left\vert \Phi^{ \prime}( \kappa\ell) \right\vert+ \left\Vert h \right\Vert_{ 1}\frac{ \left\Vert \Phi^{ \prime \prime} \right\Vert_{ \infty}}{ 2} \left( \left\vert \eta_{ t} \right\vert + \varepsilon \left\Vert h \right\Vert_{ 1}\right)\right).
\end{align*}
Recalling the hypothesis \eqref{eq:etat_small} made on $ \eta$, recalling also that $ \tau_{ 0}= \left\Vert h \right\Vert_{ 1} \left\vert \Phi^{ \prime}( \kappa\ell) \right\vert$ we obtain that, for $t\leq t^{ \ast}$,
\begin{align}
\left\vert \lambda_{t}-\ell \right\vert &\leq  \left( \frac{ \delta}{ \left\Vert h \right\Vert_{ 1}}+ \varepsilon\right) \left( \tau_{ 0}+ \left\Vert h \right\Vert_{ 1}\frac{ \left\Vert \Phi^{ \prime \prime} \right\Vert_{ \infty}}{ 2} \left( \delta + \varepsilon \left\Vert h \right\Vert_{ 1}\right)\right), \nonumber\\
&= \varepsilon \left( \tau_{ 0}+ \left\Vert h \right\Vert_{ 1}^{ 2}\frac{ \left\Vert \Phi^{ \prime \prime} \right\Vert_{ \infty}}{ 2} \varepsilon\right) +  \delta \left( \frac{ \tau_{ 0}}{ \left\Vert h \right\Vert_{ 1}}+ \frac{ \left\Vert \Phi^{ \prime \prime} \right\Vert_{ \infty}}{ 2} \left( \delta + 2\varepsilon \left\Vert h \right\Vert_{ 1}\right)\right) \label{aux:stab_stat}.
\end{align}
To complete the proof, we calibrate the necessary constraints on $\delta$ and $\varepsilon.$ By \eqref{eq:cond_stab_ell}, $ \rho:= 1- \tau_{ 0}>0$. Define further
\begin{equation*}
\tau_{ 1}= \tau_{ 0}+ \frac{ \rho}{ 3}, \ \text{ and }\ \tau_{ 2}:= \tau_{ 0}+ \frac{ 2 \rho}{ 3}.
\end{equation*}
Hence, by construction, we have
\begin{equation*}
\tau_{ 0}= \left\Vert h \right\Vert_{ 1} \left\vert \Phi^{ \prime}( \kappa\ell) \right\vert \leq \tau_{ 1} < \tau_{ 2}<1.
\end{equation*}
With this notation at hand, let $ \varepsilon_{ 0}>0$ which satisfies 
\begin{equation}
\label{eq:eps0}
\varepsilon<\varepsilon_{ 0} < \frac{2 \left(\tau_{ 1}- \tau_{ 0}\right)}{ \left\Vert \Phi^{ \prime \prime} \right\Vert_{ \infty} \left\Vert h \right\Vert_{ 1}^{ 2}}.
\end{equation}
Consider $ \delta>0$ that verifies 
\begin{equation}
\label{eq:cond_delta}
\delta \leq \min \left(\frac{\varepsilon}{2 \Vert \Phi^\prime \Vert_\infty}, \varepsilon \left\Vert h \right\Vert_{ 1}, \frac{ \varepsilon \left( \tau_{ 2}- \tau_{ 1}\right)}{ \left( \frac{ \tau_{ 0}}{ \left\Vert h \right\Vert_{ 1}} + \frac{ 3\left\Vert \Phi^{ \prime \prime} \right\Vert_{ \infty}}{ 2} \left\Vert h \right\Vert_{ 1}\varepsilon\right)}\right).
\end{equation}
By Condition \eqref{eq:eps0}, the first term in the right hand side of \eqref{aux:stab_stat} is smaller than $ \varepsilon \tau_{ 1}$, and since $ \delta\leq \varepsilon \left\Vert h \right\Vert_{ 1}$ by \eqref{eq:cond_delta}, we have
\begin{align*}
\left\vert \lambda_{t}-\ell \right\vert &\leq  \varepsilon \tau_{ 1} +  \delta \left( \frac{ \tau_{ 0} }{ \left\Vert h \right\Vert_{ 1}}+ \frac{ 3\left\Vert \Phi^{ \prime \prime} \right\Vert_{ \infty}}{ 2} \left\Vert h \right\Vert_{ 1} \varepsilon\right) \leq \varepsilon \tau_{ 2},\ t\leq t_{ \ast},
\end{align*}
where we used again \eqref{eq:cond_delta} in the last inequality. Since by construction $ \varepsilon \tau_{ 2}< \varepsilon$, we necessarily have that (recall \eqref{eq:t_ast_stat}) $t^{\ast}= +\infty$. Therefore for all $t\geq0$, $ \left\vert \lambda_{ t} -\ell\right\vert \leq \varepsilon$. This proves \eqref{eq:stab_equilibrium_eps0}. The rest of the proof consists in applying the results of Theorem~\ref{th:conv_rate}. More precisely, Condition \eqref{eq:prox_lambda_ell} is valid for $t_0=0$. Therefore, we can proceed exactly as in the proof of Theorem~\ref{th:conv_rate} in a simpler setting which allows for some modifications of the proof that we exemplify here: namely, starting again from \eqref{aux:Taylor_lambda}, we see here that the term $\xi_t - \ell \int_t^{+\infty} h(u) {\rm d}u$ is nothing else than $ \eta_t$. In particular, this term is by definition controlled by $\delta \leq \varepsilon_0$: there is no need to use the bound $\left\vert\xi_t - \ell \int_t^{+\infty} h(u) {\rm d}u\right\vert \leq \vert \xi_t\vert + \ell H_t$ as in \eqref{aux:bound_Rt}. Hence, the constraint on $t_0$ imposed by \eqref{aux:constraint_t0} is not needed here and one can take $t_0=0$. Then \eqref{eq:decay_Lambda_unif} is a direct consequence of Theorem~\ref{th:conv_rate}. This concludes the proof of Theorem~\ref{th:stab_stationary}.

\section{Proofs concerning the excitatory case}
\label{sec:proofs_excitatory}
\subsection{Monotonicity: proof of Proposition~\ref{prop:monotone}}
We address the first item of Proposition~\ref{prop:monotone}. 
We split the proof of Proposition~\ref{prop:monotone} into several steps.

\medskip
\noindent

Start with Item~\ref{it:lambda_strictly_increasing}:  suppose that $ \Phi^{ \prime}$ is strictly positive on $ [0, +\infty)$, that \eqref{eq:rho_pos} holds as well as Condition \eqref{eq:cond_rho} for some $ \delta>0$. 

\medskip

\noindent
\textit{Step 1: strict monotonicity on a neighborhood of $0$.}
We prove here that there exists some $ t_{ 0}\in (0, \delta)$ such that 
\begin{equation}
\label{eq:lambda_pos_t0}
\lambda^{ \prime}_{ t}>0 \text{ for all } t\in \left(0, t_{ 0}\right].
\end{equation} 
Start with \eqref{eq:derivative_lambda}: setting $x_{ t}=\xi_{t}+\int_{0}^{t}h(t-u)\lambda_{u}\rmd u$ and $\rho_t:=\rho(h,\xi,t)$ defined in \eqref{eq:rho}, we have
\begin{equation*}
\lambda^{ \prime}_{t}= \Phi^{ \prime}\left(x_{ t}\right)\rho_{ t}+\Phi^{ \prime}\left(x_{ t}\right)\int_{0}^{t}h(t-u)\lambda^{ \prime}_{u}\rmd u.
\end{equation*}
Iterating $n$ times the convolution in the previous equality gives
\begin{align}
\label{eq:lambdaprime_Theta_n}
\lambda^{ \prime}_{ t} &= \Phi^{ \prime}(x_{ t}) \left(\sum_{ k=0}^{ n}  \left( \Theta^{ k} \rho\right)_{ t}\right) + \Phi^{ \prime}(x_{ t}) \left(\Theta^{ n} \left(h\ast \lambda^{ \prime}\right)\right)_{ t}
\end{align}
where $ \Theta$ is the linear operator defined by $ \left( \Theta \varphi\right)_{ t}:= \int_{ 0}^{t} h(t-u) \Phi^{ \prime}(x_{ u}) \varphi_{ u} {\rm d}u$, $t\geq 0$ and we set $\Theta^0=Id$. For any $t_{ 0}\in (0, 1]$, $ \Theta$ is a bounded operator on $L^{ \infty}([0, t_{ 0}])$: 
\begin{equation*}
 \left\Vert \Theta \varphi \right\Vert_{ \infty, [0, t_{ 0}]} \leq \sup_{ u\in [0, 1]} \Phi^{ \prime}(x_{ u}) \int_{ 0}^{t_{ 0}} h(u) {\rm d}u\left\Vert \varphi \right\Vert_{ \infty, [0, t_{ 0}]}.
\end{equation*}
By continuity $ \sup_{ u\in [0, 1]} \Phi^{ \prime}(x_{ u})<\infty$, so that choosing $t_{ 0}\in (0, 1 \wedge \delta]$ sufficiently small such that $ \sup_{ u\in [0, 1]} \Phi^{ \prime}(x_{ u}) \int_{ 0}^{t_{ 0}} h(u) {\rm d}u< 1$, one obtains from \eqref{eq:lambdaprime_Theta_n} that for $t\in (0, t_{ 0}]$
\begin{align*}
\lambda^{ \prime}_{ t} &= \Phi^{ \prime}(x_{ t}) \left(\sum_{ k=0}^{ +\infty}  \left( \Theta^{ k} \rho\right)_{ t}\right) \geq \Phi^{ \prime}(x_{ t}) \rho_{ t}>0
\end{align*}
where we have used that $h\geq0$, the strict positivity of $ \Phi^{ \prime}$ and $ \rho$ and that $ \rho_{ t}>0$ on $ \left(0, \delta\right)$. This proves the claim \eqref{eq:lambda_pos_t0}.

\medskip
 \noindent
 \textit{Step 2: strict monotonicity on arbitrary time intervals:} Fix some $T> t_{ 0}$. Introduce the time $t_{ \ast}:= \inf \left\lbrace t\in (t_{ 0}, T], \lambda^{ \prime}_{ t}\leq 0\right\rbrace$, with the convention that $t_{ \ast}=T$ if $ \lambda_{ t}^{ \prime}>0$ on $[t_{ 0}, T]$. Since $ \lambda_{ t_{ 0}}^{ \prime}>0$, by continuity, $t_{ \ast}>t_{ 0}$ and on $(0, t_{ \ast})$, $ \lambda^{ \prime}_{ t}>0$. We want to prove that $ \lambda^{ \prime}_{  t_{ \ast}}>0$. Since $ \lambda$ is locally bounded, $ \left\vert \xi_{ t}  + \int_{ 0}^{t} h(t-u) \lambda_{ u} {\rm d}u\right\vert \leq \left\Vert \xi \right\Vert_{ \infty, [0, T]} + \left\Vert h \right\Vert_{ 1} \left\Vert \lambda \right\Vert_{ \infty, [0, T]}$. Hence, as $ \Phi$ is $ \mathcal{ C}^{ 1}$ and $ \Phi^{ \prime}>0$, there exists some $ \varepsilon_{ 0}= \varepsilon_{ 0, T}>0$ such $ \inf_{ t\in [0,T]} \Phi^{ \prime} \left(\xi_{ t}  + \int_{ 0}^{t} h(t-u) \lambda_{ u} {\rm d}u\right)> \varepsilon_{ 0}$. We have, on $(0, t_{ \ast})$, for $ \tilde{ h}(u):= \varepsilon_{ 0} h(u)$,
\begin{equation}
\label{eq:conv_ineq_lambdaprime}
\lambda'_{t}\geq \varepsilon_{ 0}\rho_{ t}+\int_{0}^{t} \tilde{ h}(u)\lambda'_{t-u}du = \varepsilon_{ 0} \rho_{ t}+ \left( \tilde{ h} \ast \lambda^{ \prime}\right)_{ t}.
\end{equation}
Note that with no loss of generality, changing $ \varepsilon_{ 0}$ into a smaller one if required, one can assume that $ \varepsilon_{ 0}\in (0,1)$, so that $ \Vert \tilde{ h} \Vert_{ 1}= \varepsilon_{ 0}\in (0,1)$. Iterating the linear convolution within \eqref{eq:conv_ineq_lambdaprime} one obtains, on $(0, t_{ \ast})$,
\begin{equation}
\label{eq:lambdaprime_Upsilon}
\lambda^{ \prime}_{ t}\geq \varepsilon_{ 0} \left(\Upsilon \ast \rho\right)_{ t}
\end{equation}
where
\begin{equation}
\label{eq:Upsilon}
\Upsilon(t) = \sum_{ k=0}^{ +\infty} \tilde{h}^{ \ast k}(t).
\end{equation}
Note that $ \Vert \tilde{h}^{ \ast k} \Vert_{ 1}\leq \varepsilon_{ 0}^{ k}$ and $ \Vert \tilde{h}^{ \ast k} \Vert_{ \infty}\leq \Vert h \Vert_{ \infty} \varepsilon_{ 0}^{ k-1}$ so that the above series converges in $L^{ 1}\cap L^{ \infty}$ and $ \Upsilon$ defines a nonnegative  continuous function on $[0, +\infty)$.

\medskip
\noindent
\textit{Case 1:} suppose here that $h(0)>0$. By continuity, there exists some $ \delta_{h}>0$ such that $ h(t)\geq c_{ h} \mathbf{ 1}_{ [0, \delta_{ h}]}$ for $ c_{ h}:= \frac{ h(0)}{ 2}>0$. Then, an immediate recursion shows that for all $k\geq1$, the support of $  \left(\mathbf{ 1}_{ [0, \delta_{ h}]}\right)^{ \ast k}$ on $[0, T]$ is at least $[0, (k \delta_{ h})\wedge T]$, that is equal to $[0, T]$ for $k$ sufficiently large. We conclude that $ \Upsilon(t)>0$ for all $t\in (0, T]$. 

\medskip
\noindent
\textit{Case 2: } suppose here that $h(t)>0$ for all $t>0$. In this case, keeping only the term for $k=1$ in \eqref{eq:Upsilon} (all the remaining terms are nonnegative), we simply have $ \Upsilon(t)\geq \varepsilon_{ 0} h(t)>0$.

\medskip

Hence, the conclusion of both previous cases is that $ \Upsilon(t)>0$ on $(0, T]$. Then we necessarily have $ \left( \Upsilon \ast \rho\right)_{ t_{ \ast}}=\int_{ 0}^{ t_{ \ast}} \Upsilon(t_{ \ast}-u) \rho_{ u} {\rm d}u>0$, as the integral of the nonnegative continuous function $ u \mapsto \Upsilon(t_{ \ast}-u) \rho_{ u}$ that is strictly positive for $u$ in a neighborhood of $0$, by \eqref{eq:cond_rho}. Hence, by \eqref{eq:lambdaprime_Upsilon}, $ \lambda^{ \prime}_{ t_{ \ast}}>0$. It is therefore impossible to have $t_{ \ast}<T$ (otherwise we would have $ \lambda^{ \prime}_{ t_{ \ast}}=0$ by continuity). Hence, $t_{ \ast}= T$ and $ \lambda^{ \prime}_{ t}>0$ for all $t\in [0,T]$, for all $T>0$. This concludes the proof of Item~\ref{it:lambda_strictly_increasing} of Proposition~\ref{prop:monotone}.

\medskip
\noindent
\textit{Step 3:} Suppose that $ \Phi$ is nondecreasing and that \eqref{eq:rho_pos} holds. Fix some $ \varepsilon\in (0, 1]$ and introduce
\begin{equation}
\label{eq:rhoeps_Phieps}
\begin{cases}
\xi^{ \varepsilon}_{ t}&:= \xi_{ t}+ \varepsilon t,\ t\geq0\\
\Phi^{ \varepsilon}(x)&:= \Phi(x) + \varepsilon(1- e^{ -x}),\ x\geq0,\\
h^{ \varepsilon}(t)&:= h(t) + \varepsilon e^{ -t},\ t\geq 0. 
\end{cases}
\end{equation}
The function $ \xi^{ \varepsilon}$ is locally bounded. Define $ \lambda^{ \varepsilon}$ the corresponding solution to \eqref{eq:conv_gen_lambda} with initial condition $ \xi^{ \varepsilon}$ and driven by the kernels $ \Phi^{ \varepsilon}$ and $ h^{ \varepsilon}$. The corresponding $ \rho^{ \varepsilon}$ is $\rho^{ \varepsilon}_{ t}= \rho \left(h^\varepsilon, \xi^{ \varepsilon}, t\right) =  \rho( h, \xi, t)  + \varepsilon \left(1+ (h(t)+ \varepsilon e^{-t})(1- e^{-\xi_0})+ e^{-t} \Phi(\xi_0)\right)$ so that
\begin{align}
\label{eq:rhot_eps}
0< \varepsilon\leq\rho_{ t} + \varepsilon\leq \rho_{ t}^{ \varepsilon} \leq \rho_{ t} + \varepsilon \left(2+ \left\Vert h \right\Vert_{ \infty}+\Vert\Phi\Vert_{\infty}\right).
\end{align}
Moreover, $ \left(\Phi^{ \varepsilon}\right)^{ \prime}(x)= \Phi^{ \prime}(x) + \varepsilon e^{ -x}\geq \varepsilon e^{ -x}>0$, since $ \Phi$ is nondecreasing. Finally, $h^{ \varepsilon}(t)>0$ for all $t>0$. Therefore, one can apply Step~1 to $ \lambda^{ \varepsilon}$: we have $ \lambda^{ \varepsilon, \prime}_{ t}>0$ for all $t\geq0$. All that remains to prove is that, for fixed $t\geq0$, $ \lambda^{ \varepsilon, \prime}_{ t} \xrightarrow[ \varepsilon\to 0]{} \lambda^{ \prime}_{ t}$. Fix some $T>0$ and for $t\in [0, T]$, first look at $ \Delta_{ t}^{ \varepsilon}:= \lambda_{ t}^{ \varepsilon} - \lambda_{ t}$. We have, 
\begin{align}
 \left\vert \Delta_{ t}^{ \varepsilon} \right\vert&= \left\vert \Phi^{ \varepsilon} \left( \xi_{ t}^{ \varepsilon} + \int_{ 0}^{t} h^{ \varepsilon}(t-s) \lambda^{ \varepsilon}_{ s} {\rm d}s\right) - \Phi \left( \xi_{ t} + \int_{ 0}^{t} h(t-s) \lambda_{ s} {\rm d}s\right) \right\vert \nonumber\\
&\leq  \left\vert \Phi^{ \varepsilon} \left( \xi_{ t}^{ \varepsilon} + \int_{ 0}^{t} h^{ \varepsilon}(t-s) \lambda^{ \varepsilon}_{ s} {\rm d}s\right) - \Phi^{ \varepsilon} \left( \xi_{ t} + \int_{ 0}^{t} h^{ \varepsilon}(t-s) \lambda_{ s} {\rm d}s\right) \right\vert \label{aux:Deltaeps1}\\
&+\left\vert \Phi^{ \varepsilon} \left( \xi_{ t} + \int_{ 0}^{t} h^{ \varepsilon}(t-s) \lambda_{ s} {\rm d}s\right) - \Phi \left( \xi_{ t} + \int_{ 0}^{t} h^{ \varepsilon}(t-s) \lambda_{ s} {\rm d}s\right) \right\vert\label{aux:Deltaeps2}\\
&+\left\vert \Phi \left( \xi_{ t} + \int_{ 0}^{t} h^{ \varepsilon}(t-s) \lambda_{ s} {\rm d}s\right) - \Phi \left( \xi_{ t} + \int_{ 0}^{t} h(t-s) \lambda_{ s} {\rm d}s\right) \right\vert:= A_{ 1}+ A_{ 2} + A_{ 3}.\label{aux:Deltaeps3}
\end{align}
Note that $ \left\vert \Phi^{ \varepsilon} \right\vert_{ Lip}\leq \left\vert \Phi \right\vert_{ Lip}+ \varepsilon\leq \left\vert \Phi \right\vert_{ Lip}+ 1:= c_{ 0}$, so that the term $A_{ 1}$ in \eqref{aux:Deltaeps1} can be bounded by
\begin{equation*}
A_{ 1}\leq 
c_{ 0} \left(\varepsilon t+ \int_{ 0}^{t} h^{ \varepsilon}(t-s) \left\vert \Delta_{ s}^{ \varepsilon} \right\vert {\rm d}s\right).
\end{equation*}
The term $A_{ 2}$ in \eqref{aux:Deltaeps2} is uniformly bounded by $ \varepsilon$, by construction of $ \Phi^{ \varepsilon}$. Finally, by Lipschitz continuity of $ \Phi$, the term $A_{ 3}$ in \eqref{aux:Deltaeps3} is bounded as
\begin{align*}
A_{ 3} \leq
\left\vert \Phi \right\vert_{ Lip} \sup_{ s\in [0, T]} \lambda_{ s} \int_{ 0}^{t} \left\vert h^{ \varepsilon}(t-s) - h(t-s) \right\vert {\rm d}s \leq \varepsilon \left\vert \Phi \right\vert_{ Lip} \sup_{ s\in [0, T]} \lambda_{ s} .
\end{align*}
Hence, we deduce, for some $C_{ T}>0$
\begin{equation*}
\left\vert \Delta_{ t}^{ \varepsilon} \right\vert \leq \varepsilon C_{ T} + \int_{ 0}^{t} h^{ \varepsilon}(t-s) \left\vert \Delta_{ s}^{ \varepsilon} \right\vert {\rm d}s.
\end{equation*}
Applying Lemma~\ref{lem:Gronwall_eps} gives that for all $T>0$, there is some $C_{ T}>0$ such that 
\begin{equation}
\label{eq:bound_Delta_eps}
\sup_{ t\in [0, T]} \left\vert \Delta_{ t}^{ \varepsilon} \right\vert = \sup_{ t\in [0, T]} \left\vert \lambda_{ t}^{ \varepsilon} - \lambda_{ t}\right\vert \leq C_{ T} \varepsilon.
\end{equation}
Now turn to $ \Delta_{ t}^{ \varepsilon, \prime}= \lambda_{ t}^{ \varepsilon, \prime} - \lambda_{ t}^{ \prime}$. From \eqref{eq:derivative_lambda}, we have
\begin{multline}
\label{eq:Delta_prime}
\left\vert \Delta_{ t}^{ \varepsilon, \prime} \right\vert = \Bigg\vert \left(\rho_{ t}^{ \varepsilon}+\int_{0}^{t}h^{ \varepsilon}(t-u)\lambda^{ \varepsilon, \prime}_{u} {\rm d}u\right) \Phi^{ \varepsilon, \prime}\left(\xi_{t}^{ \varepsilon}+\int_{0}^{t}h^{ \varepsilon}(t-u)\lambda^{ \varepsilon}_{u} {\rm d}u\right)\\  - \left(\rho_{ t}+\int_{0}^{t}h(t-u)\lambda^{\prime}_{u} {\rm d}u\right)\Phi^{ \prime}\left(\xi_{t}+\int_{0}^{t}h(t-u)\lambda_{u} {\rm d}u\right) \Bigg\vert.
\end{multline}
By triangular inequality, \eqref{eq:Delta_prime} can be bounded by $(A)+(B)+(C)$, where the three terms are successively defined below. First, using \eqref{eq:rhot_eps} we write
\begin{align*}
(A)&:=  \left\vert\left(\rho_{ t}^{ \varepsilon}+\int_{0}^{t}h^{ \varepsilon}(t-u)\lambda^{ \varepsilon, \prime}_{u} {\rm d}u\right) - \left(\rho_{ t}+\int_{0}^{t}h(t-u)\lambda^{\prime}_{u} {\rm d}u\right) \right\vert \Phi^{ \varepsilon, \prime}\left(\xi_{t}^{ \varepsilon}+\int_{0}^{t}h^{ \varepsilon}(t-u)\lambda^{ \varepsilon}_{u} {\rm d}u\right)\\
&\leq c_{ 0}\left(\varepsilon\left(2+ \left\Vert h \right\Vert_{ \infty}+\Vert\Phi\Vert_{\infty}\right)+\int_{0}^{t}h^{ \varepsilon}(t-u) \left\vert \Delta_{ u}^{ \varepsilon, \prime} \right\vert {\rm d}u + \left\Vert \lambda^{ \prime} \right\Vert_{ \infty, [0, T]} \int_{ 0}^{t} \left\vert h^{ \varepsilon}(u) - h(u) \right\vert {\rm d}u\right)\\
&\leq C_{ T} \left( \varepsilon + \int_{0}^{t}h^{ \varepsilon}(t-u) \left\vert \Delta_{ u}^{ \varepsilon, \prime} \right\vert {\rm d}u\right).
\end{align*}
Second, let
\begin{align*}
(B)&= \left(\rho_{ t}+\int_{0}^{t}h(t-u)\lambda^{\prime}_{u} {\rm d}u\right) \left\vert \Phi^{ \varepsilon, \prime}\left(\xi_{t}^{ \varepsilon}+\int_{0}^{t}h^{ \varepsilon}(t-u)\lambda^{ \varepsilon}_{u} {\rm d}u\right) -  \Phi^{ \varepsilon, \prime}\left(\xi_{t}+\int_{0}^{t}h(t-u)\lambda_{u} {\rm d}u\right) \right\vert\\
& \leq \left\vert \Phi^{ \varepsilon, \prime} \right\vert_{ Lip} \left(\left\Vert \rho \right\Vert_{ \infty, [0, T]} + \left\Vert \lambda^{ \prime} \right\Vert_{ \infty, [0, T]}\left\Vert h \right\Vert_{ 1, [0, T]}\right) \left(\varepsilon t + \int_{ 0}^{t} h^{ \varepsilon}(t-u) \left\vert \Delta_{ u}^{ \varepsilon} \right\vert {\rm d}u + \varepsilon \left\Vert \lambda \right\Vert_{ \infty, [0, T]}\right).
\end{align*}
Noting that $ \left\vert \Phi^{ \varepsilon, \prime} \right\vert_{ Lip} \leq \left\vert \Phi^{ \prime} \right\vert_{ Lip}+ \varepsilon \leq \left\vert \Phi^{ \prime} \right\vert_{ Lip}+ 1<\infty$ and using the a priori bound \eqref{eq:bound_Delta_eps}, we obtain that, for some constant $C_{ T}>0$
\begin{align*}
(B) \leq C_{ T} \varepsilon.
\end{align*}
Finally, set
\begin{align*}
(C)&:= \left(\rho_{ t}+\int_{0}^{t}h(t-u)\lambda^{\prime}_{u} {\rm d}u\right)\left\vert  \Phi^{ \varepsilon, \prime}\left(\xi_{t}+\int_{0}^{t}h(t-u)\lambda_{u} {\rm d}u\right)-  \Phi^{ \prime}\left(\xi_{t}+\int_{0}^{t}h(t-u)\lambda_{u} {\rm d}u\right)\right\vert.
\end{align*}
By definition of $ \Phi^{ \varepsilon}$ in \eqref{eq:rhoeps_Phieps}, we see immediately that there is some other constant $C_{ T}>0$ such that 
\begin{equation*}
(C)\leq C_{ T} \varepsilon.
\end{equation*}
Putting all these estimates together into \eqref{eq:Delta_prime}, we obtain, for some other constant $C_{ T}>0$,
\begin{equation*}
\left\vert \Delta_{ t}^{ \varepsilon, \prime} \right\vert \leq C_{ T} \left( \varepsilon+\int_{0}^{t}h^{ \varepsilon}(t-u) \left\vert \Delta_{ u}^{ \varepsilon, \prime} \right\vert {\rm d}u\right).
\end{equation*}
Another application of Lemma~\ref{lem:Gronwall_eps}  gives that for all $T>0$, $ \sup_{ t\in [0, T]} \left\vert \Delta_{ t}^{ \varepsilon, \prime} \right\vert \leq C_{ T} \varepsilon$ for some different constant $C_{ T}$. Hence $ \lambda^{ \varepsilon, \prime}$ locally converges uniformly to $ \lambda^{ \prime}$ as $ \varepsilon>0$. Since $ \lambda_{ t}^{ \varepsilon, \prime}>0$, we have $ \lambda_{ t}^{ \prime}\geq 0$ for all $t\geq0$, hence the result.

\medskip

\textit{Step 4: }The other cases where $ \lambda$ is nonincreasing, resp. strictly decreasing can be treated similarly. Note that adapting the same approximation procedure to the nonincreasing case requires to introduce $ \xi^{ \varepsilon}_{ t}= \xi_{ t}- \varepsilon t$. We stress that the positivity of $ \xi$ is not required at any point of this proof.
 This concludes the proof of Proposition~\ref{prop:monotone}.

\subsection{Proof of Proposition~\ref{prop:generic_unstable_stable} \label{sec:proof_prop_generic}}

The difference $ \eta_{ t}:= \xi_{ t}^{ \ell_{ 0}}- \xi_{ t}^{ eq, \ell}$ writes $$ \eta_{ t}= \left( \Phi \left(\|h\|_1\ell_{ 0}\right)- \Phi(\|h\|_1\ell)\right) \int_{ t}^{+\infty} h(u) {\rm d}u + \left(\ell_{ 0} - \Phi(\Vert h\Vert_1\ell_{ 0})\right) \chi_{ t}$$ so that we indeed have $ \sup_{ t\geq0} \left\vert \eta_{ t} \right\vert \to 0$ as $ \ell_{ 0}\to \ell$.

For the instability results, we first consider the case \eqref{eq:cond_Phi_unstable1}. Let $ \lambda= \lambda^{ \ell_{ 0}}$ be the solution to \eqref{eq:conv_gen_lambda} driven by $ \xi^{ \ell_{ 0}}$ for $\ell_0\in(\ell,\ell+\delta)$.
By \eqref{eq:cond_Phi_unstable1}, we have $ \lambda_{ 0} = \Phi \left( \xi_{ 0}^{ \ell_{ 0}}\right)= \Phi(\|h\|_1\ell_{ 0}) > \ell_0>\ell$. To study the monotonicity of $\lambda$, compute the quantity in \eqref{eq:cond_rho} in Proposition~\ref{prop:monotone}:
\begin{align*}
(\xi_{ t}^{ \ell_{ 0}}) ^\prime + h(t) \Phi(\xi_{ 0}^{ \ell_{ 0}})= (\ell_{ 0}- \Phi(\|h\|_1\ell_{ 0})) \chi^{ \prime}_{ t},\quad t\ge0.
\end{align*} 
Since for all $ \ell_{ 0}\in (\ell, \ell+ \delta)$, $ \Phi(\|h\|_1\ell_{ 0})> \ell_{ 0}$ and $\chi$ is nonincreasing it holds that  $ (\xi_{ t}^{ \ell_{ 0}}) ^\prime + h(t) \Phi(\xi_{ 0}^{ \ell_{ 0}})\geq 0$ for all $t\geq0$. By Proposition~\ref{prop:monotone}, this means that $ t \mapsto\lambda_{ t}$ is nondecreasing. Therefore $ \lambda_{ t}\geq\lambda_{ 0}> \ell$, and we have the result. In the case \eqref{eq:cond_Phi_unstable2}, a similar argument applies, $ \lambda$ is  nonincreasing so that $ \lambda_{ t}\leq \lambda_{ 0}<\ell$.

For the stability results, we only prove the first point and leave the second to the reader. Assume \eqref{eq:cond_Phi_stable1} and let $ \lambda= \lambda^{ \ell_{ 0}}$ be the solution to \eqref{eq:conv_gen_lambda} driven by $ \xi^{ \ell_{ 0}}$ for $\ell_0\in(\ell,\ell+\delta)$. Adapting previous computations allows to derive that $ \lambda$ is noincreasing. It remains to prove that $ \lambda_{ t}\geq \ell$ for all $t\geq0$. Note that $ \lambda_{ 0}= \Phi \left(\|h\|_1\ell_0\right) > \Phi \left(\|h\|_1\ell\right)=\ell$ by strict monotonicity of $ \Phi$. Defining $ t_{ \ast}:= \inf \left\lbrace t>0, \lambda_{ t}< \ell=\Phi(\Vert h\Vert_1 \ell)\right\rbrace$, we have by continuity that $t_{ \ast}>0$. 
Since $ \Phi$ is strictly monotone and $ \lambda_{t}= \Phi \left(x_{ t}\right)$, where $x_t=\xi^{\ell_0}_t+\int_0^th(t-s)\lambda_s\rmd s$, we also have $ t_{ \ast}:= \inf \left\lbrace t>0, x_{ t}< \|h\|_1\ell\right\rbrace$. Suppose by contradiction that $t_{ \ast}<+\infty$. In this case, we have by continuity $x_{ t_{ \ast}}=\|h\|_1\ell$. Then, we have
\begin{align*}
    x_{ t_{ \ast}}- \|h\|_1\ell&= (\Phi(\|h\|_1\ell_0)-\ell)\int_{t_\ast}^\infty h(u)\rmd u+ \int_{ 0}^{t_{ \ast}} h(t_{ \ast}-s) \left(\lambda_{ s}- \ell\right) {\rm d}s+(\ell_0-\Phi(\|h\|_1\ell_0))\chi_t\geq0.
\end{align*} By strict monotonicity of $\Phi$, $\Phi(\|h\|_1\ell_0)-\ell)\int_{t_\ast}^\infty h(u)\rmd u= \Phi(\|h\|_1\ell_0)-\Phi(\|h\|_1\ell))\int_{t_\ast}^\infty h(u)\rmd u\geq 0$. Using that on $[0, t_{ \ast}]$, $ \lambda_{ s} \geq \ell$, we also have $\int_{ 0}^{t_{ \ast}} h(t_{ \ast}-s) \left(\lambda_{ s}- \ell\right) {\rm d}s\geq 0$. Finally, $(\ell_0-\Phi(\|h\|_1\ell_0))\chi_t\geq 0$ by \eqref{eq:cond_Phi_stable1}. This quantity is moreover strictly positive as $\chi$ is strictly decreasing starting from $\chi_0>0$ with $\chi_t \xrightarrow[t\to\infty]{}0$. Hence $ x_{ t_{ \ast}}>\|h\|_1 \ell$ which is a contradiction. Hence $t_{ \ast}= +\infty$ and $ \lambda_{ t}\geq \ell$ for all $ t\geq0$. Since $ \lambda$ is nonincreasing, $ \lambda_{ t}$ converges as $t\to\infty$ to some fixed-point of $ \Phi$ (Proposition~\ref{prop:conv_X_lambda}) within $[\ell, \ell+ \delta]$, that is necessarily equal to $\ell$ by \eqref{eq:cond_Phi_stable1}.

\subsection{Proof of Theorem~\ref{th:supercritical_unstable2}\label{sec:proof_th_critical_unstable2}}

We only consider the first case and leave the second to the reader. Denote by $ \lambda_{ t}= \Phi \left(x_{t}\right)$ the solution to \eqref{eq:conv_gen_lambda} with source term $ \xi_{ t}$ where $ x_{ t}= \xi_{ t}^{ eq, \ell} + \eta_{ t} + \int_{ 0}^{t} h(t-s) \lambda_{ s} {\rm d}s$. We have $ \lambda_{ 0}= \Phi \left(\Vert h\Vert_1\ell + \eta_{ 0}\right)> \Phi(\Vert h\Vert_1\ell)=\ell$ by strict monotonicity of $ \Phi$. Let $ t_{ \ast}= \inf \left\lbrace t>0,\ \lambda_{ t}< \ell\right\rbrace$. By continuity, $t_{ \ast}>0$ and on $[0, t_{ \ast}]$, $ \lambda_{ t}\geq \ell$. Suppose by contradiction that $ t_{ \ast}< +\infty$ so that by continuity, $ \lambda_{ t_{\ast}}= \ell$. Then, $x_{ t_{ \ast}}- \Vert h\Vert_1\ell= \eta_{ t_{ \ast}} + \int_{ 0}^{t_{ \ast}} h(t_{ \ast}-s) \left(\lambda_{ s} - \ell\right) {\rm d}s>0$, as $ \eta_{ t}\geq0$ and $ s \mapsto h(t_{ \ast}- s) \left(\lambda_{ s} -\ell\right)$ is continuous and verifies in $s=0$, $h(t_{ \ast}-0) \left(\lambda_{ 0}- \ell\right)>0$. Hence $x_{ t_{ \ast}}> \Vert h \Vert_1\ell$, which contradicts the fact that $ \lambda_{ t_{ \ast}}= \ell$ so that $ x_{ t_{ \ast}}= \Phi^{ -1} \left(\lambda_{ t_{ \ast}}\right)= \Phi^{ -1} \left(\ell\right)= \Vert h\Vert_1\ell $. Therefore $t_{ \ast}= +\infty$ and $ \lambda_{ t}\geq\ell$ for all $t\geq0$. \\

Proceed now by contradiction and suppose that $ \lambda_{ t} \to \ell$ as $t\to\infty$. By Proposition~\ref{prop:conv_X_lambda}, 
\begin{equation*}
    y_t:= \frac{1}{\Vert h\Vert_1}\left( \eta_{ t}+ \int_{ 0}^{t} h(t-s) \left(\lambda_{ s} -\ell\right){\rm d}s\right) + \ell\geq \ell
\end{equation*} converges as well to $\ell$ as $t\to\infty$. Therefore, there exists some $t_{ 0}\geq0$ such that for all $t\geq t_{ 0}$, $ y_{ t}\in [\ell,\ell+ \delta)$ where $(\ell- \delta, \ell+ \delta)$ is the domain where $ \Psi$ is convex-concave. Fix now some arbitrary $ \varepsilon>0$ such that $ \tau:= \Vert h\Vert_1\Phi^{ \prime}(\Vert h\Vert_1\ell)- \varepsilon>1$. Since $ \Vert h\Vert_1\Phi^{ \prime}\left(\Vert h \Vert_1 y_{ t}\right) \xrightarrow[ t\to\infty]{} \Vert h\Vert_1\Phi^{ \prime}(\Vert h\Vert_1\ell)>1$, taking $t_{ 0}\geq0$ even larger, one can furthermore suppose that $ \Psi^{\prime}(y_t)= \Vert h\Vert_1\Phi^{ \prime}( \Vert h\Vert_1 y_{ t})> \tau$ for all $t\geq t_{ 0}$. By concavity of $ \Psi$ on $[\ell, \ell+ \delta)$, one has for $t\geq t_{ 0}$, $  \tau \left(y_{ t}- \ell\right)< \Psi^{ \prime}(y_{ t})(y_{ t}-\ell)\leq \Psi(y_{ t}) - \Psi(\ell) $. Hence denoting by $ \nu_{ t}:= \lambda_{ t}- \ell= \Psi(y_{ t})- \Psi(\ell)\geq 0$, we obtain that, for $t\geq t_{ 0}$
\begin{align}
\label{eq:nu_ineq_conv}
\nu_{ t} \geq \tau (y_t-\ell)=\frac{\tau}{\Vert h\Vert_1} \left(\eta_{ t} + \int_{ 0}^{t} h(t-s) \nu_{ s} {\rm d}s\right) = \tilde{\eta}_{ t} + \int_{ 0}^{t} \tilde{ h}(t-s) \nu_{ s} {\rm d}s,
\end{align}
where we have used the notation $ \tilde{ \eta}_{ t}:= \tau \frac{\eta_{ t}}{\Vert h\Vert_1}$ and $ \tilde{ h}(t):= \tau \frac{h(t)}{\Vert h\Vert_1}$ so that $ \Vert \tilde{ h} \Vert_{ 1}= \tau>1$. Let us write 
\begin{equation}
\label{eq:Rt_nut}
R(t)= \nu_{ t}  - \int_{ 0}^{t-t_{ 0}} \tilde{ h}(t-s) \nu_{ s} {\rm d}s,\quad \forall t\ge t_{0}.
\end{equation}
We have from \eqref{eq:nu_ineq_conv} and the fact that $\tilde \eta_{s}\geq 0$ and $ \nu_{ s}\geq0$ on $[0, +\infty)$ that $R(t)\geq0$ for all $t\geq t_{ 0}$. Moreover, $t\mapsto R(t)$ is continuous, non identically $0$ and writing
\begin{equation*}
R(t)= \nu_{ t} - \int_{ t_{ 0}}^{t} \tilde{ h}(s) \nu_{ t-s} {\rm d}s,
\end{equation*}
we obtain since by assumption $ \nu_{ s} = \lambda_{ s}-\ell \xrightarrow[ s\to\infty]{}0$, by dominated convergence theorem that $R(t) \xrightarrow[ t\to\infty]{}0$. With this notation at hand, we obtain from \eqref{eq:Rt_nut} the equality: for all $t\geq t_{ 0}$
\begin{equation*}
\nu_{ t} = R(t) + \int_{ 0}^{t-t_{ 0}} \tilde{ h}(t-s) \nu_{ s} {\rm d}s.
\end{equation*}
We are now in position to use the result of Feller \cite[Theorem 3 (iii)]{feller1941} (see also \cite{Brauer1975}) to conclude that since $ \Vert \tilde{ h} \Vert_{ 1}>1$, we have that $ \nu_{ t} \xrightarrow[ t\to\infty]{}+\infty$, which contradicts the initial statement that $ \lambda_{ t} \xrightarrow[ t\to\infty]{}\ell$. Hence, $ \lambda_{ t}$ does not converge to $\ell$ as $t\to\infty$.

\subsection{Proof of Theorem~\ref{th:conv_empty}\label{sec:proof_th_empty}}
We use similar arguments as before. Consider $ \lambda= \lambda^{ \xi}$ the solution to \eqref{eq:conv_gen_lambda} with source term $ \xi$ satisfying Assumption~\ref{ass:xi_empty}. Note that under the present hypothesis, the solution $ \lambda$ is strictly increasing, by Proposition~\ref{prop:monotone}. Denote also by $x_{ t}= x_{ t}^{ \xi}= \xi_{ t} + \int_{ 0}^{t} h(t-u) \lambda_{ u} {\rm d}u$. Denote by $ t_{ 1}= \inf \left\lbrace t>0,\ x_{ t} >  \left\Vert h \right\Vert_{ 1}\ell_{ 1}\right\rbrace$. Note that $ x_{ 0}= \xi_{ 0} = \left\Vert h \right\Vert_{ 1} \ell_{ 0}< \left\Vert h \right\Vert_{ 1}\ell_{ 1}$, so $t_{ 1}>0$ by continuity. Suppose by contradiction that $t_{ 1}<+\infty$. By continuity, we then have $ x_{ t_{ 1}}= \left\Vert h \right\Vert_{ 1}\ell_{ 1}$. In particular, $ \lambda_{ t_{ 1}}= \Phi(x_{ t_{ 1}})= \Phi( \left\Vert h \right\Vert_{ 1}\ell_{ 1})=\ell_{ 1}$. Since $ \lambda$ is strictly increasing, by strict monotonicity of $ \Phi$, so is $x$. In particular, for all $u\in[0, t_{ 1})$, $x_{ u}< \left\Vert h \right\Vert_{ 1}\ell_{ 1}$ and $ \lambda_{ u}< \ell_{ 1}$. Consequently, $u\in[0,t_{1}]\mapsto h(t_{1}-u)(\lambda_{u}-\ell_{1})$ is continuous on $[0,t_{1}]$ and is nonpositive. It is also nonzero in both cases of \eqref{eq:cond_h}:  indeed, if $h(0)>0$, $h(t_{ 1}-u)$ remains strictly positive in a neighborhood of $t_{ 1}$ and since $ \lambda_{ u}- \ell_{ 1}<0$ on $V$, we have the result. In case $h(t)>0$ for all $t>0$, $h(t_{1}-u)(\lambda_{u}-\ell_{1})<0$ in a neighborhood of $u>0$. Hence, in both cases,
\begin{align*}
\int_{0}^{t_{1}} h(t_{1}-u)\lambda_{u} {\rm d}u<  \int_{0}^{t_{1}} h(t_{1}-u) {\rm d}u\ell_{1}.
\end{align*}
Hence, we deduce that $ \left\Vert h \right\Vert_{ 1}\ell_{ 1}= x_{ t_{ 1}}= \xi_{ t_{ 1}} + \int_{ 0}^{t_{ 1}} h(t_{ 1}-u) \lambda_{ u} {\rm d}u < \xi_{ t_{ 1}} + \ell_{ 1}\int_{ 0}^{t_{ 1}} h(u) {\rm d}u$, implying that
\begin{equation*}
\ell_{ 1} \int_{ t_{ 1}}^{+\infty} h(u) {\rm d}u < \xi_{ t_{ 1}}.
\end{equation*}
But this is in direct contradiction with Assumption~\ref{ass:xi_empty} since integrating $- \xi_{ t}^{ \prime} \leq h(t) \Phi \left(\xi_{ 0}\right)$ gives $ \xi_{ t} \leq \Phi(\xi_{ 0}) \int_{ t}^{+\infty} h(u) {\rm d}u= \Phi(\left\Vert h \right\Vert_{ 1}\ell_{ 0}) \int_{ t}^{+\infty} h(u) {\rm d}u\leq \Phi(\left\Vert h \right\Vert_{ 1}\ell_{ 1}) \int_{ t}^{+\infty} h(u) {\rm d}u= \ell_{ 1}\int_{ t}^{+\infty} h(u) {\rm d}u$. Hence $t_{ 1}=+\infty$ so that $ x_{ t}\leq \left\Vert h \right\Vert_{ 1} \ell_{ 1}$ for all $t\geq0$. By monotonicity of $ \Phi$, $ \lambda_{ t} = \Phi \left(x_{ t}\right)\leq \Phi \left(\left\Vert h \right\Vert_{ 1}\ell_{ 1}\right)=\ell_{ 1}$. In particular $ \lambda$ is bounded and we conclude from Proposition~\ref{prop:conv_X_lambda} that $ \lambda_{ t} $ converges in an increasing way to some fixed point $ \ell= \Phi \left(\left\Vert h \right\Vert_{ 1} \ell\right)$. Since $ \lambda\leq \ell_{ 1}$ and $\ell_{ 1}$ is the lowest of such fixed-point, $ \ell=\ell_{ 1}$. The rest of the proof consists in applying Theorem~\ref{th:conv_rate} in the case \eqref{eq:rates_lambda_exp} as $ \xi^\emptyset$ satisfies $\left(\mathcal{P}_{a, A}\right)$ of Assumption~\ref{ass:decay} for any rate $a>0$.

\appendix

\section{Global estimates for boundedness}
\label{sec:appendix_boundedness}
We gather in this section all the results and proofs related to the boundedness of solutions to \eqref{eq:conv_gen_lambda} that is the subject of the discussion in Sections~\ref{sec:criticality_intro} and~\ref{sec:boundedness_intro}.

\subsection{Proof of Proposition~\ref{prop:subcritical_Lip}}
\label{sec:strong_subcrit} Uniqueness of a solution $\ell$ to \eqref{eq:fixed_point_lambda} follows directly from the contractivity hypothesis \eqref{eq:subcritical_Lip}. By Lipschitz continuity of $ \Phi$,
  \begin{align*}
  \lambda_{ t} &\leq \Phi(0) + \left\vert \Phi \right\vert_{ Lip} \left( \left\Vert \xi \right\Vert_{ \infty} + \int_{ 0}^{t} \left\vert h(t-s) \right\vert \lambda_{ s} {\rm d}s\right)\leq \Phi(0) + \left\vert \Phi \right\vert_{ Lip}  \left\Vert \xi \right\Vert_{ \infty} + \left\vert \Phi \right\vert_{ Lip} \left\Vert h \right\Vert_{ 1}
 \sup_{ s\in [0, t]} \lambda_{ s}   
 \end{align*}
 so that by \eqref{eq:subcritical_Lip}, $ \sup_{ s \geq0} \lambda_{ s} \leq \frac{ \Phi(0) + \left\vert \Phi \right\vert_{ Lip}  \left\Vert \xi \right\Vert_{ \infty}}{ 1- \left\vert \Phi \right\vert_{ Lip} \left\Vert h \right\Vert_{ 1}}$: $ \lambda$ is bounded. Prove now that $ \lambda_{ t}$ converges to $\ell$ as $t\to\infty$. Let $h\geq0$ and $s \geq t$. Denote by $ \Delta(s,h)= \left\vert \lambda_{ s+h}- \lambda_{ s} \right\vert$. We have,
 \begin{align}
 \Delta(s,h)&\leq \left\vert \Phi \right\vert_{ Lip} \left\vert \xi_{ s+h} - \xi_{ s}\right\vert + \left\vert \Phi \right\vert_{ Lip} \int_{ s}^{s+h} \left\vert h(u) \right\vert \lambda_{ s+h-u} {\rm d}u + \left\vert \Phi \right\vert_{ Lip}\int_{ 0}^{s} \left\vert h(u) \right\vert \left\vert \lambda_{ s+h-u} - \lambda_{ s-u}\right\vert {\rm d}u\nonumber\\
 &\leq \left\vert \Phi \right\vert_{ Lip} \left\vert \xi_{ s+h} - \xi_{ s}\right\vert + \left\vert \Phi \right\vert_{ Lip} \left\Vert \lambda \right\Vert_{ \infty}\int_{ s}^{s+h} \left\vert h(u) \right\vert {\rm d}u + \left\vert \Phi \right\vert_{ Lip}\int_{ 0}^{s} \left\vert h(s-u) \right\vert \left\vert \lambda_{ u+h} - \lambda_{u}\right\vert {\rm d}u .\label{aux:Delta_s_h}
 \end{align}
 Concentrate on the last integrand: 
 \begin{align*}
  \int_{ 0}^{s} \left\vert h(s-u) \right\vert \left\vert \lambda_{ u+h} - \lambda_{u}\right\vert {\rm d}u&=  \int_{ 0}^{t} \left\vert h(s-u) \right\vert \left\vert \lambda_{ u+h} - \lambda_{u}\right\vert {\rm d}u +  \int_{ t}^{s} \left\vert h(s-u) \right\vert \left\vert \lambda_{ u+h} - \lambda_{u}\right\vert {\rm d}u,\\
  &\leq 2 \left\Vert \lambda \right\Vert_{ \infty} \int_{ s-t}^{s} \left\vert h(u) \right\vert {\rm d}u + \sup_{ u\geq t} \Delta(u,h) \int_{ 0}^{s-t} \left\vert h(v) \right\vert {\rm d}v.
 \end{align*}
 Putting this estimate into \eqref{aux:Delta_s_h} and taking $\limsup_{ s\to \infty}$ on both sides, we obtain
 \begin{align*}
 \limsup_{ s\to\infty} \left\vert \lambda_{ s+h} - \lambda_{ s}\right\vert \leq \limsup_{ u\to\infty} \left\vert \lambda_{ u+h} - \lambda_{ u}\right\vert \left\Vert \Phi \right\Vert_{ Lip}\left\Vert h \right\Vert_{ 1}.
 \end{align*}
Since $\lambda$ is bounded, $\limsup_{ s\to\infty} \left\vert \lambda_{ s+h} - \lambda_{ s}\right\vert<+\infty$ and by \eqref{eq:subcritical_Lip}, one has  $\limsup_{ s\to\infty} \left\vert \lambda_{ s+h} - \lambda_{ s}\right\vert=0$ for all $h\geq0$ so that $ \lambda_{ t}$ converges as $t\to\infty$ and its limit is necessarily the unique fixed-point $\ell$ to \eqref{eq:fixed_point_lambda}.

\subsection{Proof of Proposition~\ref{prop:global_subcrit}}
\label{sec:global_subcrit}
Choose some $K>0$ sufficiently large and $ r<1$ so that $ \Phi(x) \left\Vert h \right\Vert_{ 1}< r  \left\vert x \right\vert$ as long as $ \left\vert x \right\vert\geq K$. For $x_{ t}= \xi_{ t} + \int_{ 0}^{t} h(t-s) \lambda_{ s} {\rm d}s$, we have
\begin{align*}
\left\vert x_{ t} \right\vert&\leq \left\vert \xi_{ t} \right\vert + \int_{ I_{ 1}} \left\vert h(t-s) \right\vert \Phi \left(x_{ s}\right) {\rm d}s +  \int_{ I_{ 2}} \left\vert h(t-s) \right\vert \Phi \left(x_{ s}\right) {\rm d}s,
\end{align*}
where $I_{ 1}= \left\lbrace u\in [0, t],\ \left\vert x_{ u} \right\vert>K\right\rbrace$ and $I_{ 2}= \left\lbrace u\in [0, t],\ \left\vert x_{ u} \right\vert \leq K\right\rbrace$. By continuity of $ \Phi$, there exists some $L>0$ such that $\sup_{ u\in I_{ 2}} \Phi(x_{ u})\leq L$. Therefore,
\begin{align*}
\left\vert x_{ t} \right\vert&\leq \left\Vert \xi \right\Vert_{ \infty} + r \int_{ I_{ 1}} \left\vert h(t-s) \right\vert \left\vert x_{ s} \right\vert{\rm d}s +  L \int_{ I_{ 2}} \left\vert h(t-s) \right\vert{\rm d}s,\\
&\leq \left\Vert \xi \right\Vert_{ \infty} + r \sup_{ s\in [0, t]} \left\vert x_{ s} \right\vert\int_{ I_{ 1}} \left\vert h(t-s) \right\vert {\rm d}s +  L \int_{ I_{ 2}} \left\vert h(t-s) \right\vert{\rm d}s,\\
&\leq \left\Vert \xi \right\Vert_{ \infty} + r \sup_{ s\in [0, t]} \left\vert x_{ s} \right\vert \left\Vert h \right\Vert_{ 1} +  L \left\Vert h \right\Vert_{ 1}.
\end{align*}
Taking the supremum on $[0, T]$ on both sides, we obtain
\begin{equation*}
\sup_{ t\in[0, T]} \left\vert x_{ t} \right\vert \leq \frac{  \left( \left\Vert \xi \right\Vert_{ \infty} + L \left\Vert h \right\Vert_{ 1}\right)}{ 1- r \left\Vert h \right\Vert_{ 1}}.
\end{equation*}
As this bound is uniform in $T$, we have the result: $(x_{ t})_{ t\geq0}$ is bounded and so is $ \lambda$.

\subsection{Optimality of the global subcritical condition}
\label{sec:optimal_boundedness}
The main contribution of this section is to show that Condition \eqref{eq:global_subcrit} is essentially optimal for the boundedness of $ \lambda$, at least in the case where $h$ is nonnegative. Start with a result in the supercritical case:
\begin{proposition}
\label{prop:global_supercrit}
Suppose that $ \Phi$ is nonnegative Lipschitz continuous, $ \mathcal{ C}^{ 1}$ with $\Phi^\prime$ Lipschitz continuous and $ \Phi^{ \prime}>0$. Suppose that $h$ is nonnegative, continuous and integrable on $[0, +\infty)$ and verifies \eqref{eq:cond_h}. Suppose that the set of fixed-points of $ \Phi$, $ \mathcal{ P}:= \left\lbrace\ell\geq0, \Phi \left( \left\Vert h \right\Vert_{ 1} \ell\right)=\ell\right\rbrace$ is bounded by above. Suppose also
\begin{equation}
\label{eq:global_supercrit}
\limsup_{ x\to +\infty} \frac{ \Phi(x)}{ x} \left\Vert h \right\Vert_{ 1}>1.
\end{equation}
Consider a source term $ \xi$ of the form  $\xi_{ t}= \ell_{ 0}\int_{ t}^{ +\infty} h(u) {\rm d}u,\ t\geq0$ as in \eqref{eq:xi_mono}. Then, there exists some sufficiently large $ \ell_{ 0}$ such that the corresponding solution $ \lambda= \lambda^{ \xi}$ to \eqref{eq:conv_gen_lambda} is such that $ \lim_{ t\to+\infty} \lambda_{ t}= +\infty$.
\end{proposition}
Of course, Condition \eqref{eq:global_supercrit} can be  compatible with the existence of convergent solutions $ \lambda$ (it is for example the case when \eqref{eq:fixed_point_lambda_pos} has a fixed-point $\ell$ such that $ \left\Vert h \right\Vert_{ 1} \Phi^{ \prime} \left( \left\Vert h \right\Vert_{ 1}\right)<1$ in which case Theorems~\ref{th:stab_stationary} and~\ref{th:conv_rate} apply). 
\begin{proof}[Proof of Proposition~\ref{prop:global_supercrit}]
By \eqref{eq:global_supercrit}, there exists some $K>0$ and some $ r>1$ such that for all $ \ell_{ 0}\geq K$, $ \Phi(\ell_{ 0} \left\Vert h \right\Vert_{ 1}) \geq r \ell_{ 0} $. Fix then some $ \ell_{ 0} > \max \left(\sup \mathcal{ P}, K\right)$. For such $ \ell_{ 0}$ and $ \xi_{ t}$ given by \eqref{eq:xi_mono}, compute $ \rho(h, \xi, t)$ given by \eqref{eq:rho}: using Remark~\ref{rem:monotone_empty_case} (see in particular \eqref{eq:rho_particular_case}) we have
  \begin{align*}
  \rho(h, \xi, t)= \left( \Phi \left( \left\Vert h \right\Vert_{ 1} \ell_{ 0}\right) - \ell_{ 0}\right)h(t)\geq \left(r-1\right) \ell_{ 0}h(t).
  \end{align*}
 By \eqref{eq:cond_h}, Condition \eqref{eq:cond_rho} is satisfied so that, using Proposition~\ref{prop:monotone}, we see that $ \lambda$ is strictly increasing, starting from $ \lambda_{ 0}= \Phi\left(\ell_{ 0} \left\Vert h \right\Vert_{ 1}\right)\geq r \ell_{ 0}> \sup \mathcal{ P}$. Therefore $ \lambda$ is not bounded (if it were bounded, it would converge, necessarily to some element of $ \mathcal{ P}$). Hence $ \lambda_{ t} \to +\infty$ as $t\to +\infty$.
\end{proof}

We then address the possible behavior of solutions $ \lambda$ to \eqref{eq:conv_gen_lambda} in the critical case:
\begin{equation}
\label{eq:global_crit}
\limsup_{ x\to\infty} \frac{ \Phi(x)}{ x} \left\Vert h \right\Vert_{ 1}=1.
\end{equation}
\begin{example}
\label{ex:no_fixed_point_crit}
If $ \Phi$ is such that $ \Phi(x) \left\Vert h \right\Vert_{ 1}>x$ for all $x$, but such that \eqref{eq:global_crit} is nonetheless true (\textit{e.g.,} $ \Phi(x)= \frac{ 1}{ \left\Vert h \right\Vert_{ 1}}\log(1+e^{ x})$, $x>0$). We see from Proposition~\ref{prop:nofixedpoint} that \emph{any solution} $ \lambda$ to \eqref{eq:conv_gen_lambda} diverge to $+\infty$ in this case.
\end{example}
Example~\ref{ex:no_fixed_point_crit} is particular in the sense that the impossibility of convergence for any solution to \eqref{eq:conv_gen_lambda} is imposed by the absence of any fixed-point solution to \eqref{eq:fixed_point_lambda}. We provide in Proposition~\ref{prop:example_infty} below an example of some $ \Phi$ satisfying \eqref{eq:global_crit} with a unique fixed-point solution to \eqref{eq:fixed_point_lambda} such that there exists some $ \xi$ such that $ \lambda= \lambda^{ \xi}$ verifies $ \lambda_{ t} \xrightarrow[ t\to \infty]{}\infty$. In particular, this gives a situation in the critical case \eqref{eq:global_crit} where there is coexistence between bounded convergent solutions $ \lambda$ to \eqref{eq:conv_gen_lambda} and diverging ones.
\begin{proposition}
\label{prop:example_infty}
Consider the following kernel
\begin{equation}
\label{eq:example_Phi}
\Phi \left(x\right) = x - e^{ -x} + e^{ - \frac{ 3}{ 2}x} \sqrt{ 2+x},\ x\geq0.
\end{equation}
Take $h(u)= e^{ -u}$, $u\geq0$. Then, $ \Phi$ is Lipschitz continuous, strictly increasing, with $ \Phi(0)>0$, a unique fixed-point $\ell$ such that $ \left\Vert h \right\Vert_{ 1}\Phi^{ \prime} \left( \left\Vert h \right\Vert_{ 1}\ell\right)<1$ and $ \lim_{ x\to \infty} \frac{ \Phi(x)}{ x} \left\Vert h \right\Vert_{ 1}=1$. There exists some sufficiently large $ \xi_{ 0}$ and some $ \xi= (\xi_{ t})_{ t\geq 0}$ with $\xi_{ t=0}= \xi_{ 0}$ such that the solution $ \lambda= \lambda^{ \xi}$ to \eqref{eq:conv_gen_lambda} with source term $ \xi$ satisfies 
\begin{equation*}
\lambda_{ t} \xrightarrow[ t\to\infty]{}+\infty.
\end{equation*}
\end{proposition}

\begin{proof}[Proof of Proposition~\ref{prop:example_infty}]
The fact that $ \Phi$ is Lipschitz continuous, strictly increasing, with $ \Phi(0)>0$, with a unique fixed-point $\ell$ such that $ \left\Vert h \right\Vert_{ 1}\Phi^{ \prime}( \kappa\ell)<1$ and $ \lim_{ x\to \infty} \frac{ \Phi(x)}{ x} \left\Vert h \right\Vert_{ 1}=1$ is basic analysis.

\noindent \textit{Step 0: definition of $ \xi$ and $ \lambda= \lambda^{ \xi}$.} Let $a>0 \mapsto \Psi(a):= \log \left( \frac{ e^{ -2a}}{ a} \left( \frac{ 1}{ 2a}+1\right)+1\right)+2a$. It is easy to see that $ \Psi^{ \prime}(a) \xrightarrow[ a\to\infty]{} 2$ so there is some $a_{ \ast}>0$ so that $ \Psi$ is strictly increasing on $[a_{ \ast}, +\infty)$ with $ \lim_{ +\infty} \Psi=+\infty$. In particular, for any $ \xi_{ 0}\geq \Psi \left(a_{ \ast}\right)$, there is some $a\geq a_{ \ast}$ such that 
 \begin{equation}
 \label{eq:xi0_a}
 \xi_{ 0}= \Psi \left(a\right).
 \end{equation}
 For the rest of the proof, we fix once and for all such $ \left(\xi_{ 0}, a\right)$ verifying \eqref{eq:xi0_a}. Once this fixed, we define
 \begin{equation}
 \label{eq:xit_a}
 \xi_{ t}= e^{ -t} \xi_{ 0} + a\int_{ 0}^{t} \frac{ e^{ -(t-s)}}{ \sqrt{ 1+s}} {\rm d}s=e^{ -t} \xi_{ 0} + a\int_{ 0}^{t} \frac{ e^{ -s}}{ \sqrt{ 1+(t-s)}} {\rm d}s.
 \end{equation}
 We easily see that $ \xi$ is $ \mathcal{ C}^{ 1}$ and a direct application of dominated convergence theorem shows that $ \xi_{ t} \xrightarrow[ t\to\infty]{}0$. From now on, we define as $ \lambda= \lambda^{ \xi}$ the solution to \eqref{eq:conv_gen_lambda} with source term $ \xi$. In particular, $ \lambda_{ t}= \Phi \left(x_{ t}\right)$ with $ x_{ t}= \xi_{ t}+ \int_{ 0}^{t} e^{ -(t-s)} \Phi(x_{ s}) {\rm d}s$. 
 
 \medskip

\noindent \textit{Step 1: Comparison result.}
Recalling \eqref{eq:example_Phi} and since $r(x):= e^{ - \frac{ 3}{ 2}x} \sqrt{ 2+x}>0$, we have
 \begin{equation*}
 \Phi(x)> \Phi_{ 0}(x):= x- e^{ -x},\ x\geq0.
 \end{equation*}
 A direct application of the comparison result of Proposition~\ref{prop:comparison} gives
 \begin{equation}
 \label{eq:x_vs_y}
 x_{ t} \geq y_{ t},\ t\geq 0
 \end{equation}
 where $y= \left(y_{ t}\right)_{ t\geq0}$ solves
 \begin{equation*}
 y_{ t}= \xi_{ t}+ \int_{ 0}^{t} e^{ -(t-s)} \Phi_{ 0}(y_{ s}) {\rm d}s.
 \end{equation*}
 
 \medskip

\noindent \textit{Step 2: Explicit expression for $y$.} One can take advantage of the exponential nature of $h$: differentiating the previous identity, we obtain that $y_{ t}$ is the solution to the following differential equation
 \begin{equation*}
 y^{ \prime}_{ t}= \Phi_{ 0} \left(y_{ t}\right) - y_{ t} + \xi_{ t} + \xi^{ \prime}_{ t},\ t\geq0,\ \text{ with } y_{ 0}= \xi_{ 0}.
 \end{equation*}
Notice here that $ \xi$ defined through \eqref{eq:xit_a} has been precisely chosen so that $ \xi_{ t} + \xi_{ t}^{ \prime}= \frac{ a}{ \sqrt{ 1+t}}$, $t\geq0$. Hence, the previous equation boils down to
 \begin{equation}
 \label{eq:yt}
 y^{ \prime}_{ t}= - e^{ - y_{ t}} + \frac{ a}{ \sqrt{ 1+t}},\ t\geq0.
 \end{equation}
A simple change of variables $y_{ t}= \log(u_{ t})$ turns \eqref{eq:yt} into a linear differential equation in $u_{ t}$: $u^{ \prime}_{ t}-  u_{ t} \frac{ a}{ \sqrt{ 1+t}}=-1$. From this, we see that any solution to \eqref{eq:yt} is explicitly solved as
\begin{equation}
\label{eq:yt_expr}
y_{ t}= \log \left(\frac{ e^{ -2a \sqrt{ 1+t}}}{ a} \left( \frac{ 1}{ 2a} + \sqrt{ 1+t}\right)+c\right) +2a \sqrt{ 1+t}
\end{equation}
where $c$ is an arbitrary constant. Recalling that $y_{ 0}= \xi_{ 0}$ and that $ \xi_{ 0}$ and $a$ have been chosen through \eqref{eq:xi0_a}, the constant $c$ is equal to $1$. 

\medskip

\noindent \textit{Step 3: Conclusion.} We directly see from \eqref{eq:yt_expr} that $y_{ t} \xrightarrow[ t\to\infty]{}+\infty$ and we conclude from \eqref{eq:x_vs_y} that $ x_{ t} \xrightarrow[ t\to\infty]{}+\infty$ too. By strict monotonicity of $ \Phi$ so does $ \lambda$: $ \lambda_{ t} \xrightarrow[ t\to \infty]{}+\infty$.
\end{proof}

\section{Proof of Theorem \ref{th:TCL}\label{sec:prf_TCL}}  

Note that we have assumed that $ \xi$ is bounded such that $ \xi_{ t} \xrightarrow[ t\to\infty]{}0$. Proposition \ref{prop:subcritical_Lip} applies:  $\lambda=\lambda^{\xi}$ solution to \eqref{eq:conv_gen_lambda} with source term $\xi$ converges to the unique fixed-point $\ell$ to \eqref{eq:fixed_point_intro}: $\lambda_t\xrightarrow[t\to\infty]{}\ell$. Since $\Vert h\Vert_1 \vert \Phi \Vert_{Lip}<1$, there exists $\varepsilon_0$ small enough such that \eqref{eq:rate_tauT} is satisfied and some $t_0\ge 0$ such that \eqref{eq:prox_lambda_ell} holds. Then, Theorem \ref{th:conv_rate} applies and allows to write 
\begin{align}\label{eq:cv_mt}
    \left|\frac{m_t}{t}-\ell\right|&\le\frac1t\int_0^t|\lambda_s-\ell |\rmd s\le\frac{(\|\lambda\|_\infty+\ell)\sigma(t_0)}{t}+C\frac{\int_{\sigma(t_0)}^t(\log s)^bs^{-b}\rmd s}{t}\le C'(\log t)^b t^{-b}.
\end{align}This in particular implies that $m_t/t\xrightarrow[t\to\infty]{}\ell$.

For any $N\geq 1$, the mean-field Hawkes process $\left(Z^{1,N}, \ldots, Z^{N,N}\right)$ whose conditional intensity is given by \eqref{eq:Hawkes_N} has the following standard representation  (see \cite[Lemma 3]{MR1411506}): consider on a filtered probability space $( \Omega, \mathcal{ F}, \left(\mathcal{ F}_{ t}\right)_{ t\geq 0}, \mathbf{ P})$ an family of i.i.d. Poisson measures $(\pi_{i}(\rmd s,\rmd z), i\in\{1,\ldots,N\})$ with intensity measure $\rmd s\times \rmd z$ on $[0,\infty)\times[0,\infty)$.
Then defining the family of c\`adl\`ag $(\mathcal{ F}_{ t})_{ t\geq0}$ point processes $(Z_{t}^{ i,N})_{ t\geq0, i=1,\ldots, N}$ as given by
\begin{equation}
\label{eq:Hawkes}
Z_{t}^{i,N}= \int_{0}^{t}\int_{0}^{\infty}\mathbf{1}_{z\le \lambda_{ s}^{ i}}\pi_{i}(\rmd s,\rmd z), i=1,\ldots, N,
\end{equation}
it is easy to see that each $Z^{i,N}$ given by \eqref{eq:Hawkes} has indeed for intensity  $ \lambda^{ i,N}_t=\lambda_{N,t}$, $i=1,\ldots, N$, given by \eqref{eq:Hawkes_N}. These $\pi_i$ being fixed, following \cite{MR3449317}, we couple $\left(Z^{1,N}, \ldots, Z^{N,N}\right)$ given by \eqref{eq:Hawkes} with i.i.d. copies $\left(\bar Z^1, \ldots, \bar Z^N\right)$ of inhomogeneous Poisson processes with intensity $\lambda_t$ solution to \eqref{eq:conv_gen_lambda} in the following way:
\begin{equation}
\label{eq:Hawkes_bar}
\bar Z^i_{t}= \int_{0}^{t}\int_{0}^{\infty}\mathbf{1}_{z\le \lambda_s}{ \pi_i}(\rmd s,\rmd z).
\end{equation}

In other words, the coupling between \eqref{eq:Hawkes} and \eqref{eq:Hawkes_bar} has been defined in such a way that both processes are built on the same underlying Poisson measure $\pi_i$.  Thus, the following result \eqref{eq:coupling} below can be proved following the lines of the proof of \cite[Theorem 8(ii)]{MR3449317} and \cite[Remark 9(a)]{MR3449317} therein. The addition of the source term $\xi_N$ entails some minor modifications that do not change the calculations of the previous references.
The details are omitted here.
It holds that  
\begin{align}\label{eq:coupling}
\sup_{ i=1, \ldots, N} \mathbb{ E} \left[\sup_{ s\in [0, t]} \left\vert Z_{s}^{i,N} - \bar Z_{ s}^{i}\right\vert\right] \leq \frac{(C_\xi+\Vert \lambda \Vert_{ \infty}^{ \frac{ 1}{ 2}} \Vert h\Vert_{ 2})}{1-\|h\|_1|\Phi|_{Lip}} \frac{t}{ \sqrt{N}}:=\tilde C \frac{t}{ \sqrt{N}},
\end{align} 
Consider the decomposition: for $u\in[0,1]$
\begin{align}\label{eq:Dec_TCL}
\frac{Z^{i,N}_{ut}-m_{ut}}{m_{t}}&=\left(\frac{\overline Z^{i}_{ut}-m_{ut}}{m_{t}}\right)+m_{t}^{-1}\left(Z^{i,N}_{ut}-\overline Z^i_{ut} \right).
\end{align}
It follows from \eqref{eq:coupling} 
\begin{align}\label{eq:Res_TCL}
m_{t}^{-1} \mathbb{ E} \left(\sup_{u\in[0,1]}\left|Z^{i,N}_{ut}-\overline Z^i_{ut} \right|\right)&\le \tilde C\frac{t}{m_{ t}\sqrt {N}}. 
\end{align}  To study the first term in the left hand side of \eqref{eq:Dec_TCL}, define $M_{ut}:=\left(\overline Z^i_{ut}-m_{ut}\right)$, a centered martingale with jumps that are all of size 1, thus uniformly bounded. Moreover,   $[M,M]_{ut}={\overline Z^i_{ut}},$ as $\overline Z^i$ is an inhomogeneous Poisson process with intensity function  $\Phi\left(\xi_.+\int_0^.h(.-u)\rmd m_u\right)$. Therefore, $ \mathbb{ E} \left(M_{ut}^{ 2}\right)= m_{ ut}$. In particular, making $u=1$ in \eqref{eq:Dec_TCL}, we obtain that
\begin{align}
\label{eq:L1Uigen}
\mathbb{ E} \left( \left\vert \frac{ Z_{ t}^{ i, N}}{ m_{ t}} -1 \right\vert\right) \leq \frac{ 1}{ \sqrt{ m_{ t}}} + \frac{\tilde C t}{m_{ t}\sqrt {N}}
\end{align}
The proof of inequality \eqref{eq:L1hatell} follows then directly from identity \eqref{eq:Dec_ellT1} and inequalities \eqref{eq:L1Uigen} and \eqref{eq:cv_mt}:
\begin{align*}
\mathbb{ E} \left( \left\vert \hat{\ell}_{ N, t} -\ell\right\vert\right)& \leq \frac{ m_{ t}}{ t} \mathbb{ E} \left(\left\vert \frac{ Z_{ t}^{ i, N}}{ m_{ t}} -1 \right\vert\right) + \left\vert \frac{ m_{ t}}{ t}  - \ell\right\vert\leq \frac{ \sqrt{ m_{ t}}}{ t} + \frac{ \tilde{ C}}{ \sqrt{ N}} + \left\vert \frac{ m_{ t}}{ t}  - \ell\right\vert.
\end{align*}
From \eqref{eq:cv_mt}, we obtain that $ m_{ t}\leq \ell t + C (\log t)^{ b} t^{ -b+1}$. Injecting this in the latter bound allows to get the non asymptotic control \eqref{eq:L1hatell}.

Turning now to the second point of Theorem~\ref{th:TCL}, we obtain that $$ \frac{ M_{ ut}}{ ut}=\frac{\overline Z^i_{ut}-m_{ut}}{ut}\xrightarrow[t\to \infty]{\PP} 0.$$ It follows that $$\left(\frac{\overline Z^i_{ut}}{m_{t}}-u\right)=\frac{ut}{m_t}\frac{\overline Z^i_{ut}-m_{ut}}{ut}+\left(\frac{m_{ut}}{m_t}-u\right)\xrightarrow[t\to \infty]{\PP} 0.$$ Using \eqref{eq:cv_mt} we derive that $m_t^{-1}[M,M]_{ut}\xrightarrow[t\to \infty]{\PP} u$ for all $u\in[0,1].$ 
Applying a functional martingale limit theorem, 
see \cite[Lemma 12]{MR3449317} and its proof (see also \cite[Theorem VIII-3.11]{JS}) we get the following: 
\begin{align*} \left(\sqrt{m_t}\frac{\overline Z^i_{ut}-m_{ut}}{m_t}\right)_{u\in[0,1]}\xrightarrow[t\to \infty]{d} B,\end{align*} where $B=(B_u)_{u\in[0,1]}$ is a standard Brownian motion. Using the coupling result, this functional limit result is then transferred to $Z^i_t$ using \eqref{eq:Dec_TCL} and \eqref{eq:Res_TCL}
and leads to the first limit result \eqref{eq:Func_CLT}.\\

Secondly, \eqref{eq:Res_TCL_est} is easily derived from \eqref{eq:Dec_ellT} combined with \eqref{eq:Func_CLT} and \eqref{eq:cv_mt}: 
from \eqref{eq:cv_mt} and \eqref{eq:Func_CLT}  with $u=1$ we immediately get $I_{ N, t}(1)\xrightarrow[{(t,N)\to(\infty,\infty)},\ {\frac{t}{N}\to0}]{d}\mathcal{N}(0,\ell).$ The remaining term $I_{ t}(2)$ goes to 0 using \eqref{eq:cv_mt} as $b>\frac12$, leading to the desired result.

\section{Auxiliary results}
\label{sec:auxiliary}
\subsection{A Gr\"onwall Lemma for convolution}

The following lemma is a perturbed version of \cite[Lemma~23]{MR3449317}:
\begin{lemma}
\label{lem:Gronwall_eps}
Let $h$ be locally integrable on $[0, +\infty)$. Let $(h_{ \varepsilon})_{ \varepsilon>0}$ such that $ h_{ \varepsilon} \xrightarrow[ \varepsilon\to 0]{} h$ in $L^{ 1}_{ loc}$. Let $ \left(g_{ \varepsilon}\right)_{ \varepsilon>0}$ a family of nonnegative locally bounded functions on $[0, +\infty)$. If $(u_{ \varepsilon})_{ \varepsilon>0}$ is a family of nonnegative locally bounded functions on $[0, +\infty)$ such that
\begin{equation*}
u_{ \varepsilon, t} \leq g_{ \varepsilon, t} + \int_{ 0}^{t} \left\vert h_{ \varepsilon}(t-s) \right\vert u_{ \varepsilon, s} {\rm d}s,\ t\geq0,
\end{equation*}
then, for all $T>0$, there exists some $C_{ T}>0$ (depending only on $T$ and $h$ but independent of $ \varepsilon>0$) and some $ \varepsilon_{ 0}>0$ (depending on $T$ and $h$) such that
\begin{equation*}
\sup_{ t\in [0, T]} u_{ \varepsilon, t} \leq C_{ T} \sup_{ t\in [0, T]} g_{ \varepsilon, t},\ \varepsilon\in (0, \varepsilon_{ 0}).
\end{equation*}
\end{lemma}
\begin{proof}[Proof of Lemma~\ref{lem:Gronwall_eps}]
The proof is an adaptation of \cite[Lemma~23]{MR3449317}. Let $ \varepsilon>0$ and $T>0$. Write for $t\in [0, T]$,
\begin{align}
\label{aux:ueps}
u_{ \varepsilon, t} \leq g_{ \varepsilon, t} + \int_{ 0}^{t} \left\vert h_{ \varepsilon}(t-s) - h(t-s) \right\vert u_{ \varepsilon, s} {\rm d}s + \int_{ 0}^{t} \left\vert h(t-s) \right\vert u_{ \varepsilon, s} {\rm d}s.
\end{align}
The second term of the righthand side of the previous inequality \eqref{aux:ueps} can be bounded as
\begin{align*}
\int_{ 0}^{t} \left\vert h_{ \varepsilon}(t-s) - h(t-s) \right\vert u_{ \varepsilon, s} {\rm d}s \leq \sup_{ s\in [0, T]} u_{ \varepsilon, s} \int_{ 0}^{T} \left\vert h_{ \varepsilon}(u) - h(u) \right\vert {\rm d}u.
\end{align*}
By assumption $ h_{ \varepsilon} \xrightarrow[ \varepsilon\to 0]{L^{ 1}_{ loc}} h$ so that there is some $ \varepsilon_{ 0}= \varepsilon_{ 0}(h,T)>0$ such that for $ \varepsilon\in (0, \varepsilon_{ 0})$, $ \int_{ 0}^{T} \left\vert h_{ \varepsilon}(u) - h(u) \right\vert {\rm d}u\leq \frac{ 1}{ 4}$. Concerning the third term in \eqref{aux:ueps}, proceed exactly as in \cite[Lemma~23]{MR3449317}: let $A>0$ (independent of $ \varepsilon$) such that $ \int_{ 0}^{T} \left\vert h(u) \right\vert \mathbf{ 1}_{ \left\vert h(u) \right\vert> A}{\rm d}u\leq \frac{ 1}{ 4}$. Inserting this and the previous estimate into \eqref{aux:ueps}, we get
\begin{align*}
u_{ \varepsilon, t} &\leq g_{ \varepsilon, t} + \frac{ 1}{ 4} \sup_{ s\in [0, T]} u_{ \varepsilon, s} + \int_{ 0}^{t} \left\vert h(t-s) \right\vert \mathbf{ 1}_{ \left\vert h(t-s) \right\vert > A} u_{ \varepsilon, s} {\rm d}s + \int_{ 0}^{t}| h(t-s)| \mathbf{ 1}_{ |h(t-s)| \leq A} u_{ \varepsilon, s} {\rm d}s,\\
&\leq g_{ \varepsilon, t} + \frac{ 1}{ 4} \sup_{ s\in [0, T]} u_{ \varepsilon, s} +\sup_{ s\in [0, T]} u_{ \varepsilon, s} \int_{ 0}^{t} \left\vert h(t-s) \right\vert \mathbf{ 1}_{ \left\vert h(t-s) \right\vert > A}  {\rm d}s + A\int_{ 0}^{t} u_{ \varepsilon, s} {\rm d}s,\\
& \leq g_{ \varepsilon, t} + \frac{ 1}{ 2} \sup_{ s\in [0, T]} u_{ \varepsilon, s}+ A\int_{ 0}^{t} u_{ \varepsilon, s} {\rm d}s.
\end{align*}
Taking the supremum in $t\in [0, T]$ on both sides, we get $ \sup_{ t\in [0, T]} u_{ \varepsilon, t} \leq 2 \sup_{ t\in [0, T]} g_{ \varepsilon, t} + 2A \int_{ 0}^{t} u_{ \varepsilon, s} {\rm d}s$. We conclude that $ \sup_{ t\in [0, T]} u_{ \varepsilon, t} \leq e^{ 2AT}\sup_{ t\in [0, T]} g_{ \varepsilon, t} $ by the classical Gr\"onwall lemma.
\end{proof}

\subsection{Comparison results}
\begin{proposition}
\label{prop:comparison}
Let $ \Phi_{ 1}$ and $ \Phi_{ 2}$ two nonnegative kernels which are Lipschitz continuous. For all $i=1,2$, let $ \lambda_{ i}$ the solution to \eqref{eq:conv_gen_lambda} for the choice of $ \Phi= \Phi_{ i}$, driven by the same kernel $h$ and with the same source term $ \xi$. Suppose that $h$ is locally integrable (non necessarily nonnegative). Suppose that $ \xi$ is locally bounded. Then the following is true:
\begin{enumerate}[label=(\alph*)]
\item for all $T>0$, there exist some constants $C=C(T, h, \Phi_{ 1})>0$ and $ \kappa= \kappa(T, h, \xi, \Phi_{ 2})$ such that
\begin{equation}
\label{eq:local_Lip_Phis}
\sup_{ u\in [0, T]} \left\vert \lambda_{1, u} - \lambda_{2, u}\right\vert \leq C \sup_{ \left\vert u \right\vert\leq \kappa} \left\vert \Phi_{ 2}(u)- \Phi_{ 1}(u) \right\vert \leq C \left\Vert \Phi_{ 1}- \Phi_{ 2} \right\Vert_{ \infty}.
\end{equation}
Moreover, in the case $ \left\vert \Phi_{ 1} \right\vert_{ Lip} \left\Vert h \right\Vert_{ 1}< 1$, the constant $C$ can be made independent of $T>0$ and we have the uniform bound
\begin{equation}
\label{eq:global_Lip_Phis}
\sup_{ u\geq0} \left\vert \lambda_{1, u} - \lambda_{2, u}\right\vert \leq C \left\Vert \Phi_{ 1}- \Phi_{ 2} \right\Vert_{ \infty}.
\end{equation}

\item Suppose here that $h$ is nonnegative, that $ \xi$ is continuous, that $ \Phi_{ 1}$ is nondecreasing as well as
\begin{equation}
\label{eq:comp_Phi12}
\text{ for all $x\in \mathbb{ R}$, } \Phi_{ 1}(x)\leq \Phi_{ 2}(x).
\end{equation} 
Then, it holds
\begin{equation}
\label{eq:comparison_lambda12}
\lambda_{1, t} \leq \lambda_{2, t},\ t\geq0.
\end{equation}
\end{enumerate}
\end{proposition}
\begin{proof}[Proof of Proposition~\ref{prop:comparison}]
First prove the first item of Proposition~\ref{prop:comparison}. Denote by $c_{ 1}$ the Lipschitz constant of $ \Phi_{ 1}$. 
For all $t\geq0$,
\begin{align}
\lambda_{2, t} - \lambda_{1, t}&= \Phi_{ 2} \left( \xi_{ t} + \int_{ 0}^{t} h(t-s) \lambda_{2, s} {\rm d}s\right)-\Phi_{1} \left( \xi_{ t} + \int_{ 0}^{t} h(t-s) \lambda_{1, s} {\rm d}s\right), \nonumber\\
&= \Phi_{ 2} \left( \xi_{ t} + \int_{ 0}^{t} h(t-s) \lambda_{2, s} {\rm d}s\right)-\Phi_{1} \left( \xi_{ t} + \int_{ 0}^{t} h(t-s) \lambda_{2, s} {\rm d}s\right) \nonumber\\
&+ \Phi_{ 1} \left( \xi_{ t} + \int_{ 0}^{t} h(t-s) \lambda_{2, s} {\rm d}s\right)-\Phi_{1} \left( \xi_{ t} + \int_{ 0}^{t} h(t-s) \lambda_{1, s} {\rm d}s\right). \label{eq:lambdas12}
\end{align}
Recall that $ t \mapsto \lambda_{2, t}$ is locally bounded: hence for all $t\in [0, T]=[0, pt_{ 1}]$,
\begin{equation*}
\left\vert \xi_{ t} + \int_{ 0}^{t} h(t-s) \lambda_{2, s} {\rm d}s\right\vert \leq \left\Vert \xi \right\Vert_{ \infty, [0, T]} + \left\Vert \lambda_{ 2} \right\Vert_{ \infty, [0, T]} \int_{ 0}^{T} \left\vert h(u) \right\vert {\rm d}u:= \kappa_{ T}<\infty,
\end{equation*}
where the constant $ \kappa_{ T}$ depends on $(T, h, \xi, \Phi_{ 2})$.  We obtain directly from the previous equality that
\begin{align}
\label{aux:lambda12Phi}
\left\vert \lambda_{2, t} - \lambda_{1, t}\right\vert \leq \sup_{\left\vert u \right\vert \leq \kappa_{ T}}\left\vert \Phi_{ 2}(u) - \Phi_{ 1}(u)\right\vert + c_{ 1} \int_{ 0}^{t} \left\vert h(t-s) \right\vert \left\vert \lambda_{2, s} - \lambda_{1, s}\right\vert {\rm d}s.
\end{align}
We are now in position to apply Lemma~\ref{lem:Gronwall_eps} (in a simple case where there is no dependence in any $ \varepsilon$): there exists some $C= C(T, \Phi_{ 1})$ such that 
\begin{equation*}
\sup_{ t\in [0, T]}\left\vert \lambda_{ 2,t} - \lambda_{ 1,t}\right\vert \leq C \sup_{\left\vert u \right\vert \leq \kappa_{ T}}\left\vert \Phi_{ 2}(u) - \Phi_{ 1}(u)\right\vert,
\end{equation*} which is \eqref{eq:local_Lip_Phis}.

In the case $ c_{1} \left\Vert h \right\Vert_{ 1}<1$, one can bound \eqref{aux:lambda12Phi} as 
\begin{align*}
\left\vert \lambda_{2, t} - \lambda_{1, t}\right\vert \leq \sup_{\left\vert u \right\vert \leq \kappa_{ T}}\left\vert \Phi_{ 2}(u) - \Phi_{ 1}(u)\right\vert + \sup_{ s\in [0, T]} \left\vert \lambda_{2, s} - \lambda_{1, s}\right\vert c_{ 1} \left\Vert h \right\Vert_{ 1}
\end{align*} 
which gives immediately \eqref{eq:global_Lip_Phis}.

\medskip

Now turn to the proof of the second item of Proposition~\ref{prop:comparison}. We first prove the result under a stronger condition: suppose for the moment that \eqref{eq:comp_Phi12} is replaced by 
\begin{equation}
\label{eq:Phi12st}
\Phi_{ 1}(x) < \Phi_{ 2}(x), x\in \mathbb{ R}.
\end{equation}
 Define 
\begin{equation*}
t_{ \ast}= \inf \left\lbrace t\geq0,\ \lambda_{2, t} < \lambda_{ 1,t}\right\rbrace.
\end{equation*}
The point is to prove that $t_{ \ast}= \infty$.   We proceed by contradiction and suppose that $t_{ \ast}<\infty$. First note that for $i=1,2$, $ \lambda_{i, 0} =\Phi_{ i} \left(\xi_{ 0}\right)$ so that $ \lambda_{1, 0}< \lambda_{2, 0}$ implying that $t_{*}>0.$ Then by Proposition \ref{prop:WP_solC1}  giving continuity of the solutions, we have $ \lambda_{2, t_{ \ast}}= \lambda_{1, t_{ \ast}}$. Write again \eqref{eq:lambdas12} for $t=t_{ \ast}$: we have
\begin{align*}
0=\lambda_{2, t_{ \ast}} - \lambda_{1, t_{ \ast}}&= \Phi_{ 2} \left( \xi_{ t_{ \ast}} + \int_{ 0}^{t_{ \ast}} h(t_{ \ast}-s) \lambda_{2, s} {\rm d}s\right)-\Phi_{1} \left( \xi_{ t_{ \ast}} + \int_{ 0}^{t_{ \ast}} h(t_{ \ast}-s) \lambda_{2, s} {\rm d}s\right) \nonumber\\
&+ \Phi_{ 1} \left( \xi_{ t_{ \ast}} + \int_{ 0}^{t_{ \ast}} h(t_{ \ast}-s) \lambda_{2, s} {\rm d}s\right)-\Phi_{1} \left( \xi_{ t_{ \ast}} + \int_{ 0}^{t_{ \ast}} h(t_{ \ast}-s) \lambda_{1, s} {\rm d}s\right).
\end{align*}
By assumption \eqref{eq:Phi12st}, the term on the first line of the righthand-side of the previous equality is strictly positive. Concerning the second term, we have that $ \int_{ 0}^{t_{ \ast}} h(t_{ \ast}-s) \lambda_{1, s} {\rm d}s \leq  \int_{ 0}^{t_{ \ast}} h(t_{ \ast}-s) \lambda_{2, s} {\rm d}s$ as, by definition of $t_{ \ast}$, $ \lambda_{2, s}\geq \lambda_{1,s}$ for $s\in [0, t_{ \ast}]$ and $h\geq0$. Therefore, since $ \Phi_{ 1}$ is nondecreasing, $\Phi_{ 1} \left( \xi_{ t_{ \ast}} + \int_{ 0}^{t_{ \ast}} h(t_{ \ast}-s) \lambda_{2, s} {\rm d}s\right)-\Phi_{1} \left( \xi_{ t_{ \ast}} + \int_{ 0}^{t_{ \ast}} h(t_{ \ast}-s) \lambda_{1, s} {\rm d}s\right)\geq 0$. Therefore the sum of the two previous terms cannot be $0$. So $ t_{ \ast}=+\infty$ and \eqref{eq:comparison_lambda12} is true. Now prove \eqref{eq:comparison_lambda12} under the weaker condition \eqref{eq:comp_Phi12}. For $ \Phi_{ 1}$ and $ \Phi_{ 2}$ verifying \eqref{eq:comp_Phi12}, define for any $ \varepsilon>0$, $\Phi_{ 2}^{ \varepsilon}(x)= \Phi_{ 2}(x)+ \varepsilon$ and denote by $ \lambda_{ 2}^{ \varepsilon}$ the solution driven by $ \Phi_{ 2}^{ \varepsilon}$. Then for all $ x\in \mathbb{ R}$, $ \Phi_{ 1}(x) \leq \Phi_{ 2}(x) < \Phi_{ 2}^{ \varepsilon}(x)$ so that one can apply the previous intermediate result, that is for all $t\geq0$, all $ \varepsilon>0$
\begin{equation}
\label{aux:comp_lambda12eps}
\lambda_{ t}^{ 1}\leq \lambda_{ 2, t}^{ \varepsilon}.
\end{equation}
But now applying the first item of Proposition~\ref{prop:comparison}, we have for fixed $t\geq0$, the existence of some $C_{ t}>0$ such that $ \left\vert \lambda_{ 2,t}^{ \varepsilon} - \lambda_{ 2,t}\right\vert \leq C_{ t} \varepsilon$, so that passing to the limit as $ \varepsilon\to 0$ ($t$ fixed) into \eqref{aux:comp_lambda12eps} gives the desired result \eqref{eq:comparison_lambda12}. This concludes the proof of Proposition~\ref{prop:comparison}.
\end{proof}

\subsection*{Acknowledgments.} This paper is dedicated to the memory of our friend and colleague Chris\-to\-phe Pouzat, who enthusiastically shared with us his profound knowledge in mathematical modeling in neuroscience and highlighted the relevance of Hawkes processes in this context.  We would like to thank warmly the reviewers for their careful reading of the paper and their useful remarks. Both authors acknowledge the support of Project HAPPY ANR-23-CE40-0007 of the French National Research Agency.


\begin{thebibliography}{10}

\bibitem{AgatheNerine2022}
Z.~Agathe-Nerine.
\newblock Multivariate {H}awkes processes on inhomogeneous random graphs.
\newblock {\em Stochastic Processes and their Applications}, 152:86--148,  
  2022.

\bibitem{agathenerine22}
Z.~Agathe-Nerine.
\newblock Long-term stability of interacting Hawkes processes on random graphs.
\newblock {\em Electronic Journal of Probability}, 28,   2023.

\bibitem{AgatheNerine2025}
Z.~Agathe-Nerine.
\newblock Stability of wandering bumps for Hawkes processes interacting on the
  circle.
\newblock {\em Stochastic Processes and their Applications}, 182:104577, 2025.

\bibitem{Amari:1977}
S.-I. Amari.
\newblock Dynamics of pattern formation in lateral-inhibition type neural
  fields.
\newblock {\em Biol. Cybern.}, 27(2):77--87,  1977.


\bibitem{Athreya1972}
K.~B. Athreya and P.~E. Ney.
\newblock {\em Branching Processes}.
\newblock Springer Berlin Heidelberg, 1972.

\bibitem{Athreya1976}
K.~B. Athreya and K.~Rama~Murthy.
\newblock Feller's renewal theorem for systems of renewal equations.
\newblock {\em {J}ournal of the {I}ndian {I}nstitute of Science}, 58(10), 1976.

\bibitem{MR3054533}
E.~Bacry, S.~Delattre, M.~Hoffmann, and J.~F. Muzy.
\newblock Some limit theorems for {H}awkes processes and application to
  financial statistics.
\newblock {\em Stochastic Processes and their Applications}, 123(7):2475--2499, 2013.

\bibitem{Brauer1975}
F.~Brauer.
\newblock On a nonlinear integral equation for population growth problems.
\newblock {\em SIAM Journal on Mathematical Analysis}, 6(2):312--317,  
  1975.

\bibitem{BRAUER197632}
F.~Brauer.
\newblock Perturbations of the nonlinear renewal equation.
\newblock {\em Advances in Mathematics}, 22(1):32--51, 1976.

\bibitem{bonnet2025}
A.~Bonnet and S.~Robin.
\newblock A Markov switching discrete-time Hawkes process: application to the
  monitoring of bats behavior. \emph{ArXiv e-print 2507.20153,}
\newblock 2025.

\bibitem{MR1411506}
P.~Br{\'e}maud and L.~Massouli{\'e}.
\newblock Stability of nonlinear {H}awkes processes.
\newblock {\em Annals of Probability}, 24(3):1563--1588, 1996.

\bibitem{cattiaux2021limit}
P.~Cattiaux, L.~Colombani, and M.~Costa.
\newblock Limit theorems for {H}awkes processes including inhibition.
\newblock {\em Stochastic Processes and their Applications}, 149:404--426, 
  2022.

\bibitem{CHEVALLIER20191}
J.~Chevallier, A.~Duarte, E.~L{\"o}cherbach, and G.~Ost.
\newblock {Mean field limits for nonlinear spatially extended Hawkes processes
  with exponential memory kernels}.
\newblock {\em Stochastic Processes and their Applications}, 129(1):1 -- 27,
  2019.
  
\bibitem{Chevallier2021}
J.~Chevallier, A.~Melnykova, and I.~Tubikanec.
\newblock Diffusion approximation of multi-class hawkes processes: Theoretical
  and numerical analysis.
\newblock {\em Advances in Applied Probability}, 53(3):716--756,  2021.





\bibitem{Chover1968}
J.~Chover and P.~Ney.
\newblock The non-linear renewal equation.
\newblock {\em Journal d'Analyse Math{\'e}matique}, 21(1):381--413,  1968.

\bibitem{Chover1973}
J.~Chover, P.~Ney, and S.~Wainger.
\newblock Functions of probability measures.
\newblock {\em Journal d'Analyse Math{\'e}matique}, 26(1):255--302,  1973.

\bibitem{Constantin2014}
O.~Constantin and J.~S. Hargraves.
\newblock Monotone solutions to a nonlinear integral equation of convolution
  type.
\newblock {\em Nonlinear Analysis: Real World Applications}, 15:38--41,  
  2014.


\bibitem{costa_graham_marsalle_tran_2020}
M.~Costa, C.~Graham, L.~Marsalle, and V.~C. Tran.
\newblock {Renewal in Hawkes processes with self-excitation and inhibition}.
\newblock {\em Advances in Applied Probability}, 52(3):879--915, 2020.

\bibitem{Coutin25}
L.~Coutin, B.~Massat, and A.~R{\'e}veillac.
\newblock Quantification of limit theorems for Hawkes processes.
\newblock {\em ArXiv e-print 2503.21273}, 2025.

\bibitem{MR3449317}
S.~Delattre, N.~Fournier, and M.~Hoffmann.
\newblock Hawkes processes on large networks.
\newblock {\em The Annals of Applied Probability}, 26(1):216--261, 2016.

\bibitem{Dermitzakis2022}
V.~Dermitzakis and K.~Politis.
\newblock Monotonicity properties for solutions of renewal equations.
\newblock {\em Statistics {\&} Probability Letters}, 180:109226,  2022.

\bibitem{DITLEVSEN20171840}
S.~Ditlevsen and E.~L\"{o}cherbach.
\newblock Multi-class oscillating systems of interacting neurons.
\newblock {\em Stochastic Processes and their Applications}, 127(6):1840 --
  1869, 2017.
  
  \bibitem{Duarte:2016aa}
A.~Duarte, E.~L\"{o}cherbach, and G.~Ost.
\newblock Stability, convergence to equilibrium and simulation of non-linear
  hawkes processes with memory kernels given by the sum of erlang kernels.
\newblock {\em {ESAIM}: Probability and Statistics}, 23:770--796, 2019.



\bibitem{duval2021interacting}
C.~Duval, E.~Lu{\c{c}}on, and C.~Pouzat.
\newblock Interacting {H}awkes processes with multiplicative inhibition.
\newblock {\em Stochastic Processes and their Applications}, 148:180--226, 
  2022.

\bibitem{feller1941}
W.~Feller.
\newblock On the integral equation of renewal theory.
\newblock {\em The Annals of Mathematical Statistics}, 12(3):243--267, 09 1941.

\bibitem{Gripenberg1979}
G.~Gripengerg.
\newblock On the boundedness of solutions of Volterra equations.
\newblock {\em Indiana University Mathematics Journal}, 28(2):279--290, 1979.

  \bibitem{MR0378093}
A.~G. Hawkes and D.~Oakes.
\newblock A cluster process representation of a self-exciting process.
\newblock {\em Journal of Applied Probability}, 11:493--503, 1974.

\bibitem{Heesen2021}
S.~Heesen and W.~Stannat.
\newblock Fluctuation limits for mean-field interacting nonlinear Hawkes
  processes.
\newblock {\em Stochastic Processes and their Applications}, 139:280--297,
   2021.

\bibitem{Hillairet2022}
C.~Hillairet, L.~Huang, M.~Khabou, and A.~R{\'e}veillac.
\newblock The Malliavin-Stein method for Hawkes functionals.
\newblock {\em Latin American Journal of Probability and Mathematical
  Statistics}, 19(2):1293, 2022.

\bibitem{Horst2024}
U.~Horst and W.~Xu.
\newblock Functional limit theorems for Hawkes processes.
\newblock {\em Probability Theory and Related Fields},   2024.

\bibitem{JS}
J.~Jacod and A.~N. Shiryaev.
\newblock {\em Limit Theorems for Stochastic Processes}, volume 288 of {\em
  Grundlehren der Mathematischen Wissenschaften [Fundamental Principles of
  Mathematical Sciences]}.
\newblock Springer Berlin Heidelberg, 1987.

\bibitem{MR3313750}
T.~Jaisson and M.~Rosenbaum.
\newblock Limit theorems for nearly unstable {H}awkes processes.
\newblock {\em The Annals of Applied Probability}, 25(2):600--631, 2015.

\bibitem{Levin1965}
J.~J. Levin.
\newblock The qualitative behavior of a nonlinear Volterra equation.
\newblock {\em Proceedings of the American Mathematical Society},
  16(4):711--718,  1965.

\bibitem{Levin1972}
J.~Levin.
\newblock On a nonlinear Volterra equation.
\newblock {\em Journal of Mathematical Analysis and Applications},
  39(2):458--476,  1972.
  
\bibitem{LONDEN1973106}
S.-O. Londen.
\newblock On a nonlinear Volterra integral equation.
\newblock {\em Journal of Differential Equations}, 14(1):106--120, 1973.

\bibitem{doi:10.1137/0505082}
S.-O. Londen.
\newblock On the asymptotic behavior of the bounded solutions of a nonlinear Volterra equation.
\newblock {\em SIAM Journal on Mathematical Analysis}, 5(6):849--875, 1974.

\bibitem{lucon:hal-05185413}
E.~Lu{\c c}on and C.~Poquet.
\newblock {Neural Field Equations and Hawkes processes: long-term stability of
  traveling wave profiles in the neutral case}.
\newblock {\em ArXiv e-print 2507.19236},  2025.

\bibitem{Ney1977}
P.~Ney.
\newblock The asymptotic behavior of a Volterra-renewal equation.
\newblock {\em Transactions of the American Mathematical Society},
  228(0):147--155, 1977.

\bibitem{ReynaudBouret2013InferenceOF}
P.~Reynaud-Bouret, V.~Rivoirard, and C.~Tuleau-Malot.
\newblock Inference of functional connectivity in neurosciences via Hawkes
  processes.
\newblock {\em 2013 IEEE Global Conference on Signal and Information
  Processing},  317--320, 2013.

  
\bibitem{MR2722456}
P.~Reynaud-Bouret and S.~Schbath.
\newblock Adaptive estimation for {H}awkes processes; application to genome
  analysis.
\newblock {\em Annals of Statistics}, 38(5):2781--2822, 2010.

\bibitem{Robert2024}
P.~Robert and G.~Vignoud.
\newblock A Palm space approach to non-linear Hawkes processes.
\newblock {\em Electronic Journal of Probability}, 29, 2024.


\bibitem{Shilepsky1974}
C.~C. Shilepsky.
\newblock The asymptotic behavior of an integral equation with an application
  to Volterra's population equation.
\newblock {\em Journal of Mathematical Analysis and Applications},
  48(3):764--779,  1974.
  
  \bibitem{Sulem2024}
D.~Sulem, V.~Rivoirard, and J.~Rousseau.
\newblock Bayesian estimation of nonlinear Hawkes processes.
\newblock {\em Bernoulli}, 30(2), 2024.

\bibitem{Torrisi:2016aa}
G.~L. Torrisi.
\newblock Gaussian approximation of nonlinear Hawkes processes.
\newblock {\em The Annals of Applied Probability}, 26(4), 2016.

\bibitem{Wilson1973}
H.~R. Wilson and J.~D. Cowan.
\newblock A mathematical theory of the functional dynamics of cortical and
  thalamic nervous tissue.
\newblock {\em Kybernetik}, 13(2):55--80, Sept. 1973.

\bibitem{MR3102513}
L.~Zhu.
\newblock Central limit theorem for nonlinear {H}awkes processes.
\newblock {\em Journal of Applied Probability}, 50(3):760--771, 2013.










\end{thebibliography}
\end{document}